\documentclass[12pt]{amsart}
\usepackage{txfonts}
\usepackage{amscd,amssymb,amsmath,amsxtra}
\usepackage[foot]{amsaddr}
\usepackage[pdftex,colorlinks,plainpages,backref,urlcolor=blue]{hyperref}
\usepackage{url}
\usepackage{breakurl}
\usepackage{array}
\usepackage{enumitem}
\usepackage{tikz}
\usetikzlibrary{cd}
\usepackage{longtable}
\usepackage{tabu}
\usepackage{shuffle}
\usepackage{verbatim,footmisc}

\newtheorem{theorem}{Theorem}[section]
\newtheorem{lemma}[theorem]{Lemma}
\newtheorem{corollary}[theorem]{Corollary}
\newtheorem{proposition}[theorem]{Proposition}
\newtheorem*{claim*}{Claim}

\theoremstyle{definition}
\newtheorem{definition}[theorem]{Definition}
\newtheorem{example}[theorem]{Example}
\newtheorem{remark}[theorem]{Remark}

\newtheorem*{ack}{Acknowledgment}

\DeclareMathOperator{\id}{id}

\DeclareMathOperator{\spn}{span}
\DeclareMathOperator{\T}{\mathsf{T}}

\DeclareMathOperator{\Hom}{Hom}

\DeclareMathOperator{\im}{im}
\DeclareMathOperator{\coker}{coker}

\DeclareMathOperator{\Int}{Int}

\DeclareMathOperator{\Frac}{Frac}

\DeclareMathOperator{\colim}{colim}
\DeclareMathOperator{\Tors}{Tors}
\DeclareMathOperator{\cont}{cts}

\DeclareMathOperator{\supp}{supp}

\newcommand{\mc}[1]{\mathcal{#1}}

\DeclareMathSymbol{\Gamma}{\mathalpha}{operators}{0}

\newcommand{\N}{\mathbb{N}}
\newcommand{\Z}{\mathbb{Z}}
\newcommand{\Q}{\mathbb{Q}}
\newcommand{\R}{\mathbb{R}}
\newcommand{\C}{\mathbb{C}}

\newcommand{\F}{\mathbb{F}}

\newcommand{\TT}{\mathcal{T}}
\newcommand{\fm}{\mathfrak{m}}

\newcommand{\z}{\zeta}

\newcommand{\bwedge}{\mbox{\normalsize $\bigwedge$}}

\newcommand{\mcm}{\mc{M}}

\newcommand{\la}{\langle}
\newcommand{\ra}{\rangle}

\newcommand{\PL}{\scriptscriptstyle{\rm PL}}

\newcommand{\bX}{\boldsymbol{X}}
\newcommand{\bY}{\boldsymbol{Y}}
\newcommand{\ba}{\mathbf{a}}
\newcommand{\bb}{\mathbf{b}}

\newcommand{\bx}{\mathbf{x}}
\newcommand{\by}{\mathbf{y}}
\newcommand{\bz}{\mathbf{0}}

\newcommand{\ot}{\otimes}

\DeclareMathAlphabet{\pazocal}{OMS}{zplm}{m}{n}

\newcommand{\surj}{\twoheadrightarrow}
\newcommand{\inj}{\hookrightarrow}

\newcommand{\isom}{\xrightarrow{
   \,\smash{\raisebox{-0.6ex}{\ensuremath{\scriptstyle\simeq}}}\,}}

\newcommand{\cupd}{\cup_1\text{--\,}d}
\newcommand{\astt}{\!\ast}

\makeatletter
\newsavebox{\@brx}
\newcommand{\llangle}[1][]{\savebox{\@brx}{\(\m@th{#1\langle}\)}%
  \mathopen{\copy\@brx\kern-0.5\wd\@brx\usebox{\@brx}}}
\newcommand{\rrangle}[1][]{\savebox{\@brx}{\(\m@th{#1\rangle}\)}%
  \mathclose{\copy\@brx\kern-0.5\wd\@brx\usebox{\@brx}}}
\makeatother

\newcommand{\arxiv}[1]
{\texttt{\href{http://arxiv.org/abs/#1}{arXiv:#1}}}

\newcommand{\abs}[1]{\left| #1 \right|}

\numberwithin{table}{section}
\numberwithin{equation}{section}
\def\norms#1#2{\left\| #1 \right\|_{\lower 1ex\hbox{$\scriptstyle #2 $}}}

\makeatletter
\def\namedlabel#1#2{\begingroup
    #2%
    \def\@currentlabel{#2}%
    \phantomsection\label{#1}.\endgroup
}
\makeatother
\makeatletter
\def\@tocline#1#2#3#4#5#6#7{\relax
  \ifnum #1>\c@tocdepth 
  \else
    \par \addpenalty\@secpenalty\addvspace{#2}%
    \begingroup \hyphenpenalty\@M
    \@ifempty{#4}{%
      \@tempdima\csname r@tocindent\number#1\endcsname\relax
    }{%
      \@tempdima#4\relax
    }%
    \parindent\z@ \leftskip#3\relax \advance\leftskip\@tempdima\relax
    \rightskip\@pnumwidth plus4em \parfillskip-\@pnumwidth
    #5\leavevmode\hskip-\@tempdima
      \ifcase #1
       \or\or \hskip 1em \or \hskip 2em \else \hskip 3em \fi%
      #6\nobreak\relax
    \dotfill\hbox to\@pnumwidth{\@tocpagenum{#7}}\par
    \nobreak
    \endgroup
  \fi}
\makeatother

\tikzset{
  knot diagram/every strand/.append style={
    ultra thick,
    red
  },
  show curve controls/.style={
    postaction=decorate,
    decoration={show path construction,
      curveto code={
        \draw [blue, dashed]
        (\tikzinputsegmentfirst) -- (\tikzinputsegmentsupporta)
        node [at end, draw, solid, red, inner sep=2pt]{};
        \draw [blue, dashed]
        (\tikzinputsegmentsupportb) -- (\tikzinputsegmentlast)
        node [at start, draw, solid, red, inner sep=2pt]{}
        node [at end, fill, blue, ellipse, inner sep=2pt]{}
        ;
      }
    }
  },
  show curve endpoints/.style={
    postaction=decorate,
    decoration={show path construction,
      curveto code={
        \node [fill, blue, ellipse, inner sep=2pt] at (\tikzinputsegmentlast) {}
        ;
      }
    }
  }
}

\definecolor{dkgreen}{RGB}{0,100,0}
\definecolor{dkbrown}{RGB}{139,69,19}

\begin{document}

\date{January 26, 2025}
    
\title[Cup-one algebras and $1$-minimal models]%
{Cup-one algebras and $1$-minimal models}

\author{Richard~D.~Porter$^1$}
\author{Alexander~I.~Suciu$^{1,2}$}
\address{$^1$Department of Mathematics,
Northeastern University,
Boston, MA 02115, USA}
\email{\href{mailto:r.porter@northeastern.edu}{r.porter@northeastern.edu}}

\email{\href{mailto:a.suciu@northeastern.edu}{a.suciu@northeastern.edu}}
\thanks{$^2$Supported in part by the Simons Foundation Collaboration Grant 
for Mathematicians \#693825}

\begin{abstract}
In previous work we introduced the notion of binomial cup-one algebras, 
which are differential graded algebras endowed with Steenrod $\cup_1$-products 
and compatible binomial operations. In this paper we show that binomial cup-one
algebras capture homotopy $1$-type. In particular, given such an $R$-dga, $(A,d_A)$, 
defined over the ring $R=\Z$ or $\F_p$ (for $p$ a prime), with $H^0(A)=R$ and with 
$H^1(A)$ a finitely generated, free $R$-module, we show that $A$ admits a functorially 
defined $1$-minimal model, $\rho\colon (\mcm(A),d)\to (A,d_A)$, which is unique up to 
isomorphism. Furthermore, we associate to this model a pronilpotent group, whose 
continuous cohomology is isomorphic to that of $\mcm(A)$. These constructions, 
which refine classical notions from rational homotopy theory, allow us to distinguish 
spaces with isomorphic torsion-free integral cohomology rings. Moreover, we show 
that there is an equivalence of categories between isomorphism classes of 
finitely-generated, torsion-free-nilpotent groups and isomorphism classes 
of finitely generated $1$-minimal models over the integers.
\end{abstract}



\maketitle

\setcounter{tocdepth}{1}
\tableofcontents
\section{Introduction}
\label{sect:intro}

\subsection{Overview}
\label{intro:overview}

In previous work we combined properties of the Steenrod cup-one 
products of cochains and binomial rings in the cochain complex of 
a space to define the algebraic categories of binomial cup-one  
differential graded algebras over the integers and over $\F_p$ 
for $p$ a prime. We construct here the $1$-minimal model $\mcm$
for such a binomial, cup-one dga $(A,d)$, and prove its key properties; 
namely, that it is a free binomial cup-one dga unique 
up to isomorphism, and that it determines a pronilpotent group, 
also unique up to isomorphism, whose continuous cohomology 
is isomorphic to that of $\mcm$.

Since the cochains of a space $X$ with coefficients 
in the ring $R=\Z$ or $\F_p$ are binomial $\cup_1$-dgas, 
it follows that invariants of the $1$-minimal model for $A = C^\ast(X;R)$  
give homotopy type invariants of $X$.
As an application, we define in Section \ref{sect:dcs-1min} 
one such invariant, called {\em $n$-step equivalence}, and 
exhibit a family of spaces that can be distinguished 
using the integral $1$-minimal model, where the corresponding 
approach in rational homotopy theory fails to distinguish those spaces. 
Moreover, we show in Section \ref{sect:nilpotent} that there is an 
equivalence of categories
between isomorphism classes of finitely-generated, 
torsion-free-nilpotent groups and isomorphism classes
of finitely generated $1$-minimal models over the integers.

The reader familiar with rational homotopy theory 
(\cite{FHT, FHT2, Griffiths-Morgan}), will note that the 
$1$-minimal models defined here over $\Z$ and over $\F_p$ can 
be viewed as analogues of Sullivan's $1$-minimal model over the rationals. 
Integral homotopy theory, as initiated in work of Ekedahl \cite{Ekedahl-2002} 
and Kadeishvili \cite{Kadeishvili-2003}, deals with $E_\infty$-algebras 
rather than commutative differential graded algebras that are used in rational homotopy theory.
In  \cite{Mandell}, Mandell showed that two nilpotent, finite type spaces are weakly equivalent 
if and only if their singular cochain complexes are quasi-isomorphic as $E_\infty$-algebras.
From this perspective, the $1$-minimal model constructed in this paper 
expresses precisely the $E_\infty$-data that matter for the fundamental group. 

In other recent work of interest, Petersen \cite{Petersen}
develops a framework for studying the compactly supported 
cohomology of configuration spaces using a forgetful functor from 
$E_\infty$-algebras to twisted cdgas;
Richter and Sagave \cite{Richter-Sagave} use diagrams of 
chain complexes to model $E_\infty$-dgas by strictly commutative objects;
Horel \cite{Horel} constructs a fully faithful functor from finite type nilpotent 
spaces to cosimplicial binomial rings;
Flynn-Connolly \cite{Flynn} explores the relationship 
between commutative algebras and $E_\infty$-algebras in 
characteristic $p$; and Gadish \cite{Gadish} connects the first 
layer of the  $E_\infty$-algebra structure on the cochain 
complex of a space to ``letter-braiding" invariants of the fundamental group.
 
\subsection{Steenrod $\cup_i$-products and binomial cup-one dgas}
\label{intro:cochain-steenrod}
Let $A=(C^*(X;R),d)$ be the cellular cochain complex of a $\Delta$-set $X$, 
with coefficients in a commutative ring $R$.  Then $A$ is, in fact, a 
differential graded $R$-algebra, with multiplication given by 
the cup-product. In \cite{Steenrod}, Steenrod introduced a 
sequence of operations, $\cup_i\colon A^p\otimes_{R}  A^q\to A^{p+q-i}$, 
starting with $\cup_0=\cup$, the usual cup-product. We are mainly interested here 
in the additional structure on the cochain algebra provided by the 
$\cup_1$-product, which is tied to the differential and the cup 
product via the Steenrod  \cite{Steenrod} and Hirsch \cite{Hirsch} identities 
 (see Section \ref{subsec:cup-i}).

We defined in \cite{Porter-Suciu-2021}  several categories of 
differential graded algebras (over $R=\Z$ or $\F_p$) with  
extra structure, coming from either the cup-one products, 
or from a binomial ring structure, or both, bound together 
by suitable compatibility conditions. 
In Sections \ref{subsec:graded-cup1} and \ref{subsec:cupd-formula} 
we recall these notions,  with some mild modifications 
to better fit the present context. 
A \emph{cup-one differential graded algebra}\/ is   
an $R$-dga $(A,d)$ with a cup-one product map 
$\cup_1\colon A^1 \otimes_R A^1 \to A^1$ 
that gives $R\oplus A^{1}$ the structure 
of a commutative ring and satisfies the 
Hirsch identity, as well as the following ``$\cupd$ formula,"
\begin{equation}
\label{eq:cupd-intro}
d(a \cup_1 b) = - a \cup b - b \cup a
+ da \cup_1 b + db \cup_1 a - da \circ db,
\end{equation}
for all $a,b\in A^1$ with $da,db$ equal to sums of cup products, 
where the map $\circ$ is bilinear and satisfies 
$(a_1 \cup a_2)\circ (b_1 \cup b_2)
= (a_1 \cup_1 b_1) \cup (a_2 \cup_1 b_2)$. 
The significance of formula \eqref{eq:cupd-intro}
is that it expresses the 
differentials of cup-one products of elements in $A^1$ in terms of the 
differentials of the factors and $\cup_1$-products
of elements in $A^1$, as opposed to Steenrod's formula, 
which involves $\cup_1$-products of elements in $A^2$
with elements in $A^1$.  Moreover, as shown in \cite{Porter-Suciu-2021}, 
if $A=C^\ast(X;R)$ is a cochain algebra, Steenrod's formula 
restricts to formula \eqref{eq:cupd-intro} for elements $a, b\in A^1$ with
$da$ and $db$ equal to sums of cup products.

In Section \ref{sect:bin-cup1}, we add extra structure to these algebras. 
A commutative ring $A$  is called a {\em binomial ring}\/  if $A$ is torsion-free 
as a $\Z$-module, and has the property that the elements 
$\binom{a}{n} \coloneqq a(a-1)\cdots (a-n+1)/n!$ 
lie in $A$ for every $a\in A$ and every $n>0$.  
An analogous notion holds for $\F_p$-algebras. 
These objects come equipped with maps 
$\zeta_n\colon A\to A$, $a\mapsto \binom{a}{n}$, 
defined for all $n>0$ over $\Z$, and only for $n<p$ over $\F_p$.
Now consider a cup-one dga $(A,d)$ over $R=\Z$ or $\F_p$. 
Such an object is called a 
{\em binomial cup-one algebra}\/ if $A^0$, with multiplication 
$A^0\otimes_{R} A^0\to A^0$ given by the cup-product, is a binomial 
$R$-algebra, and the $R$-submodule $R \oplus A^1 \subset A^{\le 1}$, 
with multiplication $A^1\otimes_{R} A^1\to A^1$ given by the cup-one 
product, is an $R$-binomial algebra.

Our main motivating example is the cochain algebra of a space. 
In Theorem \ref{thm:cochain-bincup1d} we show that, for any 
$\Delta$-set $X$, the cellular cochain algebra $C=(C^{*} (X;R),d)$
is a binomial cup-one dga, with maps $\cup_1\colon C^1\otimes_{R} C^1\to C^1$ 
given by $(a \cup_1\! b)(e) =  a(e)\cdot b(e)$ for all $1$-simplices $e$  
and with the $\circ$ map equal to Steenrod's $\cup_2$-product, with
$\circ=\cup_2\colon C^2 \otimes_R C^2 \to C^2$ given by 
$(v \circ w)(s) = v(s)\cdot w(s)$,  for all $2$-simplices $s$, 
and with binomial maps given by $\zeta_n(a)(e)=\binom{a(e)}{n}$ 
when $R=\Z$ and analogously for $R=\F_p$.

\subsection{Free binomial cup-one dgas}
\label{intro:free-bin-cup1}

These structures allow us to define in Section \ref{sect:free-bin-dga} the 
{\em free binomial cup-one graded algebra}, $\T=\T_R^*(\bX)$, 
on a set $\bX$. When $R=\Z$, the starting point is the ring 
$\Int(\Z^{\bX})=\{q\in \Q[\bX] \mid q(\Z^{\bX})\subseteq \Z\}$ 
of integrally-valued polynomials with variables in $\bX$.
This is a binomial ring, generated by the polynomials 
$\binom{\bX}{\mathbf{n}}=\prod_{x\in \bX}\binom{x}{\mathbf{n}_x}$ 
with $\mathbf{n}_x \in \Z_{\ge 0}$. 
We define $\T$ to be the tensor algebra on $\fm_{\bX}$, 
the maximal ideal at $0$ in $\Int(\Z^{\bX})$. 
When $R=\F_p$, an analogous definition applies, 
with suitable modifications. In either case, we have 
$R$-linear maps $\cup_1\colon \T^1\otimes \T^1\to \T^1$, 
given by $a\cup_1 b=ab$ and $\circ \colon \T^2 \otimes \T^2 \to \T^2$ 
given by $(a_1 \otimes a_2) \circ (b_1 \otimes b_2)=
(a_1 b_1) \otimes (a_2 b_2)$.  Using the binomial structure 
on $\T^1$, we show that the map sending each $x\in \bX$ to $0$ 
extends to a linear map $d_{\bz}\colon \T^1\to \T^2$.  
In turn, $d_{\bz}$ extends to a differential on the whole 
tensor algebra by the graded Leibniz rule, and 
$(\T_R(\bX),d_{\bz})$ is then a binomial cup-one dga.
The key result that allows us to construct the differential 
$d_{\bz}$ is Theorem \ref{thm:dsquared}, which reads as follows.

\begin{theorem}
\label{thm:intro-dsquared}
Let $d \colon \T_R(\bX) \to \T_R(\bX)$ be a degree-one map satisfying the
$\cupd$ formula and the Leibniz rule.
Then $d^2(x) = 0$ for all $x \in \bX$ if and only if 
$d^2(u)=0$ for all $u \in \T_R(\bX)$, in which case $(\T_R(\bX), d)$
is a binomial cup-one dga.
\end{theorem}

Taking $d(x)=0$ for all $x\in \bX$ yields the differential 
$d_{\bz}$ from above. As shown in \cite{Porter-Suciu-2021}, 
this differential is compatible with the binomial structure on $\T_R(\bX)$; 
more precisely, 
\begin{equation}
\label{eq:dzeta-intro}
d_{\bz} (\z_n (x)) = - \sum_{\ell = 1}^{n-1}\z_\ell(x) \otimes \z_{n-\ell}(x) , 
\end{equation}
for all $x \in \bX$, and for all $n \ge 1$ when $R=\Z$ and 
for $1 \le n \le p-1$ when $R=\F_p$. As an application of the methods 
developed here, we give in Theorem \ref{thm:cup-zeta} 
a quicker, more conceptual proof of this result. 

\subsection{Differentials defined by admissible maps}
\label{intro:admissible}
A key thread in our paper involves the correspondence (described 
in Section \ref{sect:cochain}) between a magma, that is, a set $M$ 
with a binary operation $\mu\colon M \times M \to M$, and a certain 
$2$-dimensional $\Delta$-set, $\Delta^{(2)}(M)$. In the
case when $\mu$ is associative, that is, $(M,\mu)$ is a 
semigroup, this $2$-complex extends to an
infinite-dimensional cell complex $\Delta(M)$, 
whose $n$-simplices are given
by ordered $n$-tuples of elements in $M$. 
Moreover, if $(M,\mu)$ is a group, then $\Delta(M)$ is the cell complex 
of the bar construction applied to $M$; that is, an Eilenberg--MacLane 
classifying space for $M$.

Properties of the cellular cochain algebras $(C^*(\Delta(M);R), d_{\Delta})$ 
are used in Section \ref{sect:admissible} to derive properties of 
differential graded algebras in our category of binomial $\cup_1$-dgas, 
as follows. Given a set $\bX$ and a set map $\tau\colon \bX\to \T_R^2(\bX)$, 
we start by defining a binary operation, $\mu_{\tau}\colon M\times M\to M$, 
on the $R$-module $M=M(\bX,R)$ of all functions from $\bX$ to 
the ring $R=\Z$ or $\F_p$.  Letting $\Delta^{(2)}(M_\tau)$ be the 
$2$-dimensional $\Delta$-set associated to the magma $M_{\tau}=(M, \mu_\tau)$, 
we define a degree-preserving, $R$-linear map, 
$\psi=\psi_{\bX,\mu}\colon \T_R^{\le 2}(\bX) \to C^*(\Delta^{(2)}(M_\tau);R)$. 
This map sends $1\in \T_R^0(\bX)=R$ to the unit $0$-cochain;   
a polynomial $q\in \T_R^1(\bX)=\fm_{\bX}$ to the $1$-cochain 
whose value on a $1$-simplex $\ba\colon \bX\to R$ is $q(\ba)$; 
and a tensor $q\otimes q'  \in \T_R^2(\bX)$ to the $2$-cochain 
whose value on a $2$-simplex $(\ba, \ba^\prime)$ is $q(\ba)\cdot q'(\ba^\prime)$. 
We then show in Lemma \ref{lem:one} that the map $\psi$ is 
monomorphism which commutes with cup products, cup-one products, 
and the $\circ$ maps. Using the embedding $\psi$, 
we show in Theorem \ref{thm:extend-diff} 
that there is a unique extension of the map 
$\tau\colon \bX\to \T_R^2(\bX)$ to an $R$-linear 
map  $d_{\tau} \colon \T_R(\bX) \to \T_R(\bX)$ that 
satisfies the Leibniz rule and the $\cupd$ formula.

To make further headway, we focus on the case when 
$M_{\tau}$ is a semigroup (we then say 
$\tau$ is {\em admissible}), and consider the 
associated cell complex, $\Delta(M_\tau)$. 
We then show in Theorem \ref{thm:assoc-2} that the map 
$\psi$ extends uniquely to an inclusion 
$\psi\colon (\T_R(\bX), d_{\tau}) \inj C^\ast(\Delta(M_\tau);R)$ 
that satisfies $\psi\circ d_\tau = d_\Delta\circ\psi$, from which it 
follows that $d_\tau^2$ is the zero map. These results
can be summarized as follows.
\begin{theorem}
\label{intro:tau-psi}
If the map $\tau\colon \bX \to \T^2_R(\bX)$ is admissible, then
$d^2_\tau \equiv 0$ and the map $\T^1(\bX) \to C^1(\Delta(M_\tau);R)$ 
given by $q\mapsto (\ba\mapsto q(\ba))$ 
extends uniquely to a monomorphism 
$\psi_{\bX} \colon (\T_R(\bX),d_{\tau}) \inj  (C^\ast(\Delta(M_\tau);R),d_{\Delta})$ 
of binomial cup-one dgas.
\end{theorem}

\subsection{Hirsch extensions}
\label{intro:Hirsch}

In Section \ref{sect:HirschExt} we continue laying out the groundwork 
for the construction of $1$-minimal models for $\cup_1$-dgas. 
The first step in the construction 
is a free binomial $\cup_1$-dga of the form $(\T_R(\bX), d_{\bz})$ 
with $\bX$ a finite set of $n$ elements. In this case, the map 
$\tau\colon \bX\to \T^2_R(\bX)$ is the zero map, and the 
corresponding $R$-module, $M=M_{\tau}$, is isomorphic to $R^n$.  
Using a spectral sequence argument, we prove in 
Theorem \ref{thm:coho-iso-R} that the map 
$\psi_{\bX}\colon (\T_R(\bX), d_\bz) \to C^\ast(B(R^n);R)$ 
induces an isomorphism on cohomology.

As in rational homotopy theory, Hirsch extensions of free binomial 
$\cup_1$-dgas are the basic building blocks for constructing $1$-minimal 
models. An inclusion $i \colon (\T_R(\bX),d) \to (\T_R(\bX\cup \bY), \bar{d})$
is called a {\em Hirsch extension}\/ if $\bar{d}(y)$ is a cocycle
in $\T_R^2(\bX)$ for all $y \in \bY$. 
As shown in Theorem \ref{thm:ext-bijection}, there is a bijection
between maps of sets from $Y$ to cocycles in $\T_R^2(\bX)$ 
and Hirsch extensions of this sort.

Assume now that $\bX = \bigcup_{i \ge 1} \bX_i$, with each $\bX_i$ a finite
set and $\bX_1 \ne \emptyset$. An $R$-dga $\T=(\T_R(\bX), d)$ is called 
a {\em colimit of Hirsch extensions}\/ if the differential $d$ 
restricts to differentials $d_n$ on $\T_R(\bX^n)$ such that 
$d_1(x) =0$ for all $x \in \bX_1$ and 
each dga $(\T_R(\bX^{n+1}),d_{n+1})$ is a Hirsch extension
of $(\T_R(\bX^{n}),d_{n})$. 
To such a colimit of Hirsch extensions, $\T$, we associate 
in Lemma \ref{lem:h-group} a pronilpotent group, $G_{\T}$, 
together with a $\cup_1$-dga map, 
$\psi_{\T} \colon \T \to C^*(B(G_{\T});R)$, inducing an 
isomorphism on $H^1$. As shown in 
Theorem \ref{lem:group-gives-hirsch}, if $\psi_{\T}$ is a quasi-isomorphism 
and $0 \to F \to \overline{G} \xrightarrow{\pi} G \to 1$ is a central extension 
of groups with $F$ a finitely generated, free $R$-module, then 
there is a Hirsch extension $i\colon \T \inj \overline{\T}$
such that $\overline{G}=G_{\overline{T}}$, the map $\psi_{\overline{\T} }$ 
is also a quasi-isomorphism, and  
$B(\pi)^*\circ \psi_{\T}= \psi_{\overline{\T} }\circ i$.

In work in progress \cite{Porter-Suciu-group-1min}, 
we build on this correspondence between colimits of Hirsch 
extensions and sequences of central extensions of groups. 
For the cochain algebra $A= C^\ast(Y;R)$ of a path-connected 
space $Y$, we will describe a concrete relationship between the 
group $G_{\T}$ and the fundamental group $\pi_1(Y)$,
where $\T$ is a colimit of Hirsch extensions together 
with a $1$-quasi-isomorphism
$\rho\colon \T \to C^\ast(Y; R)$.

\subsection{$1$-minimal models}
\label{intro:one-min}
In Sections \ref{sect:1-min} and \ref{subsec:1-min-unique} 
(which form the core of this work), 
we develop these ideas into a theory of $1$-minimal 
models over a ring $R$ equal to $\Z$ or $\F_p$.  
A key technical tool is provided by the following lifting criterion 
(Theorem \ref{thm:1-qi-lift}): Let $f\colon A\to A'$ be a surjective 
$1$-quasi-iso\-morphism between binomial cup-one $R$-dgas 
and let $\varphi\colon \T\to A'$ be a morphism from 
a colimit of Hirsch extensions to $A'$; there is then a morphism 
$\widehat{\varphi} \colon \T \to A$ such that $f\circ \widehat{\varphi}=\varphi$. 

Now let $(A,d)$ be a binomial cup-one $R$-dga. 
A {\em $1$-minimal model}\/ 
for $A$ is a colimit of Hirsch extensions 
$\mcm_n=(\T_R(\bX^n),d_n)$, together with 
morphisms $\rho_n\colon  \mcm_n\to A$ 
compatible with the Hirsch extensions
of $\mcm_n$ into $\mcm_{n+1}$. 
Additionally, the map 
$H^1(\rho_1)\colon H^1(\mcm_1)\to H^1(A)$ is required 
to be an isomorphism; in particular, $\bX_1$ corresponds to
a basis for $H^1(A)$. For $n \ge 1$, the set $\bX_{n+1}$ 
is a basis for the free submodule 
$\ker( H^2(\rho_n)) \subset H^2(\mcm_n)$ given by the cohomology classes of the 
$2$-cocycles $d_{n+1}(x)$ with $x\in  \bX_{n+1}$.

In Theorems \ref{thm:cupd-minmodel} and \ref{thm:1-min-lift} 
we show that every binomial cup-one dga admits (under some mild 
finiteness assumptions) a $1$-minimal model, unique up to isomorphism.
These results may be summarized, as follows.

\begin{theorem}
\label{thm:intro-minmodel}
Let $(A,d_A)$ be a binomial cup-one dga over $R=\Z$ or $\F_p$, with $p$ a 
prime. Assume $H^0(A)=R$ and $H^1(A)$ is a finitely generated, free 
$R$-module. Then,
\begin{enumerate}[itemsep=1pt]
\item There is a $1$-minimal model, $\mcm=(\T_R(\bX),d)$, for $A$, and a 
structural morphism, $\rho\colon \mcm\to A$, that is a $1$-quasi-isomorphism.
\item Given $1$-minimal models, $\rho \colon \mcm \to A$ and  
$\rho' \colon \mcm' \to A$, 
there is an isomorphism $f \colon \mcm \to \mcm'$ 
and a dga homotopy $\Phi \colon \mcm \to A \otimes_R C^{*}([0,1];R)$
from $\rho$ to $\rho' \circ f$.
\end{enumerate}
\end{theorem}

In the case when $(A,d_A)$ admits an augmentation, that is, a dga morphism 
$\varepsilon \colon A\to R$, the isomorphism $f$ from above is unique (in the 
category of augmented dgas). More precisely, we prove in Theorems \ref{thm:aug-minmodel} 
and \ref{thm:1-min-cup-aug}  that $A$ has an augmented $1$-minimal model, $\mcm$, such 
that the structural morphism is an augmentation-preserving $1$-quasi-isomorphism. Moreover, 
given augmented $1$-minimal models, $\rho \colon \mcm \to A$ and  
$\rho' \colon \mcm' \to A$, there is a {\em unique}\/ augmentation-preserving
isomorphism $f\colon \mcm \to \mcm'$ such that $\rho$  is 
augmen\-tation-preserving homotopic to $\rho' \circ f$.

\subsection{Compatibility of integer and rational $1$-minimal models}
\label{intro:compatibility}
In Section \ref{sect:Z-vs-Q} we show that
the integer $1$-minimal model of a space $Y$
tensored with the rationals is weakly equivalent 
as a dga to the $1$-minimal model for $Y$ in rational
homotopy theory. 

The algebra of polynomial forms with rational coefficients
on a standard simplex was used by Sullivan in \cite{Sullivan} 
to define the algebra $A_{\PL}(Y)$ of compatible polynomial 
forms on the singular simplices of a space $Y$; this  
algebra is a commutative dga over the rationals.
The properties of the $1$-minimal model of a cdga over $\Q$  
are analogous to---and in fact are the motivation for---the 
properties we use to define the $1$-minimal model for a 
binomial cup-one dga over $\Z$ or $\F_p$.

In addition to the Sullivan algebra $A_{\PL}(Y)$ and
the singular cochain algebra $C^\ast(Y;\Q)$, there
is a dga over the rationals $CA(Y)$ with the property
proved in \cite{FHT} that for topological spaces $Y$, there 
are natural quasi-isomorphisms
$C^\ast(Y;\Q) \to CA(Y) \leftarrow A_{\PL}(Y)$. 
Consequently, $A_{\PL}(Y)$ is weakly equivalent (as a dga) to $C^\ast(Y;\Q)$. 
The following result, Theorem \ref{thm:1-min-mod-zq},
shows that weak equivalence extends to $1$-minimal models.
\begin{theorem}
\label{thm:1-min-mod-zq-2}
Let $Y$ be a connected topological space with
$H^1(Y;\Z)$ finitely generated.
Then the $1$-minimal model for $C^\ast(Y;\Z)$ tensored 
with the rationals is weakly equivalent as a dga 
to the $1$-minimal model in rational
homotopy theory for $A_{\PL}(Y)$.
\end{theorem} 

As we shall see next, although the $1$-minimal model for 
$C^\ast(Y;\Z)$ is weakly equivalent over $\Q$ to the (rational) 
$1$-minimal model for $A_{\PL}(Y)$, the integral version does 
contain more refined information than its rational version.

\subsection{$n$-step equivalence and triple Massey products}
\label{intro:nstep}
Given a positive integer $n$, we define in Section \ref{sect:dcs-1min} 
the relation of {\em $n$-step equivalence}\/ on the set of augmented 
binomial cup-one dgas $(A,d)$ over $\Z$ for which $H^0(A)=\Z$, $H^1(A)$ is 
finitely generated and torsion-free, and $H^2(A)$ is finitely generated.
We then construct an infinite family of spaces that can be distinguished 
using the $1$-minimal model over $\Z$, though the same approach in 
rational homotopy theory fails to distinguish among the spaces in the family.

The definition of $n$-step equivalence is motivated as follows. 
If a morphism $\varphi \colon A \to A^\prime$ induces isomorphisms 
of cohomology groups in degrees up to $2$, then for each $n \ge 1$ there is 
an isomorphism of the $n$th-step in the respective $1$-minimal models, 
$f_n \colon \mcm_n \to \mcm_n^\prime$, such that the
following diagram commutes 
\begin{equation}
\label{eq:ladder-n-step-intro}
\begin{tikzcd}[column sep=32pt]
H^2(\mcm_n) \ar[r, "H^2(f_n)"] \arrow["H^2(\rho_n)"  ']{d}
& H^2(\mcm_n^\prime)  \phantom{\, .} \arrow["H^2(\rho'_n)"]{d}
\\
H^2(A) \ar[r, "H^2(\varphi)"] 
& H^2(A^\prime) \, .
\end{tikzcd}
\end{equation}
Note that the horizontal arrows in 
\eqref{eq:ladder-n-step-intro} are isomorphisms.
We say that $A$ and $A^\prime$ are 
$n$-step equivalent if there are isomorphisms
$f_n\colon \mcm_n \to \mcm_n^\prime$ and 
$e_n \colon H^2(A) \to H^2(A^\prime)$ such that
the diagram \eqref{eq:ladder-n-step-intro} commutes
with $H^2(\varphi)$ replaced by $e_n$.

If $A$ and $A^\prime$ are $n$-step equivalent, then
the cokernels of the homomorphisms $H^2(\rho_n)$ and 
$H^2(\rho_n^\prime)$ are isomorphic, and hence have isomorphic 
torsion subgroups. Given a space $X$ with $n$th-step in the
$1$-minimal model given by $(\mcm_n, \rho_n)$, we define 
$\kappa_n(X)=\Tors (\coker H^2(\rho_n))$. The following result 
(proved in Theorem \ref{thm:iso-pi1}) relates the invariant 
$\kappa_n(X)$ of the $n$-step equivalence class of $C^\ast(X;\Z)$ 
to the fundamental group of $X$.
\begin{theorem}
\label{thm:iso-pi1-intro}
Let $X$ and $X'$ be two connected $\Delta$-complexes
with first and second integral cohomology groups finitely generated. 
Then,
\begin{enumerate}[itemsep=2pt]
\item \label{kn1-intro}
If $\pi_1(X)\cong \pi_1(X')$, then $\kappa_n(X)\cong \kappa'_n(X)$ 
for all $n\ge 1$.
\item \label{kn2-intro}
If $\kappa_n(X)\not\cong \kappa'_n(X)$ 
for some $n\ge 1$, then the cochain algebras $C^*(X;\Z)$ and 
$C^*(X';\Z)$  are not $n$-step equivalent.
\end{enumerate}
\end{theorem}

We apply this result to an infinite family of links in the 
three-sphere, $\{L(n)\}_{n \ge 1}$, the first term of which is 
the well-known Borromean rings.
Set $X(n)$ equal to the complement of $L(n)$ in $S^3$. 
In Proposition \ref{prop:borro}, we show that 
$\kappa_2(X(n)) = \Z_n \oplus \Z_n$, so 
by part \eqref{kn2-intro} of Theorem \ref{thm:iso-pi1-intro},
$X(n)$ and $X(m)$ are not $2$-step equivalent for 
$n \ne m$. We also show that the Sullivan algebras $A_{\PL}(X(n))$ 
and $A_{\PL}(X(m))$ are $2$-step equivalent for all $n, m\ge 1$.

In the proof of Proposition \ref{prop:borro}, the cokernel of 
$H^2(\rho_2)$ is given by triple Massey products of cohomology 
classes in $H^1(X(n);\Z)$. 
This framework provides the context for defining restricted
Massey products, which is a particular case of a more general 
construction that will be developed in \cite{Porter-Suciu-massey}. 
The theory of generalized Massey products continues the program 
initiated in \cite{Porter-Suciu-2020} and is being developed more 
fully in \cite{Porter-Suciu-group-1min, Porter-Suciu-massey}. 

\subsection{$1$-minimal models for nilmanifolds}
\label{intro:nilpotent}
Finally, in Section \ref{sect:nilpotent} we establish  
a correspondence between finitely generated, torsion-free 
nilpotent groups (for short, $\TT$-groups) and finite colimits 
of Hirsch extensions over $\Z$. The main result (Theorem \ref{thm:equivalence}) 
may be summarized as follows.

 \begin{theorem}
\label{thm:intro:equivalence}
There is a bijection between $\TT$-groups and finite colimits of Hirsch extensions 
which preserves cohomology algebras and induces an equivalence of categories 
between isomorphism classes of $\TT$-groups and aug\-mentation-preserving, 
isomorphism classes of finitely-generated, $1$-minimal models.
\end{theorem}

Every $\TT$-group $G$ can be realized as the fundamental group of a 
nilmanifold $M$, which is a classifying space for $G$. 
In \cite{Lambe-Priddy-1982}, Lambe and Priddy sharpened 
a classical result of Nomizu \cite{Nomizu}, which identifies 
the cohomology algebra of $M$ over $\R$ with the Lie algebra 
cohomology of the corresponding nilpotent real Lie algebra. 
They associated to $G$ a Lie algebra $L(G)$ 
defined over the subring of $\Q$ generated by the coefficients 
of the Hall polynomials of $G$, and showed that 
$H^*(G;S)\cong H^*(L(G);S)$, for a certain subring 
$S\subset \Q$. 

One can ask whether this isomorphism holds over the 
(possibly smaller) ring $R$. For instance, Kuzmin 
and Semenov showed in \cite{Kuzmin-Semenov} 
that $H^*(G;\Z)\cong H^*(L(G);\Z)$ 
when $G$ is a free nilpotent group of class $2$. 
It still appears to be an open question when this 
is the case in general.
The approach we take is illustrated 
by two examples in Section \ref{subsec:nil-min}, 
where we apply our theory 
of integral $1$-minimal models to the general
question of whether the group and Lie algebra
cohomologies of a $\TT$-group are isomorphic with
integer coefficients.  We will further 
develop this approach in forthcoming work \cite{Porter-Suciu-group-1min}.

\newpage
\section{Delta-sets, magmas, and cochain algebras}
\label{sect:cochain}

\subsection{$\Delta$-sets and $\Delta$-complexes}
\label{subsec:delta}
We start the section by reviewing the notion of a $\Delta$-complex, in 
the sense of Rourke and Sanderson \cite{Rourke-Sanderson}; 
see also Hatcher \cite{Hatcher} and Friedman \cite{Friedman}. 
We will view such a complex as the geometric realization of the 
corresponding $\Delta$-set, cf.~\cite{Friedman}.

An (abstract) $n$-simplex $\Delta^n$ is simply a finite ordered set, $(0,1, \dots, n)$. 
The face maps $d_i\colon \Delta^n\to \Delta^{n-1}$, given by omitting the $i$-th 
element in the set, satisfy $d_i d_j=d_{j-1}d_i$ whenever $0\le i<j\le n$.  
The geometric realization of the simplex, $\abs{\Delta^n}$, is the convex hull of $n+1$ 
affinely independent vectors in $\R^{n+1}$, endowed with the subspace topology; the 
face maps induce continuous maps, $d_i\colon \abs{\Delta^n}\to \abs{\Delta}^{n-1}$.

More generally, a {\em $\Delta$-set}\/ consists of a sequence of sets $X=\{X_n\}_{n\ge 0}$ 
and maps $d_i\colon X_{n}\to X_{n-1}$  for each $0\le i\le n$ such that 
$d_i d_j=d_{j-1}d_i$ whenever $i<j$.  This is the generalization of the 
notion of ordered (abstract) simplicial complex, where the sets $X_n$ 
are the sets of $n$-simplices and the maps $d_i$ are the face maps. 
We refer to $X^{(n)}=\{X_i\}_{i=0}^{n}$ as the {\em $n$-skeleton}\/ of 
the $\Delta$-set, and say that $X^{(n)}$ has dimension (at most) $n$.

The geometric realization of a $\Delta$-set  $X$ is the topological space
$\abs{X} = \coprod_{n\ge 0} X_n \times \abs{\Delta^n}/\!\sim$, 
where $\sim$ is the equivalence relation generated by $(x,d^i(p))\sim (d_i(x), p)$ 
for $x\in X_{n+1}$,  $p\in \abs{\Delta^n}$, and 
$0\le i\le n$, where 
$d^i\colon \abs{\Delta^n} \to \abs{\Delta^{n+1}}$
is the inclusion of the $i$-th face.
Such a space is called 
a $\Delta$-complex, and can be viewed either as a special kind of 
CW-complex, or a generalized simplicial complex.

The assignment $X\leadsto \abs{X}$ is functorial: 
if $f\colon X\to Y$ is a map of $\Delta$-sets (i.e., $f$ is a family of maps 
$f_n\colon X_n \to Y_n$ commuting with the face maps), there is an 
obvious realization, $\abs{f}\colon \abs{X} \to \abs{Y}$, and this is a  
(continuous) map of $\Delta$-complexes.

The chain complex of a $\Delta$-set $X$, denoted $C_*(X;\Z)$, coincides with the 
simplicial chain complex of its geometric realization: for each $n\ge 0$, the chain 
group $C_n(X)$ is the free abelian group on $X_n$, while the boundary 
maps  $\partial_n\colon C_n(X)\to C_{n-1}(X)$ are the linear maps given 
by $\partial_n=\sum_i (-1)^i d_i$.  If $B$ is an abelian group, the chain complex 
of $X$ with coefficients in $B$ is defined as $C_*(X;B)=C_*(X;\Z)\otimes B$. 
The cochain complex $C^*(X;B)$ is defined by setting 
$C^n(X;B)=\Hom(C_n(X),B)$ and dualizing the differentials. 
We denote by $H_*(X;B)$ and $H^*(X;B)$, respectively, the homology
groups of these complexes.

\subsection{From binary operations to $\Delta$-sets}
\label{subsec:construction}

Let $M=(M,\mu)$ be a {\em magma}, that is, a set $M$ equipped 
with a binary operation, $\mu\colon M\times M \to M$, commonly 
written as $(a_1,a_2) \mapsto a_1a_2$. 
These data determine a $2$-dimensional $\Delta$-set, 
denoted $\Delta^{(2)}(M)$, whose geometric realization 
can be described as follows.  There is a single vertex; each 
element $a\in M$ gives a $1$-simplex, and to each 
ordered pair of 
$1$-simplices, $a_1$ and $a_2$, we assign a $2$-simplex, 
$(a_1,a_2)$, with 
front face $a_1$, back face $a_2$, and third face 
equal to $a_1a_2$. 

\begin{figure}[h!]
\begin{tikzpicture}[scale = 0.6]
\draw (-2,0) -- (0, 3.464101616) -- (2,0) -- (-2,0);
\draw (-2,0) -- ( -4, 3.464101616) -- (0, 3.464101616);
\draw (2,0) -- ( 4, 3.464101616) -- (0, 3.464101616);
\draw (-2,0) -- (0, -3.464101616) -- (2,0) -- (-2,0);
\node at (0,3.8) {\footnotesize$0$};
\node at (-2.3,0) {\footnotesize$1$};
\node at (2.3,0) {\footnotesize$2$};
\node at  (-4.3, 3.464101616) {\footnotesize$3$};
\node at  (4.3, 3.464101616) {\footnotesize$3$};
\node at  (0, -3.7) {\footnotesize$3$};
\node at (-1.3, 1.732051) {\footnotesize$a_1$};
\node at (0, -0.3) {\footnotesize$a_2$};
\node at (3.3, 1.732051) {\footnotesize$a_3$};
\node at (1.5, 1.732051) {\footnotesize$a_1a_2$};
\node at (2,3.7) {\footnotesize$(a_1a_2)a_3$};
\node at (1.4, -1.732051) {\footnotesize$a_3$};
\node at (-1.5, -1.732051) {\footnotesize$a_2a_3$};
\node at (-3.5, 1.732051) {\footnotesize$a_2a_3$};
\node at (-2,3.7) {\footnotesize$a_1(a_2a_3)$};
\end{tikzpicture}
\caption{}
\label{fig:3-simplex}
\end{figure}

This construction is functorial, in the following sense. 
Suppose $h\colon (M,\mu)\to (M',\mu')$ is a morphism of magmas, 
that is, $\mu'(h(a_1),h(a_2))=h(\mu(a_1,a_2))$ for all 
$a_1,a_2\in M$. Then $h$ determines in a straightforward manner 
a simplicial map between the respective $\Delta$-complexes, 
$\Delta(h)\colon \Delta^{(2)}(M',\mu') \to\Delta^{(2)}(M,\mu)$, so 
that $\Delta(h\circ g)=\Delta(h)\circ \Delta(g)$.

Now suppose $M=(M,\mu)$ is a {\em semigroup}, that is, the operation $\mu$ 
on the magma $M$ is associative. Then, as indicated in Figure \ref{fig:3-simplex}, 
the $2$-dimensional $\Delta$-complex corresponding to $\Delta^{(2)}(M)$ extends to a
$3$-dimensional $\Delta$-set, whose $3$-simplices are ordered triples, $(a_1,a_2,a_3)$, 
of elements in $A$.  A routine computation shows that this construction can be pushed 
through in all dimensions; the upshot is summarized in the next lemma.

\begin{lemma}
\label{lem:delta-mu}
Let $M$ be a semigroup, and let $\Delta^{(2)}(M)$ be the  
$2$-dimensional $\Delta$-set determined by the underlying magma. 
Then $\Delta^{(2)}(M)$ is the $2$-skeleton of a $\Delta$-set, $\Delta(M)$,  
whose $n$-simplices are given by ordered $n$-tuples of elements in $M$.
\end{lemma}

Given a semigroup $M$, we define a $\Delta$-set $S(M)$, as follows. 
We let $S_n(M)$ equal to the set of all functions, $f$, from the 
$1$-simplices of the abstract $n$-simplex $\Delta^n$ to 
$M$ with the property that $f(i,\ell) = f(i,j)\cdot f(j,\ell)$ 
for all $ 0 \le i < j < \ell \le k$. The face maps 
$d_i\colon S_{n}(M)\to S_{n-1}(M)$  are given by the restriction 
of $f$ to the faces of $\Delta^{n}$. It is readily seen that 
the $\Delta$-set $S(M)$ coincides with $\Delta(M)$.

\begin{remark}
\label{rem:bar-monoid}
Of particular importance is the case when $M$ is a {\em monoid}, that is, 
a semigroup with multiplication $\mu\colon M\times M\to M$ and two-sided 
identity $e$. Then $\Delta(M)$ is the bar construction on $M$:  the corresponding 
$\Delta$-complex, $B(M)=\abs{\Delta(M)}$, has a single $0$-cell, 
and an $n$-cell $[g_1| \dots |g_n]$ for each $n$-tuple $(g_1, \dots ,g_n)\in M^n$. 
The chain complex $C_\ast(B(M);\Z)$ yields a resolution by free 
$\Z[M]$-modules of the group $\Z$, viewed as a trivial module over the monoid-ring 
$\Z[M]$. Finally, if $M=G$ is a group, then $B(G)$ is an Eilenberg--MacLane 
classifying space $K(G,1)$; see \cite[Ch.~10]{MacLane-1963} 
and also \cite{Brown,Cartan-Eilenberg, Eilenberg-MacLane}.
\end{remark}

\subsection{Cocycles and $\Delta$-complexes}
\label{subsec:delta-ext}
Let $M=(M,\mu)$ be a magma, with multiplication $\mu\colon M\times M\to M$ 
written as $\mu(a_1,a_2)=a_1a_2$, and let $B$ be an abelian group, 
together with a map of sets $\nu \colon M\times M \to B$, written 
$(a_1,a_2)\mapsto a_1\astt a_2$. Defining a binary operation 
on $M \times B$ by
\begin{equation}
\label{eq:assoc}
(a_1,b_1) \cdot (a_2,b_2)  =
(a_1a_2, b_1 + b_2 + a_1\astt a_2)
\end{equation}
for all $a_i \in M$ and $b_i \in B$
turns the set $M\times B$ into a magma which we call
the {\em extension of $(M, \mu)$ by $\nu$}.
It is readily seen that this binary operation  
is associative if and only if 
\begin{equation}
\label{eq:h-assoc}
a_2\ast a_3 + a_1\astt (a_2a_3) =  a_1 \astt a_2 + (a_1a_2)\astt a_3
\end{equation}
for all $a_i\in M$. 
If $(M, \mu)$ has a two-sided identity, $e$, and if 
$ (a,b_1)\cdot (e,b_2) = (e,b_2)\cdot (a,b_1)$ and 
$(e,b_1)\cdot (e,b_2)= (e, b_1+b_2)$ for all
$a \in M$ and $b_1,b_2 \in B$, then
the extension is called a \textit{central extension}.

From the correspondence between $2$-simplices in $\abs{\Delta(M)}$ 
and ordered pairs of elements in $M$, it follows that
$\nu$ may be viewed as an element in $C^2(\Delta(M);B)$. 
If the magma $(M,\mu)$ has a two-sided identity $e$, we say that 
$\nu$ is {\em normalized}\/ if $\nu (a,e) = \nu (e,a) = 0$ 
for all $a \in M$. In the case when $(M,\mu)$ is a monoid (that is,
the operation $\mu$ is associative and has a two-sided identity $e$), 
a cochain $\xi \in C^k(\Delta(M);B)$ is called {\em normalized}\/
if $\xi(a_1, \ldots , a_k) = 0$ whenever $a_i = e$ for some $i$.

The following lemma gives conditions on $(M,\mu)$
and $\nu$ for the extension of a semigroup to be a semigroup, 
the extension of a monoid to be a monoid, and
for the extension of a group to be a group. Note that a
monoid is a semigroup with identity.
 
\begin{lemma} 
\label{lem:assoc}
Given a magma $(M,\mu)$ and a map $\nu \colon M\times M \to B$, 
the extension $E = (M \times B, \cdot)$ of $M$ by the abelian group $B$ as
defined above has the following properties.
\begin{enumerate}
\item
\label{ext:p1}
Suppose $(M, \mu)$ has a two-sided identity $e$ and 
$\nu$ is a normalized cochain. Then
$E$ is a central extension.
\item 
\label{ext:p2}
Suppose $(M,\mu)$ is a semigroup. Then $(E,\cdot)$ is a semigroup
if and only if $\nu$ is a cocycle in $C^2(\Delta(M);B)$.
\item
\label{ext:p3}
Suppose $(M, \mu)$ is a monoid and $\nu$ is a normalized 
cocycle in $Z^2(\Delta(M);B)$. Then $(E,\cdot)$ is a monoid.
\item
\label{ext:p4}
Suppose $(M, \mu)$ is a group and $\nu$ is a normalized 
cocycle in $Z^2(\Delta(M);B)$. 
Then $(E,\cdot)$ is a group.
\end{enumerate}
\end{lemma}
\begin{proof}
Part \eqref{ext:p1} follows at once from the definitions. 

Part \eqref{ext:p2}:
let us view the map $\nu\colon M\times M\to B$ as 
a $B$-valued $2$-cochain on $\Delta(M)$. 
Using  Figure \ref{fig:3-simplex} to identify a triple
$(a_1, a_2, a_3)\in M^3$ with the standard
$3$-simplex on vertices $0, \dots, 3$ 
and then to identify ordered pairs of elements in $M$ 
with $2$-simplices, we get 
\begin{align*}
\delta \nu(a_1,a_2,a_3)
			& = \nu([1,2,3]) - \nu([0,2,3]) + \nu([0,1,3]) - \nu([0,1,2])
			\\
			& = \nu(a_2,a_3) - \nu(a_1a_2,a_3) 
						+ \nu(a_1, a_2a_3) - \nu(a_1,a_2)\\
			& = a_2\astt a_3 - (a_1a_2) \ast a_3 
						+ a_1 \astt (a_2a_3) - a_1 \astt a_2.
\end{align*}
Hence, the identity \eqref{eq:h-assoc} is satisfied precisely when 
$\nu$ is a cocycle.

Part \eqref{ext:p3}: we know from part \eqref{ext:p2} that 
$E=(M\times B,\cdot)$ is a semigroup. Moreover, if $e$ is a 
$2$-sided identity for $M$, a routine computation shows that 
$(e,0)$ is a $2$-sided identity for $E$.

Part \eqref{ext:p4}: it follows from parts 
\eqref{ext:p2} and \eqref{ext:p3} that $E$ is a monoid. 
A routine computation shows $(a,b) \in M \times B$ has 
$(a^{-1},-b - a* a^{-1})$ as a $2$-sided inverse, 
and we are done.
\end{proof}

\section{Differential graded algebras and homotopies}
\label{sect:dga}

\subsection{Differential graded algebras}
\label{subsec:dga}
Throughout this section, we work over a fixed coefficient ring $R$, 
assumed to be commutative and with unit $1$.  We start with 
some basic definitions.

\begin{definition}
\label{def:ga}
A {\em graded algebra}\/ over $R$ is an $R$-algebra $A$ 
such that the underlying $R$-module is a direct sum of 
$R$-modules, $A=\bigoplus_{i\ge 0} A^i$, and such that the
product $A\otimes_{R} A\to A$ sends $A^i \otimes_R A^j$ to $A^{i+j}$.
\end{definition}

We refer to the multiplication maps $\cup\colon A^i\otimes_{R} A^j\to A^{i+j}$, 
given by $\cup(a\otimes b)=a \cup b$ as the 
{\em cup-product maps};  
we also refer to the elements of $A^{i}$ as $i$-cochains. A morphism 
of graded algebras is a map of $R$-algebras preserving degrees. 

\begin{definition}
\label{def:dga}
A {\em differential graded algebra}\/ over $R$  (for short, a dga) 
is a graded $R$-algebra $A=\bigoplus_{i\ge 0} A^i$  
endowed with a degree $1$ map, $d\colon A\to A$, 
satisfying $d^2=0$ and the graded Leibniz rule,
$d(a \cup b) = da\cup b + (-1)^{\abs{a}} a\cup db$, 
for all homogenous elements $a, b\in A$, where $\abs{a}$ is the 
degree of $a$.  
\end{definition}

We denote by $[a]\in H^i(A)$ the cohomology class of a cocycle 
$a\in Z^i(A)$. As usual, the graded $R$-module $H^*(A)$ inherits 
an algebra structure from $A$. 

Observe that $A^0$ is a subring of $A$ and the structure map $R \to A$ 
sends the unit $1\in R$ to the unit of $A$, which we will also denote 
by $1$, and which necessarily has degree $0$. 
Consequently, $R$ may be viewed as a subring of $A^0$, and the 
graded pieces $A^i$ may be viewed as $A^0$-modules. 
We say that $A$ is {\em connected}\/ if the structure map 
$R\to A$ maps $R$ isomorphically to $A^0$, and we say that $A$ 
is {\em graded commutative} (for short, $A$ is a cdga) if 
$ab=(-1)^{\abs{a}\abs{b}} ba$ for all homogeneous 
elements $a, b\in A$.

A morphism of differential graded $R$-algebras is an $R$-linear 
map $\varphi\colon A \to B$ between two dgas which preserves 
the grading and commutes with the respective differentials
and products. The induced map in cohomology, 
$\varphi^*\colon H^*(A)\to H^*(B)$, $[a] \mapsto [\varphi(a)]$, 
is a morphism of graded $R$-algebras. The map $\varphi$ 
is called a \emph{quasi-isomor\-phism}\/ if $\varphi^*$ is an isomorphism. 
Two dgas are called \emph{weakly equivalent}\/ if there is a zig-zag 
of quasi-isomorphisms connecting one to the other; plainly, this is an 
equivalence relation among dgas. A dga $(A,d)$ is said to be {\em formal}\/ 
if it is weakly equivalent to its cohomology algebra, $H^*(A)$, endowed 
with the zero differential.

All these notions have partial analogues. Fix an integer $q\ge 1$. 
A dga map $\varphi\colon A \to B$ is a \emph{$q$-quasi-isomor\-phism}\/ 
if the induced homomorphism, $\varphi^*\colon H^i (A) \to H^i(B)$,  
is an isomorphism for $i \le q$ and a monomorphism for $i=q+1$. 
Two dgas are called \emph{$q$-equivalent}\/ if they may be connected 
by a zig-zag of $q$-quasi-isomorphisms. Finally, a dga $(A,d)$ is 
{\em $q$-formal}\/ if it is $q$-equivalent to $(H^*(A), d=0)$.

\begin{remark}
\label{rem:coho-vanish}
Suppose $\varphi\colon A \to B$ is a surjective $q$-quasi-isomorphism. 
Then the long exact sequence in cohomology induced by the 
exact sequence of cochain complexes
$0\to \ker(\varphi) \to A \xrightarrow{\varphi} B \to  0$ implies 
 that $H^i(\ker(\varphi)) = 0$ for $i \le q+1$.
\end{remark}

If $(A,d_A)$ and $(B,d_B)$ are two dgas, then the tensor product 
of the underlying graded $R$-modules, $A\otimes_{R} B$, acquires a 
dga structure, with multiplication and differential given on 
homogeneous elements by $ (a\otimes b)\cdot  (a'\otimes b')=
(-1)^{\abs{a'}\abs{b}}  aa' \otimes bb'$ and 
$d_{A\otimes_{R} B} (a\otimes b)=d_{A}(a) \otimes b+ 
(-1)^{\abs{a}} a\otimes  d_{B}(b)$. The direct product 
$A\times B$ also has a natural structure of a dga, 
with $ (a, b)\cdot  (a',b')= aa' \otimes bb'$ and 
$d_{A\times B} (a, b)=(d_{A}(a) , d_{B}(b))$. 
If $\varphi\colon A\to B$ and $\varphi'\colon A'\to B$ are two dga maps, 
then their {\em fiber product}, denoted $A\times_B A'$, is the sub-dga 
of $A\times B$ consisting of all pairs $(a,b)$ with $\varphi(a)=\varphi'(b)$.

\subsection{Cochain algebras}
\label{subsec:cochains}
The motivating example for us is the singular cochain algebra 
$C^*(X;R)$ on a space $X$, with coefficients in a commutative 
ring $R$. This is an $R$-dga, with differentials given by the usual 
coboundary maps, and with multiplication given by the cup product.
We will be mostly interested 
in the case when $X$ is a simplicial complex, or, more generally, a 
$\Delta$-complex (see \cite{Rourke-Sanderson, Hatcher, Friedman}). 
We will view such a complex as the geometric realization of the corresponding 
abstract simplicial complex or $\Delta$-set, respectively, and we will use 
the simplicial cochain algebra of $X$, still to be denoted by $C^*(X;R)$.  
Let us note that the structure map 
$R \to C^0(X;R)$ sends an element $r \in R$ to the 
cochain whose value on every vertex is $r$. 

\begin{example}
\label{ex:interval}
Let $I$ be the closed interval $[0, 1]$, viewed as a simplicial complex 
in the usual way, and let $C=C^\ast(I;R)$ be its cochain algebra over $R$.  
Then $C^0 = R \oplus R$ with generators $t_0, t_1$ corresponding
to the endpoints $0$ and $1$, and $C^1 = R$ with generator $u$.
The differential $d\colon C^0 \to C^1$ is given by
$d t_0 = -u$ and $d t_1 = u$, while 
the multiplication is given on generators by $t_i  t_j = \delta_{ij} t_i$, 
$t_0  u = u  t_1 =u$, and $u t_0 = t_1  u = 0$. Note that 
the cocycle $t_0+t_1$ is the unit of $C$.  Furthermore, $H^*(C)=R$, 
concentrated in degree $0$.
\end{example}

\begin{example}
\label{ex:chain-bar}
Now let $G$ be a group. Recall from Section \ref{subsec:construction} 
that the $\Delta$-complex $B(G)$ for the bar construction on 
$G$ has one $0$-cell, one $1$-cell $[g]$ for each $g \in G$,
and one $2$-cell for each ordered pair $[g_1| g_2]$ of elements
in $G$. Thus, the $1$-cochains are functions $f\colon G \to R$ and the
$2$-cochains are functions from $G\times G$ to $R$.
The cup product and differential  are given by 
$(f \cup h) ([g_1| g_2]) = f(g_1)\cdot h(g_2)$ and 
$(df)([g_1|g_2]) = f([g_1]) + f([g_2]) - f([g_1\cdot g_2])$, 
where $\,\cdot\,$ denotes the product in $R$ or $G$, depending on the 
context.
\end{example}

Following R.H.~Fox \cite{Fox} and J.H.C.~Whitehead \cite{Wh48, Wh}, 
we say that two maps of spaces, 
$f, g\colon X\to Y$, are {\em $n$-homotopic}\/ if $f\circ h\simeq g\circ h$, 
for every map $h\colon K\to X$ from a CW-complex $K$ of dimension 
at most $n$. A map $f\colon X\to Y$ is an $n$-homotopy equivalence 
(for some $n\ge 1$) if it admits an $n$-homotopy inverse. If such a map 
$f$ exists, we say that $X$ and $Y$ have the same {\em $n$-homotopy type}.  
Two CW-complexes, $X$ and $Y$, are said to be 
of the same {\em $n$-type}\/ if their $n$-skeleta have the same $(n-1)$-homotopy 
type. Any two connected CW-complexes have the same $1$-type, and 
they have the same $2$-type if and only if their fundamental groups are 
isomorphic.

A (cellular) map $f\colon X\to Y$ between two CW-complexes 
induces a morphism of dgas, $f^{\sharp} \colon C^\ast(Y;R) \to C^\ast(X;R)$, 
between the respective cochain algebras, and thus a morphism,  
$f^{\ast}\colon H^*(Y;R)\to H^*(X;R)$, between their 
cohomology algebras. If $f$ is a homotopy equivalence, 
then $f^{\sharp}$ is a quasi-isomorphism of $R$-dgas. 
The next result, which develops ideas from \cite{Wh}, was proved 
in \cite{Porter-Suciu-2021}.

\begin{theorem}[\cite{Porter-Suciu-2021}]
\label{thm:quasi-iso}
If $X$ and $Y$ are CW-complexes of the same $n$-type, then the 
cochain algebras $C^\ast(X;R)$ and $C^\ast(Y;R)$ are $(n-1)$-equivalent.
In particular, if $\pi_1(X)\cong \pi_1(Y)$, then $C^\ast(X;R)$ and $C^\ast(Y;R)$ 
are $1$-equivalent.
\end{theorem}

The $(n-1)$-equivalence between the $n$-skeleta of $X$ and $Y$ takes a special form,  
which we now recall, for it will be needed in the proof of Theorem \ref{thm:iso-pi1}.
By \cite[Theorem 6]{Wh}, there is a homotopy equivalence, $f$, from 
$\overline{X}^{(n)}  = X^{(n)} \vee \bigvee_{i\in I} S^n_i$ 
to $\overline{Y}^{(n)} =Y^{(n)}\vee \bigvee_{j\in J} S^n_j$, for some 
indexing sets $I$ and $J$. Let $q_X\colon \overline{X}^{(n)} \to X^{(n)}$ and 
$q_Y\colon \overline{Y}^{(n)} \to Y^{(n)}$ be the obvious collapse maps,
and consider the induced morphisms on cochain algebras, 
\begin{equation}
\label{eq:wh}
\begin{tikzcd}
C^\ast(X^{(n)};R) \ar[r, "q_X^{\sharp}"] 
&C^\ast(\overline{X}^{(n)};R)
& C^\ast(\overline{Y}^{(n)};R)
\arrow{l}[swap, pos=0.4]{f^{\sharp}} 
&  C^\ast(Y^{(n)};R)\, .\ar[l, pos=0.4, "q_Y^{\sharp}"']
\end{tikzcd}
\end{equation}
The map $f^{\sharp}$ is a quasi-isomorphism, while 
$q_X^{\sharp}$ and $q_Y^{\sharp}$ are $(n-1)$-quasi-isomorphisms; 
thus, \eqref{eq:wh} is the desired $(n-1)$-equivalence between 
$C^*(X;R)$ and $C^*(Y;R)$.

\subsection{Homotopy invariance}
\label{subsec:homotopy}
Let $C^\ast(I;R)$ be the cochain algebra of the interval $I$, as 
described in Example \ref{ex:interval}, and 
let $\eta_0, \eta_1\colon C^\ast(I;R) \to R$ denote the 
$R$-linear maps induced by restriction to the endpoints of $I$; 
that is to say, $\eta_i(t_j)=\delta_{ij}$, and $\eta_i(u)=0$. 

\begin{definition}
\label{def:homotopy}
Two dga maps, $\varphi_0, \varphi_1\colon A \to B$, are said 
to be \emph{homotopic}\/  (denoted $\varphi_0\simeq \varphi_1$) if 
there is a dga map $\Phi \colon A \to B \otimes_R C^\ast(I;R)$ such that 
the following diagram commutes for $i=0,1$:
\begin{equation}
\label{eq:homotopy}
\begin{gathered}
\begin{tikzcd}[column sep=3pc]
B\otimes_R C^{\ast}(I;R) \ar[r, "\id_B \otimes \eta_i"] & B\otimes_R  R
\phantom{.} \\
A\ar[u, "\Phi"]  \ar[r, "\varphi_i"]& B . \ar[equal]{u} 
\end{tikzcd}
\end{gathered}
\end{equation}
\end{definition}
From the commutativity of the diagram \eqref{eq:homotopy}
it follows that a homotopy $\Phi$ is given on elements $a\in A^i$ by
$\Phi(a)=\varphi_0(a)t_0 + \varphi_1(a)t_1 - c(a) u$, 
for some $c(a) \in B^{i-1}$. In particular, if $a\in A^0$, then $c(a)=0$, and so 
$\Phi(a)=\varphi_0(a)t_0 + \varphi_1(a)t_1$. The next theorem can be 
proven by a standard argument, in a manner similar to the proof of 
\cite[Proposition~12.8(i)]{FHT} (see also \cite[Remark 5.10.3]{Halperin}).

\begin{theorem}
\label{thm:cup-coho}
Homotopic dga maps induce the same map on cohomology; that is, 
if $\varphi_0\simeq \varphi_1$, then $\varphi_0^*= \varphi_1^*$.
\end{theorem}

\subsection{Augmented dgas and the wedge sum}
\label{subsec:augmented}

Let $(A,d_A)$ be a differential graded algebra over a unital commutative 
ring $R$. Let us view the ground ring $R$ as a dga concentrated in 
degree $0$ and with differential $d=0$. An {\em augmentation}\/ 
for $A$, then, is a dga-map, $\varepsilon_A \colon A \to R$. 
We call the triple $(A,d_A,\varepsilon_A)$ an 
{\em augmented dga}. A morphism in this category 
is a dga map, $\varphi \colon (A,d_A)\to (B,d_B)$, 
such that $\varepsilon_B\circ \varphi= \varepsilon_A$.

Recall that $A$ is connected if the structure map 
$\sigma_A \colon R \to A^0$ is an isomorphism
of rings; in this case we assume the augmentation
$\varepsilon_A \colon A \to R$ then restricts to an
isomorphism from $A^0$ to $R$. The
composition $\varepsilon_A \circ \sigma_A\colon R\to R$
then is an isomorphism of $R$-algebras, and hence, is the identity map. 
Thus, if $A$ is connected, it has a unique augmentation map.
Moreover, if $\varphi \colon A \to B$ is an augmentation-preserving 
morphism between connected $R$-dgas, the map $\varphi^0\colon A^0\to B^0$ 
may be identified with $\id_R$.  In general, though, a dga may admit 
many augmentations.

If $A$ and $B$ are two augmented dgas, we denote by 
$A\vee B = A\times_R B$ the fiber product of the augmentation maps 
$\varepsilon_A \colon A \to R$ and $\varepsilon_B \colon B \to R$.
Note that $(A \vee B)^0$ is the kernel of the map 
$(\varepsilon_A, -\varepsilon_B)\colon A^0\oplus B^0 \to R$, 
while  $(A\vee B )^i=A^i\oplus B^i$ for $i>0$.

The motivation for these definitions comes from topology. 
Let $X$ be a topological space, and let $C^*(X;R)$ 
be its singular cochain algebra. Choosing a basepoint $x_0\in X$ 
yields an augmentation, $\varepsilon_0\colon C^*(X;R)\to R$, which 
sends a $0$-cochain $\xi$ to its evaluation $\xi(x_0)\in R$ and 
any cochain of higher degree to $0$.  

\begin{example}
\label{ex:aug-cochain}
Let $C^*(I;R)$ be the cochain algebra of the unit interval $I=[0,1]$ 
as in Example \ref{ex:interval}, and let $x_0=0$. Then $\varepsilon_0(t_0)=1$, 
while $\varepsilon_0(t_1)=\varepsilon_0(u)=0$.
\end{example}

If $f\colon (X,x_0)\to (Y,y_0)$ 
is a pointed map, then the induced morphism of cochain algebras, 
$f^*\colon C^*(Y;R) \to C^*(X;R)$, preserves the respective augmentations.
Finally, if $X\vee Y$ is the wedge sum of two pointed spaces, then 
$C^*(X\vee Y;R)$ is the subalgebra of $C=C^*(X;R)\times C^*(Y;R)$ 
equal to $C^n$ in degrees $n>0$ and  $\{(a,b)\in C^0 : a(x_0)=b(y_0)\}$ 
in degree $0$. Therefore, $C^*(X\vee Y;R)$ isomorphic to $C^*(X;R) \vee C^*(Y;R)$.%

\subsection{Augmentation-preserving homotopies}
\label{subsec:aug-homotopy}

Two augmented dga maps, $\varphi_0, \varphi_1\colon A \to B$, are said 
to be \emph{augmentation-preserving homotopic}\/ if there is a 
homotopy $\Phi \colon A \to B \otimes_R C^\ast(I;R)$ between them 
such that the following diagram commutes,
\begin{equation}
\label{eq:aug-homotopy}
\begin{gathered}
\begin{tikzcd}[row sep=20pt, column sep=30pt]
A \ar[r, "\Phi"] \ar[d, "\varepsilon_A"] & B\otimes_R C^{\ast}(I;R)  
\ar[d, "\varepsilon_B \otimes \id"] \\
R  \ar[r] \ar[dr] & R\otimes_R C^{\ast}(I;R) & \\
& C^{\ast}(I;R) \ar[equal]{u}  ,
\end{tikzcd}
\end{gathered}
\end{equation}
where the diagonal map is the structure map for the $R$-algebra 
$C^{\ast}(I;R)$, which sends $1$ to $t_0+t_1$. As noted in 
Section \ref{subsec:homotopy}, the homotopy $\Phi$ is given on elements $a\in A^i$ by 
$\Phi(a)=\varphi_0(a)t_0 + \varphi_1(a)t_1 - c(a) u$, for some $c(a)\in B^{i-1}$. 
The commutativity of \eqref{eq:aug-homotopy} implies that $\varepsilon_B(c(a))=0$. 
When both $A$ and $B$ are connected and $A$ is generated in degree $1$, 
augmentation-preserving homotopies take a very special form, which we 
describe in the proof of the next lemma.

\begin{lemma}
\label{lem:miraculous}
Let $A$ and $B$ be augmented $R$-dgas such that $A$ and $B$ are connected
and $A$ is generated as a graded $R$-algebra by $A^1$. Let 
$\varphi_0, \varphi_1\colon A \to B$ be augmentation-preserving morphisms. 
If there is an augmentation-preserving homotopy between 
$\varphi_0$ and $\varphi_1$, then $\varphi_0 = \varphi_1$.
\end{lemma}

\begin{proof}
Let $\Phi \colon A \to B \otimes_R C^\ast(I;R)$ be an 
augmentation-preserving homotopy between $\varphi_0$ and $\varphi_1$. 
We claim that 
\begin{equation}
\label{eq:miracle1}
\Phi(a) = \varphi_0(a)t_0 + \varphi_1(a)t_1,\\
\end{equation}
for all $a\in A$. 

To prove this claim, let $a\in A^i$ and write as before 
$\Phi(a)=\varphi_0(a)t_0 + \varphi_1(a)t_1 - c(a) u$, for some $c(a)\in B^{i-1}$. 
When $a\in A^0$, we necessarily have $c(a)=0$. 
When $a \in A^1$, we have that $c(a) \in B^0$ and $\varepsilon_B(c(a))=0$; 
since $B$ is connected, it follows that $c(a)= 0$. 
Therefore, \eqref{eq:miracle1} holds for all $a \in A^0\oplus A^1$.
Now recall that $t_i  t_j = \delta_{ij} t_i$. 
Since $\varphi_0$, $\varphi_1$, and $\Phi$ 
are all maps of graded algebras, $A$ is generated in degree $1$,  
and \eqref{eq:miracle1} holds for all $a \in A^1$, it follows that 
equation \eqref{eq:miracle1} holds for all $a\in A$.

Using \eqref{eq:miracle1}, we have that
$(\Phi \circ d_A)(a)  = \varphi_0(d_A a)t_0 + \varphi_1(d_A a)t_1$ and 
\begin{align*}
(d_{B \otimes_R C^\ast(I;R)}\circ \Phi)(a)
		& = d_{B \otimes_R C^\ast(I;R)}
					(\varphi_0 (a)t_0 + \varphi_1(a)t_1)\\
		& = \varphi_0	(d_A a)t_0
		     +\varphi_1 (d_A a)t_1
		     - (\varphi_0(a) - \varphi_1(a)) u	
\end{align*}
for every $a\in A$.  Since $\Phi$ is a map of dgas, it follows that 
$\varphi_0(a) - \varphi_1(a)=0$ for all $a\in A$, and the proof is complete. 
\end{proof}

\section{The Steenrod $\cup_i$-products}
\label{sect:steenrod}

\subsection{The $\cup_i$ operations}
\label{subsec:cup-i}
We now enrich the notion of a differential graded algebra with extra structure, 
motivated by properties of the cochain algebra of a space, as laid out in the 
foundational paper of Steenrod \cite{Steenrod}, and further developed 
by Hirsch in \cite{Hirsch}.

Let $X$ be a $\Delta$-complex, and let $A=(C^*(X;R),d)$ be its cochain algebra 
with coefficients in a commutative ring $R$, with multiplication given by the 
cup product $\cup\colon A^p\otimes_{R}  A^q\to A^{p+q}$. This  
$R$-dga comes endowed with $R$-linear maps, 
$\cup_i\colon A^p\otimes_{R}  A^q\to A^{p+q-i}$, 
which coincide with the usual cup product when $i=0$, 
vanish if either $p<i$ or $q<i$, and satisfy
\begin{align}
\label{eq:cup-i-steenrod}
&d(a \cup_i b)  = (-1)^{\abs{a}+\abs{b}-i} a \cup_{i-1} b +
(-1)^{\abs{a}\abs{b}+\abs{a}+\abs{b}} b\cup_{i-1} a 
+ da \cup_i b + (-1)^{\abs{a}} a \cup_i db
\\
\label{eq:cup-i-hirsch}
&(a \cup b) \cup_1 c = a \cup (b\cup_1 c) + 
(-1)^{\abs{b}(\abs{c}-1)} (a \cup_1 c)\cup b 
\end{align}
for all homogeneous elements $a,b,c\in A$.  
We shall refer to \eqref{eq:cup-i-steenrod} as the ``Steenrod identities" and to 
\eqref{eq:cup-i-hirsch}, with the cup product to the left of the
cup-one product, as the ``Hirsch identity". 

Steenrod's $\cup_i$ operations enjoy the following naturality property. 
Suppose $f\colon X\to Y$ is a map of $\Delta$-complexes which 
preserves the ordering of the vertices of simplices.
Then, by \cite[Theorem 3.1]{Steenrod}, the induced map on cochains, 
$f^*\colon C^*(Y;R)\to C^*(X;R)$, is a morphism of differential graded
algebras that commutes with the $\cup_i$ products.

Steenrod's $\cup_i$ products also occur in the theory of 
non-commutative differential forms, as developed by 
M.~Karoubi, N.~Battikh, and A.~Abbassi. We refer to 
our prior work \cite{Porter-Suciu-2021} for an overview 
of this subject and detailed references. 

\subsection{Cup and cup-one operations on $1$-cochains}
\label{subsec:cup-one-cochains}
Henceforth, we will focus on the aforementioned operations on cochains in low degrees. 
Let  $(A,d)$ be an $R$-dga and assume we have an $R$-linear map 
$\cup_1\colon A^1\otimes_{R}  A^1\to A^{1}$ that satisfies the Steenrod 
identity,
\begin{equation}
\label{eq:steenrod-1c}
d(a \cup_1 b)  = -a\cup b- b\cup a  + da \cup_1 b - a\cup_1 db,
\end{equation} 
for all $a,b\in A^1$. In particular, if $a,b\in Z^1(A)$ are $1$-cocyles, 
we then have
\begin{equation}
\label{eq:cup-1-1}
d(a \cup_1 b)  = -( a\cup b + b\cup a).
\end{equation}
Under these assumptions, the operation 
$\cup_1\colon Z^1(A)\otimes_{R}  Z^1(A)\to A^2$ 
provides an explicit witness for the non-commutativity 
of the multiplication map 
$\cup\colon Z^1(A)\otimes_{R}  Z^1(A)\to Z^2(A)$ and shows
that $uv = -vu$ for elements $u,v \in H^1(A)$.  

Now let $X$ be a $\Delta$-complex and let $C^*(X;R)$ be its cochain algebra 
with coefficients in a commutative ring $R$.
By  \cite[Theorem 2.1]{Steenrod}, $u\cup_1 v=0$ 
if either $u$ or $v$ is a $0$-cochain. Formulas for computing the 
cup products and cup-one products for $1$-cochains $u, v\in C^1(X;R)$ 
are as follows:
\begin{equation}
\label{eq:cup-simplicial}
(u \cup v)(s) = u(e_1)\cdot v(e_2), \quad
(u \cup_1\! v)(e) =  u(e)\cdot v(e) ,
\end{equation} 
where in the first formula $s$ is a $2$-simplex with front face $e_1$ 
and back face $e_2$, while in the second formula
$e$ is a $1$-simplex, and $\,\cdot\,$ denotes the product in $R$. 
In particular, the $\cup_1$-product on $C^1(X;R)$ is both associative 
and commutative, and thus defines an $R$-algebra structure on $C^{\le 1}(X;R)$. 

\begin{example}
\label{ex:cochain-I}
For the cochain algebra $C=C^{\ast}(I;R)$ from Example \ref{ex:interval}, 
the cup-one product $C^1 \otimes_R C^1 \to C^1$ is given by
$u \cup_1 u = u$.
\end{example}

\begin{example}
\label{ex:cochains-bg}
Let $G$ be a group, and let $C^*(B(G);R)$ be 
the cochain algebra of 
the bar construction on $G$, as described in Example \ref{ex:chain-bar}. 
The $\cup_1$-product on $C^1(B(G);R)$ 
is given by $(f \cup_1 h) ([g]) =  f(g) \cdot h(g)$.
\end{example}

\subsection{Graded algebras with cup-one products}
\label{subsec:graded-cup1}
Recall from Section \ref{subsec:dga} that the structure map $R \to A$ 
sends the unit $1\in R$ to the unit of $A$ (which belongs to $A^0$);  
thus, $R$ may be viewed as a subring of $A^0$. Furthermore, the 
restriction of the map $\cup\colon A^0\otimes A^1 \to A^1$ 
to $R\otimes_{R} A^1 \to A^1$ may be identified with the identity of $A^1$, 
while the restriction of the map $\cup\colon A^0\otimes_{R} A^0 \to A^0$ 
to $R\otimes_{R} R \to R$ may be identified with the identity of $R$.

Given a graded $R$-algebra $A$, we let $D^2(A)$ denote 
the {\em decomposables}\/ in $A^2$, that is, the 
image of the cup-product map $A^1\otimes_{R} A^1\to A^2$, 
or, the $R$-submodule of $A^2$ spanned by all elements 
of the form $a\cup b$ with $a,b\in A^1$.  

\begin{definition}
\label{def:gr-cup1}
A  {\em graded $R$-algebra with cup-one products}\/ is a graded $R$-algebra 
$A$ equipped with cup-one product maps, $\cup_1\colon A^1 \otimes_R A^1 \to A^1$ and 
$\cup_1 \colon D^2(A) \otimes_R A^1 \to A^2$, such that 
\begin{enumerate}[label=(\roman*), itemsep=2pt]
\item \label{grc-1}
The $R$-module $R \oplus A^1$, with multiplication defined by the maps 
$R\otimes_{R} R\to R$ and $R\otimes_{R} A^1\to A^1$ mentioned above, 
together with the map $\cup_1\colon A^1 \otimes_R A^1 \to A^1$ is a 
commutative, graded $R$-algebra, with identity $1\in R$.
\item \label{grc-2}
The $\cup$- and $\cup_1$-maps satisfy the Hirsch identity 
\begin{equation}
\label{eq:hirsch-1c}
(a \cup b) \cup_1 c = a \cup (b\cup_1 c) + (a \cup_1 c)\cup b
\end{equation} 
for all $a,b,c\in A^1$.
\end{enumerate}
\end{definition}

A morphism of graded algebras with cup-one products is a  
map $\varphi \colon A \to B$ between two such objects which  
is a map of graded algebras such that 
$\varphi(a_1\cup_1 a_2) = \varphi(a_1) \cup_1 \varphi(a_2)$  
for all $a_1, a_2\in A^1$.

\begin{lemma}
\label{lem:cup1bin-tensor}
Let $A$ and $B$ be two graded $R$-algebras with cup-one products. 
Then the tensor product $A \otimes_R B$ is again a graded algebra 
with cup-one products.
\end{lemma}

\begin{proof}
First, we extend the $\cup_1$-products on $A^1$ and $B^1$ to a $\cup_1$-product 
on $(A\otimes_R B)^1=(A^1\otimes_R B^0)\oplus (A^0\otimes_R B^1)$ by setting 
\begin{gather}
\begin{aligned}
\label{eq:cup1-tensor}
(a_1 \otimes b_0)\cup_1(a'_1 \otimes b'_0)&=(a_1\cup_1a_1')\otimes b_0b'_0\\
(a_0 \otimes b_1)\cup_1(a'_0 \otimes b'_1)&=a_0a_0'\otimes (b_1\cup_1 b'_1)\\
(a_1 \otimes b_0)\cup_1(a'_0 \otimes b'_1)&=(a_0 \otimes b_1)\cup_1(a'_1 \otimes b'_0)=0
\end{aligned}
\end{gather}
for all $a_i,a_i'\in A^i$ and $b_i,b_i'\in B^i$ ($i=0,1$) and extending linearly to 
$(A\otimes_R B)^1$. Since the $\cup_1$-product on $A^1$ and $B^1$ and the multiplication 
on $A^0$ and $B^0$ are all commutative, it follows that the  $\cup_1$-product on 
$(A\otimes_R B)^1$ is also commutative. 

Next, note that $ (A\otimes_R B)^2 = (A^2\otimes_R B^0)\oplus (A^1\otimes_R B^1) 
\oplus (A^0\otimes_R B^2)$ and the $R$-submodule $D^2(A\otimes_R B)$ is generated 
by all possible cup-products of elements from $A^1\otimes_R B^0$ and $A^0\otimes_R B^1$. 
We now extend the cup-one product maps 
$\cup_1 \colon D^2(A) \otimes_R A^1 \to A^2$ and 
$\cup_1 \colon D^2(B) \otimes_R B^1 \to B^2$ to a map 
$\cup_1 \colon D^2(A\otimes_R B) \otimes_R (A\otimes_R B)^1 \to (A\otimes_R B)^2$ 
by setting
\begin{gather}
\begin{aligned}
\label{eq:cup1-tensor-deg2}
((a_1 \otimes b_0)\cup (a_1' \otimes b_0'))\cup_1 (a_1'' \otimes b_0'')
&=((a_1\cup a_1')\cup_1 a_1'')\otimes b_0b_0'b_0''
\\
((a_0 \otimes b_1)\cup (a_0' \otimes b_1'))\cup_1 (a_0'' \otimes b_1'')
&=a_0a_0'a_0''\otimes ((b_1\cup b_1')\cup_1 b_1'')
\\
((a_1 \otimes b_0)\cup (a_0 \otimes b_1))\cup_1 (a_0'\otimes b_1')
&= a_1a_0a_0' \otimes b_0 (b_1\cup_1 b_1')
\\
((a_0 \otimes b_1)\cup (a_1 \otimes b_0))\cup_1 (a_1' \otimes b_0')
&= a_0(a_1\cup_1 a_1') \otimes b_1 b_0b_0'
\end{aligned}
\end{gather}
for $a_i,a_i',a_i''\in A^i$ and $b_i,b_i',b_i''\in B^i$ ($i=0,1$), with the other 
$4$ types of products equal to $0$, and then extending linearly to 
$D^2(A\otimes_R B) \otimes_R (A\otimes_R B)^1$. 
To verify that this operation is well defined, suppose 
$(a_1\otimes b_0)\cup (a'_1\otimes b'_0)=
(\bar a_1\otimes \bar b_0)\cup (\bar a'_1\otimes \bar b'_0)$. 
Then $a_1 \cup a_1^\prime = \bar a_1 \cup \bar a_1^\prime$
and $b_0b_0^\prime = \bar b_0 \bar b_0^\prime$, 
and it follows that the operation is well defined on
products such as those in the first line of equation
\eqref{eq:cup1-tensor-deg2}. The other cases follow similarly.

Since the Hirsch identity \eqref{eq:hirsch-1c}
holds for both $A$ and $B$, it also holds for $A\otimes_R B$; for instance, 
\begin{align*}
\begin{split}
(a_1 \otimes b_0 \cup a_1'\otimes b_0')\cup_1 (a_1''\otimes b_0'') 
&= ((a_1 \cup a_1')\cup_1 a_1'') \otimes b_0b_0'b_0'' \\
&=((a_1 \cup_1 a_1'') \cup a_1' + a_1\cup (a_1' \cup_1 a_1'')) \otimes b_0b_0'b_0'' \\
&=(a_1\otimes b_0 \cup_1 a_1''\otimes b_0'') \cup a_1' \otimes b_1' \\
&\qquad + a_1\otimes b_0\cup (a_1' \otimes b_0'\cup_1 a_1''\otimes b_0'') ,
\end{split}
\end{align*}
and similarly for the other types of $\cup$- and $\cup_1$-products. 
This completes the proof.%
\end{proof}

\subsection{Cup-one differential graded algebras}
\label{subsec:cupd-formula}
In this section we make a definition that will play an important role in our investigation. 
We begin with some motivation. Note that if a dga is generated by a set of elements
$\{x_i\}_{i\in J}$ in degree $1$, then the Leibniz rule gives a formula for the differential 
of any product of the $x_i$ as a sum of cup products of the $x_i$ and $dx_i$. 
Hence, the differential on the algebra is completely determined by the differentials
of the generators $x_i$.
This raises the question of whether there is a formula for the differential of cup-one 
products of the generators $x_i$ that allows one to write $d(x_i\cup_1 x_j)$
as a sum of cup products of $1$-cochains. If so, then it follows that if a dga is
generated by elements $x_i$ in degree one and by iterated cup-one products 
of the $x_i$, then the differential on the algebra is completely determined by
the differentials of the $x_i$.   

The next definition answers this question by giving as part of hypothesis \ref{cup1-3} 
a formula for the differential of a cup-one product of $1$-cochains 
that, along with the Hirsch identity, allows one to write the differential
of cup-one products of $1$-cochains as a sum of cup products.
This definition is a slight modification of a notion introduced 
in \cite{Porter-Suciu-2021}, better adapted to the current context 
by including the additional hypothesis  \ref{cup1-2.5}.

\begin{definition}
\label{def:cup-one-d-algebra}
A differential graded $R$-algebra $(A,d)$ is called a 
{\em cup-one differential graded algebra} if 
the following conditions hold.
\begin{enumerate}[label=(\roman*), itemsep=2pt]
\item \label{cup1-1}
$A$ is a graded $R$-algebra with cup-one products.
\item \label{cup1-2}
There is an $R$-linear map $\circ \colon D^2(A) \otimes_R D^2(A) \to D^2(A)$ 
such that 
\begin{equation}
\label{eq:circ-op}
(u\cup v) \circ (w\cup z) = (u \cup_1 w) \cup (v \cup_1 z)
\end{equation}
for all  $u,v,w,z\in A^1$. 
\item \label{cup1-2.5}
The differential $d$  and the $\cup$ and $\cup_1$ products satisfy the identity 
\begin{equation}
\label{eq:c0d}
a \cup_1 dc = a \cup c-c\cup a 
\end{equation}
for all $a\in A^1$ and $c\in A^0$.

\item \label{cup1-3}
The differential $d$ satisfies the ``$\cupd$ formula,"
\begin{equation}
\label{eq:c1d}
d(a \cup_1 b) = - a \cup b - b \cup a
+ da \cup_1 b + db \cup_1 a - da \circ db,
\end{equation}
for all $a,b\in A^1$ with $da,db \in D^2(A)$.
\end{enumerate}
\end{definition}

\begin{remark}
\label{rem:cupd-steenrod}
Formula \eqref{eq:c1d} comes from Steenrod's definition of the $\cup_i$
products in a cochain algebra $A=C^*(X;R)$, as follows. By 
equation \eqref{eq:steenrod-1c}, we have that
$d(a \cup_1 b) = - a \cup b - b \cup a + da \cup_1 b- a \cup_1 db$ 
for all $a,b\in A^1$. 
If $da$ is decomposable, then $da \cup_1 b$ can
be written as a sum of cup products using the Hirsch
formula \eqref{eq:hirsch-1c}. This leaves the problem 
of writing $a \cup_1 db$ as a sum of cup products.
By a direct computation using Steenrod's definition of
the cup-one product $\cup_1\colon A^1 \otimes_R A^2\to A^2$,  
it follows that
\begin{equation}
\label{eq:cup1-again}
a \cup_1 (b_1 \cup b_2) = da \circ (b_1 \cup b_2) 
  - (b_1\cup b_2) \cup_1 a,
\end{equation}
where we assume $da$ is decomposable and $\circ$ is given 
by \eqref{eq:circ-op}. This then gives the $\cupd$ formula, 
equation \eqref{eq:c1d}. 
\end{remark}

Our motivation for Definition \ref{def:cup-one-d-algebra} arises from the 
cochain algebras of $\Delta$-complexes. As shown in 
\cite[Theorem~4.4]{Porter-Suciu-2021}, such algebras 
are indeed $\cup_1$-algebras. We briefly review this result, 
with the necessary modifications for our context here.

\begin{theorem}[\cite{Porter-Suciu-2021}]
\label{thm:cochain-cup1}
Let $X$ be a non-empty $\Delta$-complex, and let $R$ be a unital 
commutative ring. The the cochain algebra $(C^\ast(X;R),\delta)$ 
is a cup-one dga.
\end{theorem}

\begin{proof}
As we saw in Sections \ref{subsec:cochains} and \ref{subsec:cup-one-cochains}, 
the cellular cochain algebra $C=(C^\ast(X;R),\delta)$ is a graded algebra 
with cup-one products. Moreover, it is a differential graded algebra, 
and the Steenrod identities \eqref{eq:cup-i-steenrod} hold in full generality.  
Setting $(c_1 \circ c_2)(s) = c_1(s)\cdot c_2(s)$ for any $2$-cochains 
$c_1,c_2$ and any $2$-simplex $s$ defines an $R$-linear map 
$\circ\colon C^2 \otimes_R C^2 \to C^2$. It follows straight 
from the definitions of the $\cup$- and $\cup_1$-products that 
\begin{equation}
\label{eq:circ-simp}
(u \cup v)(s) \cdot (w \cup z)(s)= 
((u \cup_1 w) \cup (v \cup_1 z))(s),
\end{equation}
for all $1$-cochains $u,v,w,z$. 
Thus, the restriction of the $\circ$-map to decomposable elements 
yields a map, $\circ\colon D^2(C) \otimes_R D^2(C) \to D^2(C)$, 
which clearly obeys formula \eqref{eq:circ-op}. It is now 
straightforward to verify that the simplicial differential $d$ satisfies 
formula \eqref{eq:c1d}.

It remains to check that formula \eqref{eq:c0d} holds. Given a $1$-cochain $u$, 
a $0$-cochain $c$, and a $1$-simplex $e$ with endpoints $v_0$ and $v_1$, 
we have
\begin{gather}
\begin{aligned}
\label{eq:c0d-cochains}
(u \cup_1 \delta c)(e) 
&= u(e)\cdot (\delta c)(e)
= u(e)\cdot c(v_1-v_0)\\
&= u(e)\cdot c(v_1)-c(v_0) \cdot u(e)
=(u \cup c)(e) - (c\cup u) (e),
\end{aligned}
\end{gather}
and this completes the proof.
\end{proof}

\begin{remark}
\label{rem:cup2}
Comparing the definition of the map $\circ\colon C^2 \otimes_R C^2 \to C^2$ 
given in the above proof to that of Steenrod's map 
$\cup_2\colon C^2 \otimes_R C^2 \to C^2$, we readily see that 
these two maps coincide. Moreover, as shown in Remark \ref{rem:cupd-steenrod}, 
in this case the $\cupd$ formula \eqref{eq:c1d} is a consequence 
of Steenrod's formula \eqref{eq:steenrod-1c}.
\end{remark}

\subsection{Tensor products of $\cup_1$-dgas}
\label{subsec:tensor-dga}
We conclude this section with a result showing that the category of 
$\cup_1$-dgas is closed under taking tensor products. 

\begin{proposition}
\label{prop:cup1alg-tensor}
If $(A,d_A)$ and $(B,d_B)$ are cup-one differential graded
algebras, then the tensor product $(A \otimes_R B,d_{A\otimes B})$ is 
again a cup-one differential graded algebra.
\end{proposition}

\begin{proof}
By Lemma \ref{lem:cup1bin-tensor}, $A \otimes_{R} B$ is a graded algebra with 
cup-one products. The $\circ$ operations on $D^2(A)$ and $D^2(B)$ extend 
to a binary operation, $\circ \colon D^2(A\otimes_R B) \otimes_R D^2(A\otimes_R B) \to 
D^2(A\otimes_R B)$, by letting
$[(a_1\otimes b_0)\cup (a'_1\otimes b'_0)] \circ [(a''_1\otimes b''_0)\cup (a'''_1\otimes b'''_0)] $
equal to $ [(a_1 \cup a'_1)\circ (a''_1\cup a'''_1)]\otimes b_0b'_0b''_0b'''_0$, 
and so on. To verify that this operation is well-defined, suppose 
$(a_1\otimes b_0)\cup (a'_1\otimes b'_0)=
(\bar a_1\otimes \bar b_0)\cup (\bar a'_1\otimes \bar b'_0)$
and 
$(a''_1\otimes b''_0)\cup (a'''_1\otimes b'''_0)=
(\bar a''_1\otimes \bar b''_0)\cup (\bar a'''_1\otimes \bar b'''_0)$. 
Then $a_1\cup a'_1=\bar a_1\cup \bar a'_1$ and 
$a''_1\cup a'''_1=\bar a''_1\cup \bar a'''_1$, and likewise, 
$b_0b'_0 = \bar b_0 \bar b'_0 $ and 
$b''_0b'''_0 =  \bar b''_0 \bar b'''_0$. The claim follows.%

Using \eqref{eq:cup1-tensor}, it is readily checked that 
equation \eqref{eq:circ-op} holds for $(A \otimes_R B,d_{A\otimes B})$. 
Next, we verify that equation \eqref{eq:c0d} holds:
\begin{gather}
\begin{aligned}
\label{eq:c0d-tensor}
(a_1\otimes b_0) \cup_1 d_{A \otimes B}(a_0'\otimes b_0') 
&= (a_1\otimes b_0) \cup_1 (d_{A}(a_0')b_0' +a_0'd_B(b_0')) \\
&=(a_1\cup_1 d_{A}(a_0')) \otimes b_0b_0'\\
&=(a_1\cup a_0' -a_0'\cup a_1) \otimes b_0b_0'\\
&= (a_1\otimes b_0) \cup (a_0' \otimes b_0') - 
(a_0'\otimes b_0')\cup (a_1\otimes b_0),
\end{aligned}
\end{gather}
and similarly for the other cases. 

The last step is to verify that equation \eqref{eq:c1d} holds for the tensor product 
of $A$ and $B$. First let $a\in A^1$ such that $d_A(a)\in D^2(A)$, 
and let $b\in B^0$. It is readily seen that 
$d_{A\otimes B}(a\otimes b)\in D^2(A\otimes_R B)$; for instance, 
if $d_A(a)=u\cup v$ for some $u,v\in A^1$, then
\begin{equation}
\label{eq:tensor-0}
d_{A\otimes B}(a\otimes b)=d_Aa\otimes b-a\otimes d_Bb=
(u\otimes 1)\cup (v\otimes b) - (a\otimes 1)\cup (1\otimes d_Bb).
\end{equation}
Now also let $a'$ be an element in $A^1$ such that $d_A(a')\in D^2(A)$. 
and let $b'\in B^0$.  Using equations \eqref{eq:circ-op} and \eqref{eq:cup1-tensor}, 
we find that 
\begin{gather}
\begin{aligned}
\label{eq:tensor-1}
(d_{A}(a) \otimes b)\circ (a'\otimes d_B (b'))&=
((u\cup v)\otimes b)\circ ((a'\otimes 1)\cup (1\otimes d_B (b')))\\
&=
(u\cup_1 (a'\otimes b)) \cup ((v\otimes 1)\cup_1 (1\otimes d_B (b')))\\
&=((u\cup_1 a') \otimes b )\cup 0\\
&=0,
\end{aligned}
\end{gather}
and similarly $(a \otimes d_B(b))\circ (d_A (a')\otimes b')=0$. 
Furthermore, using equations \eqref{eq:circ-op} and \eqref{eq:c0d}, we get
\begin{gather}
\begin{aligned}
\label{eq:tensor-2}
(a \otimes d_B(b))\circ (a'\otimes d_B(b'))&=
((a\otimes 1)\cup (1 \otimes d_B(b)))\circ((a'\otimes 1)\cup (1 \otimes d_B(b')))\\
&=((a\otimes 1)\cup_1 (a'\otimes 1))\cup ((1 \otimes d_B(b))\cup_1 (1 \otimes d_B(b'))) \\
&=(a\cup_1 a')\otimes  (d_B(b' )\cup_1 d_B(b))\\
&=(a\cup_1 a')\otimes  (d_B(b')b -bd_B(b')).
\end{aligned}
\end{gather}

Finally, using formula \eqref{eq:c1d} for $A$ as well as 
equations \eqref{eq:tensor-1} and \eqref{eq:tensor-2} we find that 
\begin{multline*}
d_{A \otimes B} [(a \otimes b)\cup_1 (a' \otimes b') ]=
- (a \otimes b) \cup (a' \otimes b') -
(a' \otimes b')\cup (a \otimes b) + \\
\quad
d_{A\otimes B} (a \otimes b) \cup_1 (a' \otimes b') + 
d_{A\otimes B} (a' \otimes b') \cup_1 (a \otimes b) - 
d_{A\otimes B}(a \otimes b) \circ d_{A\otimes B} (a' \otimes b').
\end{multline*}
This shows that \eqref{eq:c1d} holds for elements in
$A \otimes_R B$ of the form 
$(a \otimes b) \cup_1 (a^\prime \otimes b^\prime)$ with
$|a| = |a^\prime| = 1$ and $|b| = |b^\prime|=0$.
The case $|a| = |b^\prime| =1$, 
$|b|= |a^\prime| = 0$ follows using similar computations
to show that in this case the right side of 
equation \eqref{eq:c1d} equals zero.
The case
$|a| = |a^\prime| = 0$, $|b| =|b^\prime|=1$ follows by 
the same computations as in the first case with the elements
in $A$ and $B$ interchanged.
\end{proof}

\section{Binomial cup-one differential graded algebras}
\label{sect:bin-cup1}
In this section we begin by reviewing the definition and basic properties of
binomial rings and then define $\F_p$-binomial algebras for $p$ a prime. 
This leads to the definition of binomial cup-one differential graded algebras 
over the ring $R = \Z$ or $\F_p$. A consequence of including the
binomial algebra structure is that it then follows that the cohomology 
of the free binomial cup-one differential graded algebra
on a single generator in degree $1$ is isomorphic
to the cohomology ring $H^\ast(K(R,1);R)$ of the
Eilenberg--MacLane space $K(R,1)$ with $R=\Z$ or $\F_p$
(see Theorem \ref{thm:coho-iso-R}).
   
\subsection{Binomial rings and $R$-valued polynomials}
\label{subsec:binomial}
Following P.~Hall \cite{Hall-1976} and J.~Elliott \cite{Elliott-2006},
we say that a commutative ring $A$ is a {\em binomial ring}\/ 
if $A$ is torsion-free (as a $\Z$-module) and the element
\begin{equation}
\label{eq:d-a}
\tbinom{a}{n}\coloneqq a(a-1)\cdots (a-n+1)/n! \in A\otimes_{\Z} \Q
\end{equation}
lies in $A$ for every $a\in A$ and every $n>0$.  Therefore 
we have maps $\zeta_n\colon A\to A$, $a\mapsto \binom{a}{n}$ 
for all $n\in \N$, with the convention that $\zeta_0(a)=1$ for all $a\in A$. 

Let $(x)_n \coloneqq x(x-1)\cdots (x-n+1)\in \Z[x]$ be the ``falling factorial" 
polynomial. Writing $(x)_n=\sum_{k=0}^{n} s(n,k)x^k$, the coefficients 
$s(n,k)$ of this polynomial are the Stirling numbers of the first kind. 
Now note that numerator of the fraction in \eqref{eq:d-a} is 
obtained by evaluating the polynomial $(x)_n$ at the value $x=a$, 
and so we may also write $\zeta_n(a)=\frac{(a)_n}{n!}$. 

The next lemma follows straight from the definitions.

\begin{lemma}
\label{lem:br}
Let $R$ be a binomial ring and let $M=R^{\bX}$ be a free $R$-module on 
a set $\bX$. Then the dual module, $M^{\vee}=\Hom_R(M,R)$, 
is a binomial ring with product defined by
$f\cdot g (x) = f(x)\cdot g(x)$ for all $x \in \bX$ and 
with maps $\zeta_n\colon M^{\vee} \to M^{\vee}$ given by
$\zeta_n(f)(x) = \zeta_n(f(x))$ for all $f \in M^{\vee}$ and $x \in \bX$.
\end{lemma}

Now suppose the binomial ring $R$ is an integral domain, and let 
$K=\Frac(R)$ be its field of fractions. Let $K[\bX]$ be the ring of polynomials in a 
set of formal variables $\bX$, with coefficients in $K$. 
Following \cite{CC97, Elliott-2006}, we define 
the {\em ring of $R$-valued polynomials}\/ (in the 
variables $\bX$ and with coefficients in $K$) as the subring 
\begin{equation}
\label{eq:bin-int}
\Int(R^{\bX})\coloneqq \{p\in K[\bX] \mid p(R^{\bX})\subseteq R\}.
\end{equation}
If, moreover, the domain $R$ has characteristic $0$ (that is, 
$R$ is torsion-free as an abelian group), then $\Int(R^{\bX})$ is a 
binomial ring, generated by the polynomials 
$\binom{\bX}{\mathbf{n}}\coloneqq \prod_{x\in \bX}\binom{x}{\mathbf{n}_x}$, 
for all multi-indices $\mathbf{n}=(\mathbf{n}_x)_{x\in \bX} \in \bigoplus_{\bX} \Z_{\ge 0}$,  
and $\Int(R^{\bX})$ satisfies a type of universality property which makes it into 
the {\em free binomial ring}\/ on variables in $\bX$ (see Corollary \ref{cor:extend-map}). 
As a consequence (at least when $R=\Z$), any binomial ring is a quotient 
of $\Int(R^{\bX})$, for some set $\bX$.

Let $I\colon \bX\to \Z_{\ge 0}$ be a function whose support, 
$\supp(I) \coloneqq \{x \in \bX \mid I(x)\ne 0 \}$, is a finite set. 
Given a binomial domain $R$ with field of fractions $K$, we 
associate to such a function the polynomial $\zeta_{I} \in K[\bX] $ given by
\begin{equation}
\label{eq:zeta-x}
\zeta_I  \coloneqq \prod_{x\in \bX} \zeta_{I(x)}.
\end{equation}
Note that this possibly infinite product is well defined, since the support of $I$ is finite 
and since $\z_0(x) = 1$ for all $x \in \bX$. Clearly, $\zeta_I$ is an $R$-valued 
polynomial in $\Int(R^{\bX})$. That is to say, given an element 
$\ba=\sum_{x\in \bX} \ba_x x \in R^{\bX}$ with $\ba_x\in R$ and 
$\ba_x=0$ for all but finitely many $x\in \bX$, the evaluation 
$\zeta_I(\ba)=\prod_{x\in \bX} \zeta_{I(x)} (\ba_x)$ is an element of $R$.

\subsection{A basis for integer-valued polynomials}
\label{subsec:basis}
We restrict now to the case when $R=\Z$. 
The next theorem provides a useful $\Z$-basis for the ring $\Int(\Z^{\bX})$ 
of integer-valued polynomials; for a proof, we refer to \cite[Proposition XI.I.12]{CC97} 
and \cite[Lemma 2.2]{Elliott-2006}. 

\begin{theorem}[\cite{CC97,Elliott-2006}]
\label{thm:basis}
The $\Z$-module $\Int(\Z^{\bX})$ is free, with basis consisting of all 
polynomials of the form $\zeta_I$ with $I\colon \bX\to \Z_{\ge 0}$ 
a function with finite support.
\end{theorem}

Alternatively, one may take as a basis for $\Int(\Z^{\bX})$ all polynomials 
$\zeta_I(\bx)$ with $\supp(I)=\bx$, a finite set of variables $\bx\subset \bX$, 
together with the constant polynomial $\zeta_{\bz}$. 
We emphasize that in the products $\zeta_{I}(\bx)$, 
there is no repetition allowed among the variables comprising the set $\bx$.  
For instance, the product $\zeta_m \zeta_n$ is not part of the 
aforementioned $\Z$-basis; rather, it may be expressed 
as a linear combination of the binomials $\zeta_m,\dots,$ $\zeta_{m+n}$. 
On the other hand, if $I$ and $J$ have disjoint supports, we have that 
$\zeta_{I}\cdot \zeta_{J}=\zeta_{I+J}$, and this polynomial is again 
part of the aforementioned basis for $\Int(\Z^{\bX})$.

\begin{corollary}[\cite{Elliott-2006, Porter-Suciu-2021}]
\label{cor:extend-map}
Let $\bX$ be a set, let $A$ be a binomial ring, and let 
$\phi \colon \bX \to A$ be a map of sets. There is then a unique 
extension of $\phi$ to a map $\tilde\phi \colon \Int(\Z^{\bX}) \to A$ 
of binomial rings.
\end{corollary}

\begin{corollary}[\cite{Elliott-2006, Porter-Suciu-2021}]
\label{cor:bproduct}
Let $R_1$ and $R_2$ be binomial rings. Then the tensor product 
$R_1 \otimes_{\Z} R_2$, with product 
$(a \otimes b)\cdot (c \otimes d) = ac \otimes bd$, is a binomial ring.
\end{corollary}

\subsection{$\F_p$-binomial algebras}
\label{subsec:zp-bin-alg}
Fix a prime $p$, and let $\F_p=\Z/p\Z$ be the field with $p$ elements. 
Let $A$ be a commutative $\F_p$-algebra; we will assume that the structure 
map $\F_p\to A$ which sends $1\in \F_p$ to the identity $1\in A$ is injective. 
Note that the binomial operations $\zeta_n(a) = (a)_n/n!$ with $a\in A$ 
are defined for $1 \le n \le p-1$, since $n!$ is then a unit in $\F_p$.

\begin{example}
\label{ex:cochains-zp}
Let $A = C^*(X;\F_p)$ be the cochain algebra  over $\F_p$ 
of a $\Delta$-complex $X$. For a cochain $a \in A^1$, 
we have that $(a)_p= 0$, where the product is the $\cup_1$-product on $A^1$. 
To see this, let $e$ be any $1$-simplex in $X$; then the elements
$a(e), a(e)-1, \ldots , a(e)-p+1$ are distinct elements in
$\F_p$. Since there are $p$ of these elements, one of the
elements must be $0$ and the property follows.
\end{example}

This motivates the following definition.

\begin{definition}
\label{def:zp-binomial algebra}
Let $A$ be a commutative $\F_p$-algebra. We say that $A$ is a 
\emph{$\F_p$-binomial algebra}\/ if $(a)_p=0$, for all $a \in A$.  
\end{definition}

Clearly, this condition is equivalent to $(a)_n=0$ for all integers $n \ge p$ 
and all $a \in A$. Note that in $\F_p[x]$ we have the equality
$(x)_p = x^p -x$. 
Indeed, both polynomials are monic, of degree $p$, and both have the 
same set of $p$ distinct roots, namely $0,1,\dots , p-1$. Therefore, 
a commutative $\F_p$-algebra $A$ is a 
$\F_p$-binomial algebra if and only if $a^p=a$, for all $a \in A$; that is, 
$A$ is a {\em $p$-Boolean algebra}, in the sense of Kriz \cite{Kriz}.

The next step is to derive properties of binomials in a 
$\F_p$-binomial algebra analogous to those for a binomial
ring over $\Z$. We start by defining the analog of $\Int(\Z^{\bX})$.

Given a set $\bX$, we will denote by $\Int(\F_p^{\bX})$ the quotient of
the free binomial algebra $\Int(\Z^{\bX})$ by the ideal generated 
by the elements $\z_n(x)$ for $x \in \bX$ and $n \ge p$, tensored with
$\F_p$. The next result shows that, modulo the constant terms,
$\Int(\F_p^{\bX})$ has $\F_p$-basis given by products of the elements
$\z_i(x)$ for $0< i <p$ and $x \in \bX$. Recall from \eqref{eq:zeta-x} 
that, for a finitely supported function $I\colon \bX\to \Z_{\ge 0}$, we write 
$\zeta_I=\prod_{x\in \bX}\zeta_{I(x)}$.

\begin{lemma}[\cite{Porter-Suciu-2021}]
\label{lem:basis-p} 
The ring $\Int(\F_p^{\bX})$ is a $\F_p$-binomial algebra, with $\F_p$-basis 
given by the $\F_p$-valued polynomials $\zeta_I(\bx)$ with 
$I\colon \bX\to \{0,\dots, p-1\}$.
\end{lemma}

\begin{theorem}[\cite{Porter-Suciu-2021}]
\label{thm:universal-p}
Let $A$ be a $\F_p$-binomial algebra. There is then a  
bijection between maps of    
$\F_p$-binomial algebras from $\Int(\F_p^{\bX})$ to
$A$ and set maps from $\bX$ to $A$.
\end{theorem}
\begin{lemma}
\label{lem:tpzp}
Let $A$ and $B$ be $\F_p$-binomial algebras. Then the tensor product 
$A \otimes_{\F_p} B$, with product $(a \otimes b)\cdot (a^\prime \otimes b^\prime)
= aa^\prime \otimes bb^\prime$, is a $\F_p$-binomial algebra.
\end{lemma}

\begin{proof}
Let $a \in A$ and $b\in B$, so that $a^p=a$ and $b^p=b$. 
Then $(a\otimes b)^p=a^p\otimes b^p=a\otimes b$, and the 
claim follows.
\end{proof}

\subsection{Binomial cup-one differential graded algebras}
\label{subsec:bc1-dga}
Following our previous work \cite{Porter-Suciu-2021}, we combine the 
notions of cup-one algebras and binomial algebras 
into a single package. 
Henceforth we will assume that our ground ring $R$ is equal to either 
$\Z$ or $\F_p$ for some prime $p$.

\begin{definition}
\label{def:binomial-cup-one-algebra}
A differential graded algebra $(A,d)$ over $R=\Z$ or $\F_p$ is called a 
{\em binomial cup-one algebra}\/ if 
\begin{enumerate}[label=(\roman*), itemsep=2pt]
\item $A$ is a cup-one algebra.
\item $A^0$, with multiplication $A^0\otimes_{R} A^0\to A^0$ given 
by the cup-product, is an $R$-binomial algebra.
\item The $R$-algebra $R \oplus A^1$ from Definition \ref{def:gr-cup1}, 
part \ref{grc-1} is an $R$-binomial algebra.
\end{enumerate}
\end{definition}

A morphism of binomial cup-one $R$-algebras is a map $\varphi\colon A\to B$ 
between two such objects which is a map of cup-one $R$-algebras that satisfies 
$\varphi(\zeta_n(a)) = \zeta_n(\varphi(a))$, for all $n\ge 1$ and all $a\in  A^1$.

\begin{proposition}
\label{prop:cup1bin-tensor}
Let $(A,d_A)$ and $(B,d_B)$ be binomial cup-one algebras over  
$R=\Z$ or $\F_p$. Then the tensor product $(A \otimes_R B,d_{A\otimes B})$ 
of the underlying dgas is again a binomial cup-one algebra.
\end{proposition}

\begin{proof}
Recall that $(A\otimes_R B)^1=
(A^1 \otimes_R B^0) \oplus (A^0 \otimes_R B^1)$. 
The claim follows at once from Proposition \ref{prop:cup1alg-tensor}, 
Corollary \ref{cor:bproduct}, and Lemma \ref{lem:tpzp}.
\end{proof}

\subsection{Binomial operations on cochains}
\label{subsec:binomials}
In this section, we show that cochain complexes with coefficients in a
binomial algebra are binomial cup-one algebras and give examples. The 
next result builds on work from \cite{Porter-Suciu-2021}.

\begin{theorem}
\label{thm:cochain-bincup1d}
For any non-empty $\Delta$-complex $X$, the cochain 
algebra $C^{*} (X;R)$, where $R=\Z$ or $\F_p$, is a binomial $\cup_1$-dga.
\end{theorem}

\begin{proof}
First assume $R=\Z$. 
The cochain algebra $C^\ast(X;\Z)$ is then a cup-one dga by Theorem
\ref{thm:cochain-cup1}, and the claim follows from Lemma \ref{lem:br}.

Now assume $R=\F_p$. We define maps 
$\zeta^X_n\colon C^{1}(X;\F_p)\to C^{1}(X;\F_p)$  for 
$1 \le n \le p-1$, by setting 
$\zeta^X_n(f)(e) \coloneqq (f(e))_n/{n!}$
for each $1$-cochain $f \in C^1(X;\F_p)= \Hom(C_1(X;\F_p),\F_p)$ 
and each $1$-simplex $e$ in $X$. As noted in 
Example \ref{ex:cochains-zp}, we have that $(f)_p=0$. 
With this structure, it is readily verified that 
$C^\ast(X;\F_p)$ is a $\F_p$-binomial $\cup_1$-dga.
\end{proof}

This theorem together with Proposition \ref{prop:cup1bin-tensor} yield 
the following corollary.

\begin{corollary}
\label{cor:cup1-ci-tensor}
Let $A$ be a binomial $\cup_1$-dga over $R=\Z$ or $\F_p$. Then the tensor product 
$A \otimes_{R} C^*(I;R)$ is again a binomial $\cup_1$-dga.
\end{corollary}

It is readily seen that the $\zeta$-maps enjoy the following naturality 
property: If $h\colon X \to Y$ is a map of $\Delta$-complexes, 
then each $\zeta_n$ commutes with the pullback of cochains, 
that is, $h^*\circ \zeta^Y_n=\zeta^X_n\circ h^*\colon C^*(Y;R)\to C^*(X;R)$.

Note that in the case $R = \Z$, the evaluation $\zeta_n(f)(e)$ is simply the 
binomial coefficient $\binom{f(e)}{n}$, for all $f\in \Hom(C_1(X;R),R)$ and 
all $e\in C_1(X;R)$.

\begin{example}
\label{ex:cochain-I-bis}
For the cochain algebra $C=C^{\ast}(I;\Z)$ from Example \ref{ex:interval}, 
the $\zeta_n$-maps are given by $\z_n(k u) = \binom{k}{n} u$. In particular,
$\z_n(u) = 0$ for $n \ge 2$.
\end{example}

\begin{example}
\label{ex:cochains-bg-bis}
Let $G$ be a group, and let $C^*(B(G);R)$ be the cochain algebra of 
the bar construction on $G$, as described in Example \ref{ex:chain-bar}, 
with coefficient ring $R$ a binomial algebra.
Then the $\zeta_n$ maps on $C^1(B(G);R)$ are given by
$\zeta_n(f)([g])  = \binom{f(g)}{n}$.
\end{example}

\section{Free binomial $\cup_1$-differential graded algebras}
\label{sect:free-bin-dga}

In this section, we define $\T(\bX)$, the free binomial
graded cup-one algebra over the rings $R=\Z$ and
$R=\F_p$ generated by a set $X$. 
Given a map $d\colon \T(\bX) \to \T(\bX)$ that satisfies 
the $\cupd$ formula and the Leibniz rule, we show 
that $d^2\equiv 0$ if and only if $d^2(x)=0$
for all $x \in \bX$. In particular, setting $d_{\bz}(x)=0$ 
for all $x\in \bX$ yields the dga $(\T(\bX),d_{\bz})$, which 
we call the free binomial $\cup_1$-dga generated by $\bX$.

\subsection{The free binomial cup-one graded algebra}
\label{subsec:free-bin-alg}
Let $R= \Z$ or $\F_p$. 
Given a set $\bX$, we let $\fm_{\bX,R}$ denote the $R$-submodule 
of $\Int(R^{\bX})$ consisting of polynomials with zero constant term.%

\begin{definition}[\cite{Porter-Suciu-2021}]
\label{def:free-binomial-dga}
The {\em free binomial $\cup_1$-graded algebra}\/ 
over $R$ on a set $X$, denoted $\T=\T_R(\bX)$, 
is the tensor algebra on the free $R$-module $\fm_{\bX,R}$; 
that is,
\begin{equation}
\label{eq:free-binomial-dga}
\T^{*}_R(\bX)=T^{*}(\fm_{\bX,R})  .
\end{equation} 
\end{definition}

By construction, $\T^0_R(\bX)=R$ and $\T^1_R(\bX)=\fm_{\bX,R}$,
and so $\T^{\le 1}_R(\bX)=\T^0\oplus \T^1$ is isomorphic to the free 
binomial algebra $\Int(R^{\bX})$. By Theorem \ref{thm:basis} and 
Lemma \ref{lem:basis-p}, respectively, $\T^1$ is a free $R$-module, 
with basis consisting of all $R$-valued polynomials of the form 
$\zeta_I$, where $I\colon \bX\to \Z_{\ge 0}$ has finite, non-empty 
support when $R=\Z$, and $I\colon \bX\to \{0,1,\dots, p-1\}$ when 
$R=\F_p$, excluding the constant-$0$ function. 
Furthermore, the $R$-module $\T^1$ comes 
endowed with a cup-one product map, 
$\cup_1\colon \T^1\otimes \T^1\to \T^1$, 
given by $a\cup_1 b=ab$. 
By analogy with the classical Hirsch formula for cochain algebras, we 
use this cup-one product to define a linear map $\T^2 \otimes \T^1\to \T^2$ by 
\begin{equation}
\label{eq:hirsch-b}
(a \otimes b)\otimes c \mapsto ac \otimes b 
+ a \otimes bc \, . 
\end{equation}

Recall that the $\cupd$ formula \eqref{eq:c1d} involves 
an operation (denoted by $\circ$) between degree $2$ elements. 
For this to work, we include the linear map 
$\circ \colon \T^2 \otimes \T^2 \to \T^2$ defined 
on basis elements by
\begin{equation}
\label{eq:circ-2}
(a_1 \otimes a_2) \circ (b_1 \otimes b_2)=
(a_1 b_1) \otimes (a_2 b_2)\, .
\end{equation}
With this structure, $\T(\bX)$ is a graded $R$-algebra 
with cup-one products, in the sense of Definition \ref{def:gr-cup1}.

The assignment $X\leadsto \T_R(\bX)$ is functorial: 
a map of sets, $h\colon \bX\to \bX'$, extends to a map
of polynomials from $\Int(R^{\bX})$ to $\Int(R^{\bX'})$ that
restricts to an $R$-linear map $\fm_{\bX,R} \to \fm_{\bX',R}$
which then extends to a map between tensor 
algebras, $\T(h)\colon \T_R(\bX)\to \T_R(\bX')$. 
Clearly, $\T(h)$ is a morphism of graded algebras that preserves 
$\cup_1$-products; moreover, $\T(h\circ g)=\T(h)\circ \T(g)$.

In the following, the ring $R$ will be equal to either $\Z$ or $\F_p$. 
A graded $R$-algebra $A$ with cup-one products such that the 
augmented algebra $R \oplus A^1$ is a binomial algebra is called a 
\emph{binomial graded $R$-algebra with cup-one products}.  In this 
category, the free binomial graded $R$-algebra $\T_R(\bX)$ enjoys the 
following universality property.

\begin{lemma}
\label{lem:TtoA}
Let $A$ be a binomial graded $R$-algebra with cup-one products, 
let $\bX$ be a set, and let $\phi \colon \bX \to A^1$ be a map of sets. 
There is then a unique extension of $\phi$ to a map $f \colon \T_R(\bX) \to A$ 
of binomial graded $R$-algebras with cup-one products.
\end{lemma}

\begin{proof}
From \cite[Lemmas 7.4 and 8.13]{Porter-Suciu-2021},
it follows that there is a unique extension of $\phi$
to a degree-preserving map $f^{>0}\colon \T_R^{>0}(\bX) \to A^{>0}$ 
which commutes with cup products, cup-one products, and the 
$\zeta$ maps. Let $\sigma_{\T} \colon R \to \T_R^0(\bX)$
and $\sigma_A \colon R \to A^0$ be the structure
maps for $\T$ and $A$; respectively, and define
$f^0\colon \T_R^0(\bX) \to A^0$ to be the composition
$\sigma_A \circ \sigma_{\T}^{-1}$. Then the resulting
map $f\colon \T_R(\bX) \to A$ is a morphism of
binomial graded $R$-algebras and the proof is complete.
\end{proof}

\subsection{Maps from free binomial cup-one dgas}
\label{subsec:free-cup1-dga}
Consider now a differential $d\colon \T_R(\bX) \to \T_R(\bX)$ making $(\T_R(\bX),d)$ 
into a binomial cup-one dga, and let $(A,d_A)$ be an arbitrary binomial 
cup-one dga over $R$. The next lemma gives a handy criterion for deciding whether 
a map $f \colon \T_R(\bX) \to A$ between the underlying binomial graded algebras 
with cup-one products is a morphism of binomial cup-one dgas.

\begin{lemma}
\label{lem:map-bcd}
A map $f \colon (\T_R(\bX),d) \to (A,d_A)$ of binomial graded $R$-algebras 
with cup-one products commutes with the differentials
if and only if $d_A f(x)=f(dx)$ for all $x \in \bX$.
\end{lemma}

\begin{proof} 
Recall from Section \ref{subsec:free-bin-alg} that $\T^1_R(\bX)$ is the free 
$R$-module with basis consisting of all $R$-valued polynomials of the 
form $\zeta_I$, where $I\colon \bX\to \Z_{\ge 0}$ has finite, non-empty 
support when $R=\Z$ and $I\colon \bX\to \{0,\dots, p-1\}$ has non-empty 
support when $R=\F_p$. The claim follows from formulas \eqref{eq:circ-op} and 
\eqref{eq:c1d}; formula \eqref{eq:zeta-binomial} expressing 
$\zeta_{n+1}(x)$ in terms of $\zeta_n(x)$; and induction on $n$.
\end{proof}

Under a connectivity assumption on $H^*(A)$, we may improve on the conclusion 
of Lemma \ref{lem:TtoA}, as follows.

\begin{lemma}
\label{lem:TtoA-diff}
Let $(A,d_A)$ be a binomial  cup-one $R$-dga with $H^0(A)\cong R$, 
let $(\T_R(\bX),d)$ be a free binomial cup-one dga, and let 
$\phi \colon \bX \to A^1$ be a map of sets. 
There is then a unique extension of $\phi$ to a map $f \colon \T_R(\bX) \to A$ of
binomial graded $R$-algebras with cup-one products such that 
$H^0(f)\colon H^0( \T_R(\bX) )\to H^0(A)$ is an isomorphism. 
\end{lemma}

\begin{proof}
Let  $f \colon \T_R(\bX) \to A$ be the extension of $\phi$ constructed in 
Lemma \ref{lem:TtoA} and let $\varepsilon$ denote the 
isomorphism from $H^0(A)$ to $R$. 
Since $H^0(A) = \ker d_A \colon A^0 \to A^1$, it follows
that the image of the structure map $\sigma_A$
is contained in $H^0(A) \subseteq A^0$. Thus, 
the composition $\varepsilon \circ \sigma_A$
is an isomorphism of rings from $R$ to $R$, and so
equals $\id_R$. It then follows that
$f^0$ is the unique $R$-linear map from $\T^0(\bX)$ to
$A^0$ that commutes with the structure maps;
that is, $f^0 \circ \sigma_{\T} = \sigma_A$, and the
proof is complete.
\end{proof}

\subsection{Differentials on $\T_R(\bX)$}
\label{subsec:diff-tx}
As before, let  $\T=\T_R(\bX)$ be a free binomial $\cup_1$-graded 
$R$-algebra on a set $X$. In this section we show that if a map 
$d \colon \T\to \T$ satisfies the $\cupd$ formula
and the Leibniz rule, then $d^2(u) = 0$ for
all $u \in \T$ if and only if $d^2(x) = 0$ for all $x \in X$. 
For that, we define additional $\cup_1$ 
and $\circ$ maps in $\T$, as follows.
First, we define a linear map $\cup_1 \colon \T^3 \ot \T^1 \to \T^3$ by
\begin{equation}
\label{eq:cup-1-3-1}
(u_1 \cup u_2 \cup u_3) \cup_1 v
		 = (u_1 \cup_1 v) \cup u_2 \cup u_3
		+ u_1 \cup (u_2 \cup_1 v) \cup u_3
		 + u_1 \cup u_2 \cup (u_3 \cup_1 v) 
\end{equation}
and a map $\cup_1 \colon \T^2 \ot \T^2 \to \T^3$ by
\begin{equation}
\label{eq:cup-1-2-2}
\begin{split}
(a_1\cup a_2)\cup_1 (b_1 \cup b_2)
		& =
			{-a_1 \cup (b_1 \cup_1 a_2) \cup b_2}%
					{-a_1 \cup b_1 \cup (b_2 \cup_1 a_2)}%
					\\
		& \quad	
			+ \sum_j a_1 \cup (a_{2,1,j} \cup_1 b_1) 
			\cup (a_{2,2,j} \cup_1 b_2)\\		
	& \quad   
	 + (b_1 \cup_1 a_1) \cup b_2 \cup a_2
				+ {b_1 \cup (b_2 \cup_1 a_1) \cup a_2}
						\\	
			& \quad	 - \sum_i (a_{1,1,i}\cup_1 b_1) 
					\cup (a_{1,2,i} \cup_1 b_2)\cup a_2 ,
\end{split}									
\end{equation}
where $d a_1 = \sum_i a_{1,1,i} \cup a_{1,2,i}$ and
$da_2 = \sum_j a_{2,1,j} \cup a_{2,2,j}$. 
Next, we define a map $\circ \colon \T^2 \ot \T^3 \to \T^3$ by
\begin{equation}
\label{eq:circ-2-3}
\begin{split}
(a_1 \cup a_2) \circ (v_1 \cup v_2 \cup v_3)
		& = (a_1 \cup_1 v_1) \cup (a_2 \cup_1 v_2) \cup v_3\\
		& \quad + (a_1 \cup_1 v_1) \cup v_2 
									\cup (a_2 \cup_1 v_3)\\
		& \quad  + v_1 \cup (a_1 \cup_1 v_2) 
										\cup	(a_2 \cup_1 v_3)\\		
		& \quad - \sum_i (a_{1,1,i} \cup_1 v_1)
			\cup (a_{1,2,i} \cup_1 v_2) 	\cup (a_2 \cup_1 v_3),						
\end{split}
\end{equation}
where $da_1 = \sum_i a_{1,1,i} \cup a_{1,2,i}$. Finally, we define a map 
$\circ \colon \T^3 \otimes \T^2 \to \T^3$ by 
\begin{equation}
\label{eq:circ-3-2}
\begin{split}
(u_1 \cup u_2 \cup u_3) \circ (b_1 \cup b_2)
		& = u_1 \cup (u_2 \cup_1 b_1) \cup (u_3 \cup_1 b_2)\\
		& \quad + (u_1 \cup_1 b_1) \cup (u_2 \cup _1 b_2)
								\cup u_3\\
		&\quad 	+ (u_1 \cup_1 b_1) \cup u_2
									\cup (u_3 \cup_1 b_2)\\
		& \quad - \sum_k (u_1 \cup_1 b_1) \cup(u_2 \cup_1 b_{2,1,k})
					\cup (u_3 \cup_1 	b_{2,2,k}) , 									
\end{split}
\end{equation}
where $db_2 = \sum_{k} b_{2,1,k} \cup b_{2,2,k}$.

The proof of the following two equations is a 
straightforward, though computationally intensive, 
verification using the definitions of the $\cupd$ 
formula along with the $\cup_1$ and $\circ$
maps in $\T$. 
\begin{gather}
\begin{aligned}
\label{eq:da1b}
d(da \cup_1 b) &= da \cup b - b \cup da + da \cup_1 db 
						+ d^2 a \cup_1 b ,\\
d(da \circ db) &= da \cup_1 db + db \cup_1 da 
					+ d^2 a \circ db + da \circ d^2 b.
\end{aligned}
\end{gather}
Note that these equations are analogous to
equation \eqref{eq:cup-i-steenrod} for $i=1,2$ and 
$\abs{a} + \abs{b} -i \le 3$, with the $\cup_2$ map replaced
by the $\circ$ map.

\begin{lemma}
\label{lem:sp}If $d \colon \T \to \T$ is a degree one map
satisfying the $\cupd$ formula and the Leibniz rule, and if 
$a,b$ are elements in $\T^1$ with
$d^2a = d^2 b =0$, then $d^2 (a \cup_1 b) = 0$.
\end{lemma}

\begin{proof}
By the $\cupd$ formula \eqref{eq:c1d}, we have that
\[
d(a \cup_1 b) = - a \cup b - b \cup a 
			+ da \cup_1 b + db \cup_1 a - da \circ db.
\]
From equations \eqref{eq:da1b}, it follows that 
\[
d^2(a \cup_1 b)=d^2 a \cup_1 b + d^2 b \cup_1 a - d^2 a \circ db - da \circ d^2b ,
\]
and this proves the claim, since $d^2 a = d^2 b = 0$.
\end{proof}

\begin{theorem}
\label{thm:dsquared}
Let $d \colon \T_R(\bX) \to \T_R(\bX)$ be a degree-one map satisfying the
$\cupd$ formula and the Leibniz rule.
Then $d^2(x) = 0$ for all $x \in \bX$ if and only if 
$d^2(u)=0$ for all $u \in \T_R(\bX)$, in which case $(\T_R(\bX), d)$
is a binomial $\cup_1$-dga.
\end{theorem}

\begin{proof}
Let $x \in \bX$.  Recall that the binomial $\zeta$-functions satisfy the formula
\begin{equation}
\label{eq:zeta-binomial}
\z_{n+1}(x) = \frac{\z_n(x)\cup_1 x - n\z_n(x)}{n+1}.
\end{equation}
where $n+1 \le p-1$ for $R=\F_p$.
Using this formula and Lemma \ref{lem:sp}, 
induction on $n$ shows 
that $d^2\z_n(x) = 0$ for all $n \ge 1$ in the case
$R=\Z$ and for all $n \le p-2$ in the case $R=\F_p$.

Making use of this fact and 
of Lemma \ref{lem:sp} once again, induction on the number of elements 
in the support of $I$ shows that $d^2(\z_I(\mathbf{x}))=0$ 
for all $I$. The claim now follows by the Leibniz rule.
\end{proof}

As a corollary, we recover a result from \cite{Porter-Suciu-2021}, which gives 
the free binomial graded algebra $\T_R(\bX)$ the structure of a binomial 
$\cup_1$-dga structure, with differential $d_{\bz}$ vanishing on all the 
generators $x\in X$. 

\begin{corollary}[\cite{Porter-Suciu-2021}]
\label{cor:cup1-t}
For any set $X$, the algebra $\T_R(\bX)$ is a binomial $\cup_1$-dga, 
with differential $d_{\bz}$ satisfying $d_{\bz}(x) = 0$ for all $x \in \bX$.
\end{corollary}

\section{Differentials defined by admissible maps}
\label{sect:admissible}

In this section, we define embeddings of the free binomial $R$-algebra 
$\T_R(\bX)$ (or its truncation in degree $2$) into a suitable cochain algebra.  
Using these embeddings, we show in Theorem \ref{thm:extend-diff} that there 
is a bijection between degree one linear maps $d\colon \T_R(\bX)\to \T_R(\bX)$ 
that satisfy the $\cupd$ formula and the Leibniz rule and maps of sets
$\tau \colon \bX \to \T_R^2(\bX)$. Theorem \ref{thm:assoc-2} gives a sufficient 
condition on $\tau$, called admissibility, for $d^2$ to be the zero map. 

\subsection{Embedding $\T_R^{\le 2}(\bX)$ into a cochain algebra}
\label{subsec:phi-map}
Recall that the ring $R$ equals $\Z$ or $\F_p$ with $p$ a prime.
Given a set $\bX$, we let $M(\bX,R)$ be the set of all functions 
$\ba\colon \bX\to R$. This is an abelian group under pointwise 
addition, with neutral element the zero function, denoted $\bz$. 
Furthermore, to every set map $f\colon \bX\to \bY$ we 
assign (in a functorial way) the $R$-linear map 
$f^{\vee}\colon M(\bY,R)\to M(\bX,R)$ given by 
$f^{\vee}(\bb)(x)=\bb(f(x))$, for $\bb\colon \bY\to R$ 
and $x\in \bX$. 

Now let $\mu\colon M\times M\to M$ be an arbitrary binary operation 
on $M=M(\bX,R)$. By the construction from Section \ref{subsec:construction}, 
the magma $(M,\mu)$ determines a $2$-dimensional $\Delta$-set, 
$\Delta^{(2)}(M,\mu)$. Let $C_{\mu}(\bX)=C_{\mu}(\bX;R)$ denote 
the cochain complex (over $R$) of this $\Delta$-set. Next, we 
define a degree-preserving, $R$-linear map, 
\begin{equation}
\label{eq:phi-map}
\begin{tikzcd}[column sep=20pt]
\psi=\psi_{\bX,\mu}\colon \T_R^{\le 2}(\bX) \ar[r]& C_{\mu}(\bX),
\end{tikzcd}
\end{equation}  
as follows. First define a map 
$\psi\colon \T_R^0(\bX) \to C^0_{\mu}(\bX)$ by sending 
$1\in \T_R^0(\bX)=R$ to the unit cochain $1\in C^0_{\mu}(\bX)$. 
For each polynomial $p\in \T_R^1(\bX)=\fm_{\bX}$, we set 
$\psi(p)\in C^1_{\mu}(\bX)$ equal to the $1$-cochain whose value 
on a $1$-simplex $\ba$ is $p(\ba)$.
Finally, we set $\psi(p\otimes q)\in C^2_{\mu}(\bX)$ 
equal to the $2$-cochain whose value on a $2$-simplex 
$(\ba, \ba^\prime)$ is $p(\ba)\cdot q(\ba^\prime)$. 
If the zero function $\bz$ is a two-sided identity in
$(M,\mu)$, it is readily seen that the image of 
$\psi$ is contained in the normalized cochains on 
$\Delta^{(2)}(M,\mu)$.

\begin{lemma}
\label{lem:one}
With notation as above, the map 
$\psi=\psi_{\bX,\mu}\colon \T_R^{\le 2}(\bX) \to C_{\mu}(\bX)$ is a
monomorphism that commutes with
cup products, cup-one products, and the $\circ$ map.
\end{lemma}
\begin{proof}
The proof of the lemma follows in general outline the proof of the 
first part of \cite[Theorem 7.2]{Porter-Suciu-2021}; since the context 
here is somewhat different, we provide full details. 

It follows directly from the definitions that the map $\psi$ commutes with
cup products, cup-one products, and the $\circ$ map, so
it suffices to show that $\psi$ is a monomorphism. To prove 
this, first suppose that $\psi(p) =0$, for some $p\in \fm_{\bX}$. 
Then for all functions $\ba\colon \bX\to R$, we have that $p(\ba)= 0$, 
and so $p$ is the zero polynomial. Therefore, 
$\psi\colon \T_R^1(\bX) \to C^1_{\mu}(\bX)$ is a monomorphism.
Now suppose $\psi\bigl( \sum \alpha_{I,J} \zeta_{I}
\otimes \zeta_{J}\bigr) = 0$,  
for some $\alpha_{I,J}\in R$ and each of $I, J$ 
not identically $0$. Then 
$ \sum \alpha_{I,J} \zeta_{I}(\ba)\cdot \zeta_{J}(\ba')=0$, for all functions 
$\ba, \ba'\colon \bX\to R$. Let $X'$ be another (disjoint) copy of $\bX$, 
and for each $J$, 
let $J'\colon \bX'\to \Z_{\ge 0}$ be the corresponding indexing function. 
Viewing each $\zeta_{J'}$ as a polynomial in 
$\Int(R^{\bX'})$, it follows that 
$\sum \alpha_{I,J} \zeta_{I}\cdot\zeta_{J'} \in \Int(R^{\bX\sqcup \bX'})$ 
is the zero polynomial. 
For each pair $I$ and $J$ of indexing functions, 
the functions $I$ and $J'$ have disjoint supports; hence, $\z_{I} \cdot \z_{J'}=\zeta_{K}$, 
where $\left.K\right|_{\bX}=I$ and $\left.K\right|_{\bX'}=J'$. 
Since these polynomials are elements in a basis for 
$\Int(R^{\bX\sqcup \bX'})$, 
it follows that each $\alpha_{I,J}$ is equal to $0$, thus showing that 
$\psi \colon \T_R^2(\bX) \to C^2_{\mu}(\bX)$ is a monomorphism.
\end{proof}

The map $\psi=\psi_{\bX,\mu}$ constructed above enjoys a naturality property 
that we now proceed to describe. Let $h\colon \bX\to \bX'$ be a map of sets. 
By the discussion in Section \ref{subsec:free-bin-alg}, we have an 
induced morphism, $\T(h)\colon \T_R^{\le 2}(\bX)\to \T_R^{\le 2}(\bX')$.  
Now let $\mu'$ be a binary operation on $M'=M(\bX',R)$. 
Composition with $h$ defines a map $h^*\colon M'\to M$. 
Suppose this map is a morphism of magmas, 
that is, $\mu'(\ba\circ h,\ba'\circ h)=h(\mu(\ba,\ba'))$ for all 
$\ba, \ba'\colon \bX\to R$. Then, as noted in Section \ref{subsec:construction}, 
$h^*$ yields a simplicial map between the respective $\Delta$-complexes, 
$\Delta(h^*)\colon \Delta^{(2)}(M',\mu') \to\Delta^{(2)}(M,\mu)$, which in 
turn induces a morphism between the corresponding cochain algebras, 
$\Delta(h^*)^*\colon C_{\mu}(\bX;R)\to C_{\mu'}(\bX';R)$. The next lemma 
now follows straight from the definitions.

\begin{lemma}
\label{lem:nat-psi}
Let $h\colon \bX\to \bX'$ be a set map and suppose 
$h^*\colon (M',\mu')\to (M,\mu)$ is a magma map. Then 
the following diagram commutes
\begin{equation}
\label{eq:psi-nat}
\begin{tikzcd}
\T^{\le 2}_R(\bX) \ar[r, "\psi_{\bX,\mu}"] \ar[d, "\T(h)"] 
& C_{\mu}(\bX;R)\phantom{.} \ar[d, "\Delta(h^*)^*"]
\\
\T^{\le 2}_R(\bX') \ar[r, "\psi_{\bX',\mu'}"]& C_{\mu}(\bX';R) .
\end{tikzcd}
\end{equation} 
\end{lemma}

In particular, if $h\colon \bX\to \bX'$ is injective, $\mu'\colon M'\times M'\to M'$ 
is a binary operation on $M'$, and $\mu$ is the restriction of $\mu'$ to 
$M\times M$, then clearly $h^*\colon (M',\mu')\surj(M,\mu)$ is a magma map, 
and thus $\psi_{X',\mu'} \circ \T(h) = \Delta(h^*)^*\circ \psi_{\bX,\mu}$.

\subsection{From $2$-tensors to simplicial complexes}
\label{subsec:tau-delta}
We now refine the above construction, in a more specialized setting. 
Consider a set map $\tau \colon \bX\to \T_R^2(\bX)$.  
Recall from \eqref{eq:free-binomial-dga} that $\T_R(\bX)$ is the 
tensor algebra on the maximal ideal $\fm_{\bX,R} \subset  \Int(R^{\bX})$. 
Thus, for each $x\in \bX$, the $2$-tensor $\tau(x)\in \T_R^2(\bX)$ 
may be written as 
\begin{equation}
\label{eq:delta-x}
\tau(x) = \sum_{i=1}^{s_x} p_{x,i} \otimes q_{x,i} ,
\end{equation}
for some polynomials $p_{x,i},q_{x,i} \in \Int(R^{\bX})$ with $p_{x,i}(\bz)=q_{x,i}(\bz)=0$.

As before, let $M=M(\bX,R)$ be the set of all functions 
$\ba\colon \bX\to R$, with $R$-module structure given by 
pointwise addition. We define an operation, 
$\mu_\tau\colon M \times M\to M$, by setting 
\begin{equation}
\label{eq:mu-tau}
\mu_\tau (\ba,\ba^\prime) = \ba + \ba^\prime - \sum_{i=1}^{s_x} p_{x,i}(\ba)\cdot q_{x,i}(\ba').
\end{equation}

The pair $M_{\tau} \coloneqq (M,\mu_{\tau})$ is a unital magma, with unit  
the zero function $\bz$; indeed, equation \eqref{eq:mu-tau} 
implies that $\mu_\tau (\ba,\bz)=\mu_\tau (\bz,\ba)=\ba$, for all $\ba$. 
In the particular case when $\tau$ itself is the zero function 
(that is, $\tau(x)=0$ in $\T_R^2(\bX)$ for all $x\in \bX$), the 
corresponding magma is just the aforementioned 
abelian group $M$. In general, though, the operation $\mu_\tau$ 
is not associative, and so $M_{\tau}$ need not be a (unital) 
semigroup (also known as monoid).

We denote by $\Delta^{(2)}(M_{\tau})$ the $2$-dimensional $\Delta$-set 
associated to the magma $M_{\tau}$ by the 
constructions from Sections \ref{subsec:construction} and \ref{subsec:phi-map}, 
and we let $C_{\tau}(\bX)\coloneqq (C^{\bullet}( \Delta^{(2)}(M_{\tau})), 
d_{\Delta})$ denote the simplicial cochain algebra associated to  
the $\Delta$-set $\Delta^{(2)}(M_{\tau})$.

\subsection{The endomorphism $d_{\tau}$ of $\T_R(\bX)$}
\label{subsect:d-tau}

Our next objective is to define a degree-$1$ endomorphism 
$d_{\tau}\colon \T_R(\bX) \to \T_R(\bX)$ of the free binomial algebra 
on $\bX$ that extends the map $\tau$ and satisfies some 
desirable properties. We achieve this by embedding $\T_R^{\le 2}(\bX)$ 
into the cochain algebra $C_{\tau}(\bX)$ defined above.

\begin{theorem}
\label{thm:extend-diff}
Given a map of sets, $\tau \colon \bX\to \T_R^2(\bX)$, there is 
a unique degree-$1$ linear map, $d_{\tau}\colon \T_R(\bX) \to \T_R(\bX)$, 
such that $d_{\tau}(x)=\tau(x)$ 
for all $x\in \bX$ and both the $\cupd$ formula and the
graded Leibniz rule are satisfied.
\end{theorem}

\begin{proof}
Let $\psi=\psi_{\bX,\tau} \colon \T_R^{\le 2}(\bX) \inj C_{\tau}(\bX)$
be the monomorphism from Lemma \ref{lem:one}. 
First note that from the formula for the coboundary of a $1$-cochain in 
a $\Delta$-complex, we have 
\begin{align}
d_{\Delta}\psi(x)(\ba,\ba') 
			& = \ba(x) + \ba'(x) 
				- (\ba(x) + \ba'(x) -\tau(\ba,\ba')(x)) \notag\\ \notag
			& =  \tau(\ba,\ba')(x)\\ \notag
			& = \sum_i ( \varphi(p_{x,i})(\ba) 
						\cdot \varphi(q_{x,i})(\ba')) \label{eq:d}\\
			& = \sum_i ( \varphi(p_{x,i})(\bx) 
						\cup \varphi(q_{x,i})(\bx)) (\ba,\ba')\\ \notag
			& = \sum_i \varphi(p_{i,x}(\bx)) 
							\ot \varphi(q_{x,i}(\bx)))
						(\ba,\ba')	\\ \notag
			& = (\varphi \tau(x)) (\ba,\ba').					
\end{align}
Since this holds for all pairs $(\ba,\ba')\in M\times M$, it follows that
$d_{\Delta}(\psi(x)) = \psi(\tau(x))$. 

The next step is to show that the differential 
$d_{\Delta}\colon C^{\bullet}_{\tau}(\bX)\to C^{\bullet+1}_{\tau}(\bX)$ 
leaves invariant the subgroup 
$\psi\bigl(\T_R^{\le 2}(\bX)\bigr) \subset C_{\tau}(\bX)$. 
Let $p = p(\bx)$ be a polynomial in $\T_R^1(\bX)$, where
$\bx \subset \bX$
denotes the subset of variables appearing in the monomials 
comprising $p$. Given $1$-chains $\ba$, $\ba'$, we have
\begin{equation}
\label{eq:d-del-p}
d_{\Delta}p(\ba,\ba') 
= p(\ba) + p(\ba') - p\bigl( \ba + \ba' - \tau(\ba,\ba') \bigr).
\end{equation}
By Theorem \ref{thm:basis}, we may write
\begin{equation}
\label{eq:dpxx}
d_{\Delta}p(\bx, \bx') 
		=\sum c_{I_i,J_i} \cdot \z_{I_i}(\bx) \otimes\z_{J_i}(\bx') ,
\end{equation}
where the $c_{I_i,J_i}$ are constants in $R$.
From equation \eqref{eq:d-del-p} it follows 
that  the polynomial $d_{\Delta}p(\bz,\bx') = \sum c_{0,J_i} \cdot \z_{J_i}(\bx')$ 
vanishes for all values of $\bx'$; thus all coefficients $c_{0,J_i}$ vanish. 
A similar argument shows that $c_{I_i,0} =0$ for all $i$. Therefore, 
$d_{\Delta}p$ is a sum of products of polynomials in 
$\Int(R^{\bX})$ and polynomials in $\Int(R^{\bX'})$  with 
zero constant term in each factor; that is, $d_{\Delta}p$ 
is in the image of $\T_R^{\le 2}(\bX)$ under the map $\psi$. 
This completes the proof that $\psi \bigl(\T_R^{\le 2}(\bX)\bigr)$ 
is closed under $d_{\Delta}$.

Now set $d_{\tau}\colon \T_R^{\le 2}(\bX)\to \T_R^{\le 2}(\bX)$ equal to the 
restriction of $d_{\Delta}$ to the invariant subgroup $\psi\bigl(\T_R^{\le 2}(\bX)\bigr)$.
Since the differential $d_{\Delta}$ satisfies the $\cupd$ formula \eqref{eq:c1d}, 
it follows that $d_{\tau}$ also satisfies this formula.
Finally, we extend $d_{\tau}$ to the whole free cup-one algebra 
$\T=\T_R(\bX)$ using the graded Leibniz rule. 

The final step is to show that the map $d_{\tau}\colon \T\to \T$ defined above 
is the unique degree $1$ linear map for which $d_\tau (x) = \tau(x)$ 
for all $x \in \bX$, and the $\cupd$ formula and the graded Leibniz rule
are satisfied. Let $d\colon\T^1 \to \T^2$ be any map
that satisfies the $\cupd$ formula with $d(x) = d_{\tau}(x)$
for all $x \in \bX$. It suffices to show that $d(p)=d_{\tau}(p)$
for all $p \in \T^1$.  Since both $d$ and $d_\tau$ satisfy the 
$\cupd$ formula, it follows that
\begin{equation}
\label{eq:ddt}
d(p \cup_1 q) =d_\tau(p \cup_1 q) \:
\text{ if $d(p) = d_\tau(p)$ and $d(q) = d_\tau(q)$}.
\end{equation}
Then from equation \eqref{eq:ddt} and induction on $i$
using the formula 
$\z_{i+1}(x) = (\z_{i}(x)\cup_1 x - i\z_i(x))/(i+1)$,
it follows that $d(\z_i(x)) = d_\tau(\z_i(x)$ for
all $x \in \bX$ and all $i \ge 1$ in the case $R=\Z$ 
and $1 \le i \le p-2$ in the case $R=\F_p$.

It then follows using \eqref{eq:ddt} and induction on 
the length of $\supp(I)$ that $d(\z_I(\bx)) = d_\tau(\z_I(\bx))$ for
all $I$ and $\bx$. Since the polynomials $\z_I(\bx)$ form a basis for
$\T^1$, the proof of uniqueness is complete. 

This completes the proof of the theorem.
\end{proof}

\begin{remark}
\label{rem:f-tau}
Let $\mu \colon M \times M \to M$ be an operation on the set 
$M=M(\bX,R)$ of all functions from $\bX$ to $R$ and 
let  $\psi \colon \T_R^{\le 2}(\bX) \inj C_{\mu}(\bX)$ be the 
monomorphism in Lemma \ref{lem:one}. 
It can be shown that the differential 
$d_{\Delta}\colon C_{\mu}(\bX)\to C_{\mu}(\bX)$ 
leaves invariant the subgroup 
$\psi\bigl(\T_R^{\le 2}(\bX)\bigr) \subset C_{\mu}(\bX)$ if 
and only if $\mu=\mu_{\tau}$, for some function 
$\tau \colon \bX \to \T^2(\bX)$.
\end{remark}

\subsection{The differential $d_{\tau}$ associated to an admissible map $\tau$}
\label{subsect:diff-tau}

Recall that a map $\tau \colon \bX \to \T_R^2(\bX)$ given by \eqref{eq:delta-x}
determines a binary operation, $\mu_\tau\colon M \times M\to M$, given by \eqref{eq:mu-tau}.%
The case when the corresponding unital magma, $M_{\tau}=(M, \mu_{\tau})$, 
is a monoid is particularly interesting.

\begin{definition}
\label{def:admissible}
A map of sets  $\tau \colon \bX \to\T_R^2(\bX)$ is said to be {\em admissible}\/ 
if the corresponding binary operation, $\mu_\tau\colon M \times M\to M$, 
is associative, or, equivalently, the magma $M_{\tau}$ is a monoid.
\end{definition}

\begin{theorem}
\label{thm:assoc-2}
If the map $\tau\colon \bX \to\T_R^2(\bX)$ is admissible, then
the map $\psi=\psi_{\bX,\tau} \colon \T_R(\bX) \to C^\ast(\Delta(M_\tau))$
is a monomorphism and $d^2_\tau \equiv 0$.
\end{theorem}

\begin{proof}
Let $\Delta_\tau=\Delta(M_{\tau})$ be the $\Delta$-set associated to the 
monoid $M_\tau$. Note that the $2$-skeleton of $\Delta_\tau$ 
is the previously defined $\Delta$-set $\Delta^{(2)}(M_{\tau})$.
The arguments in the proofs of Lemma \ref{lem:one} and
Theorem \ref{thm:extend-diff} generalize as follows 
to show that the map
$\psi_{\le 2} \colon\T_R^{\le 2}(\bX) \to 
C^{\le 2}(\Delta_\tau^{(2)})$
extends to a monomorphism
$\psi\colon \T_R(\bX) \to C^*(\Delta_\tau)$ with
$d_\Delta \circ \psi = \psi \circ d_\tau$.

Set $\psi\colon \T_R(\bX) \to C^*(\Delta_\tau)$ equal to the
unique map of algebras that restricts to $\psi_{\le 2}$
on elements of degree less than or equal to $2$.

The next step is to show that $\psi$ is a monomorphism.
Let $p_1(\bx) \otimes \cdots \otimes p_n(\bx)$ be a basis
element in $\T^n_R(\bX)$ and let $\bX_1, \ldots , \bX_n$ be disjoint
copies of $\bX$. Set $e \colon \T_R^n(\bX) \to 
\Int\bigl( R^{\bX_1\cup \cdots \cup \bX_n} \bigr)$ 
equal to the map that sends 
$p_1(\bx) \otimes \cdots \otimes p_n(\bx)$ to the
product of polynomials $p_1(\bx_1)\cdots p_n(\bx_n)$, where 
$p_i(\bx_i)$ denotes the polynomial $p_i(\bx)$ with
the variables $x \in \bX$ replaced by the corresponding
variables $x_i \in \bX_i$. Since $e$ is a bijection on basis
elements, it follows that $e$ is a bijection.
If $\psi \left( \sum_i p_{1,i}(\bx) \otimes \cdots 
\otimes p_{n,i}(\bx) \right)$ is the zero element in $C^n(\Delta_\tau)$, then
$e\left( \sum_i p_{1,i} (\ba_1) \otimes \cdots 
\otimes p_{n,i} (\ba_n) \right)$ is zero
for all maps $\ba_i\colon \bX_i \to R$, and hence,
is the zero polynomial. The result that $\psi$ is a 
monomorphism now follows since $e$ is a monomorphism.

Since $d_\Delta \circ \psi = \psi \circ d_\tau
\colon\T_R^1(\bX) \to C^2(\Delta_\tau)$
by Theorem \ref{thm:extend-diff}, and since $\T_R(\bX)$ is generated
by products of elements in $\T_R^1(\bX)$, it follows that
$d_\Delta \circ \psi = \psi \circ d_\tau
\colon \T_R^i(\bX) \to C^{i+1}(\Delta_\tau)$ for all $i \ge 1$.
Then since $d_\Delta^2 \equiv 0$ and $\psi$ is a
monomorphism, it follows that $d_\tau^2 \equiv 0$,
and the proof is complete.
\end{proof}
\begin{remark}
\label{rem:cup-2}
It can be shown that the homomorphism $\psi \colon 
\T_R(\bX) \to C^{\bullet}(\Delta_\tau)$ sends the $\cup_1$ and 
$\circ$ products in $\T_R(\bX)$ given in Section \ref{subsec:diff-tx}; respectively, 
to the $\cup_1$ and $\cup_2$ product maps in $C^{\bullet}(\Delta_\tau)$.
\end{remark}
\subsection{The differentials of the $\zeta$-maps}
\label{subsec:diff-zeta}

We consider now in more detail the simplest possible case of 
Theorem \ref{thm:extend-diff}; namely, the case when 
$\tau = 0$. To begin, we recall the following result, which is 
proved in \cite[Theorems 7.5 and 8.14]{Porter-Suciu-2021}.

\begin{theorem}
\label{thm:tx-a}
Let $(A,d_A)$ be a $\cup_1$-dga over $R=\Z$ or $\F_p$, let 
$\bX$ be a set, and let $f\colon \T_R(\bX) \to A$ be a morphism
of graded $R$-algebras with cup-one products. Then 
$f\colon (\T_R(\bX),d_{\bz}) \to (A,d_A)$ is a map of $\cup_1$-dgas
if and only if $d_A \circ f(x) = f \circ d_{\bz}(x)$ for all $x \in \bX$.
\end{theorem}

Lemma \ref{lem:TtoA-diff} and the above theorem have the following 
immediate corollary. 

\begin{corollary}
\label{cor:dx0}
If $(A,d_A)$ is a binomial cup-one $R$-dga with $H^0(A)=R$, 
then there is a bijection between binomial $\cup_1$-dga 
maps from $(\T_R(\bX), d_{\bz})$ to $(A, d_A)$ 
and maps of sets from $\bX$ to $Z^1(A)$.
\end{corollary}

The following theorem (which will be used in the proof 
of Proposition \ref{lem:ss}) gives an explicit formula for the differential
$d_{\bz} \colon\T_R^1(\bX) \to\T_R^2(\bX)$. This result 
recovers Theorem 6.11 from \cite{Porter-Suciu-2021}, proved 
there by other methods. Since the current proof is shorter, 
more transparent, and illustrates the strength of the techniques 
developed here, we provide full details. 

\begin{theorem}
\label{thm:cup-zeta}
Let $X$ be a set, let $\tau \colon \bX \to\T_R^2(\bX)$ be the
zero map, and let $d_{\bz}$ be the corresponding differential on $\T_R(\bX)$, 
given by $d_{\bz} (x)=0$ for all $x \in \bX$. Then
we have 
\begin{equation}
\label{eq:da-cup-2}
d_{\bz} (\z_k (x)) = - \sum_{\ell = 1}^{k-1}\z_\ell(x) \otimes \z_{k-\ell}(x).
\end{equation}
for all $x \in \bX$ and for $k \ge 1$ in the case $R=\Z$ and 
for $1 \le k \le p-1$ in the case $R=\F_p$.
More generally,
\begin{equation}
\label{eq:d-zeta-2}
d_{\bz} (\z_I(\bx)) 
= - \sum_{\substack{I_1 + I_2=I \\ I_j\ne \bz}} 
\z_{I_1}(\bx) \otimes \z_{I_2}(\bx) .
\end{equation}
where in the case $R=\F_p$ we have $k\le p-1$ and each
of the indices in $I$ is less than or equal to $p-1$.
\end{theorem}
\begin{proof}
Since $\tau$ is the zero map, the magma $M=M_\tau$ is a 
monoid (in fact, an abelian group), and hence, $\tau$ is admissible. 
By Theorem \ref{thm:assoc-2}, we have that $d_{\bz}=d_{\tau}$ 
is a differential on $\T_R(\bX)$, and hence, $\T=(\T_R(\bX), d_{\bz})$ is a 
binomial $\cup_1$-dga.

To prove equations \eqref{eq:da-cup-2} and 
\eqref{eq:d-zeta-2} recall that in a binomial algebra 
with elements $a,b$, we have
\begin{equation}
\label{eq:binzsum}
\z_k(a+b) = \sum_{i+j=k} \z_i(a)\z_j(b), 
\end{equation}
for $k \ge 1$ in the case $R=\Z$ and for $1 \le k \le p-1$
in the case $R=\F_p$.
Set $C_{\tau}(\bX)= (C^{\bullet}( \Delta_{\tau}(\bX)), d_{\Delta})$, 
and let $\psi \colon \T^{\le 2} \inj C^{\le 2}_\tau(\bX)$ be the injective 
map of binomial $\cup_1$-dgas defined in the proof of Theorem \ref{thm:extend-diff}, 
so that $\psi(\z_k(x)) = \z_k(\ba(x))$ for all  $\ba\in C^1_\tau(\bX)$ and all $k\ge 1$. 
Then, 
\begin{align*}
\psi(d_\tau \z_k(x))(\ba,\ba')&=
d_\Delta \psi (\z_k(x))(\ba,\ba')\\
		&= \ba(x) + \ba'(x) - \z_k(\ba(x) + \ba'(x)) 
		      &&\text{by \eqref{eq:d-del-p}}\\
		& = \ba(x) + \ba'(x) - \ba(x) - \ba'(x)
						-\sum_{\ell = 1}^{k-1}\z_\ell(\ba(x) 
						\cdot \z_{k-\ell}(\ba'(x))
		       &&\text{by \eqref{eq:binzsum}}\\
		& =	-\sum_{\ell = 1}^{k-1}\z_\ell(\ba(x) 
						\cdot \z_{k-\ell}(\ba'(x))\\
		& = -\sum_{\ell = 1}^{k-1}[\psi(\z_\ell(x))
						\cup \psi(\z_{k-\ell}(x))](\ba,\ba')\\		
		& = -\sum_{\ell=1}^{k-1} \psi [\z_\ell (x)
						\otimes \psi(\z_{k-\ell}(x))](\ba,\ba') .
\end{align*}
Since this equality holds for all $2$-simplices $(\ba, \ba')$ in $\Delta_{\tau}(\bX)$,
equation \eqref{eq:da-cup-2} now follows.
Equation \eqref{eq:d-zeta-2} follows by a similar argument, 
by applying equation \eqref{eq:binzsum} to products
of the form $\z_{i_1}(a_1+b_1)\z_{i_2}(a_2+b_2)
\cdots \z_{i_n}(a_n+b_n)$.
\end{proof}

\begin{corollary} 
\label{cor:dzeta-cocycle}
Let $A$ be a binomial $\cup_1$-dga over $R$. Then for $a \in Z^1(A)$
we have 
\begin{equation}
\label{eq:da-cup-2-bis}
d_{A} (\z_k (a)) = - \sum_{\ell = 1}^{k-1}\z_\ell(a) 
\otimes \z_{k-\ell}(a),
\end{equation}
for all $k \ge 1$ in the case $R=\Z$ and 
for $1 \le k \le p-1$ in the case $R=\F_p$.
\end{corollary}

\begin{proof}
By Corollary \ref{cor:dx0}, an element $a \in Z^1(A)$
corresponds to a map of binomial $\cup_1$-dgas
from $(\T_R(\{ x \}), d_{\bz})$ to $(A,d_A)$ 
which sends $x$ to $a$, and the result follows from
Theorem \ref{thm:cup-zeta}.
\end{proof}

\subsection{Homotopies between binomial cup-one dga maps}
\label{subsec:cup1-homotopy}

As we saw in Theorem \ref{thm:cup-coho}, homotopic dga maps induce the 
same homomorphism on cohomology. The next lemma provides a partial 
converse to this theorem, in the context of binomial cup-one algebras.

\begin{lemma}
\label{lem:hh}
Let $(A,d_A)$ be a binomial cup-one dga over $R=\Z$ or $\F_p$ such that 
$H^0(A)=R$ and $H^1(A)$ is a finitely generated, free $R$-module. Suppose 
$\varphi_0, \varphi_1\colon (\T_R(\bX),d_{\bz})\to (A,d_A)$ are morphisms 
of binomial $\cup_1$-dgas such that $H^1(\varphi_0)=H^1(\varphi_1)$.
Then $\varphi_0\simeq \varphi_1$.
\end{lemma}

\begin{proof}
We construct a homotopy $\Phi\colon \T_R(\bX)\to A\otimes_R C^*(I;R)$ between 
$\varphi_0$ and $\varphi_1$, as follows. For each $x\in X$, set 
\begin{equation}
\label{eq:Phi-map}
\Phi(x)=\varphi_0(x)t_0 + \varphi_1(x)t_1 - c(x) u, 
\end{equation}
where $c(x)$ is an element in $A^0$ such that 
$d_A(c(x)) = \varphi_0(x) - \varphi_1(x)$; 
such an element exists by our assumption that $H^1(\varphi_0)([x])=H^1(\varphi_1)([x])$. 
It is readily verified that $\Phi(x)$ is a $1$-cocycle in $A\otimes_R C^*(I;R)$; indeed,
\begin{equation}
\label{eq:dac-phi}
d_{ A\otimes C^*(I;R)}(\Phi(x))=
\varphi_0(x)u - \varphi_1(x)u  - (\varphi_0(x) - \varphi_1(x)) u = 0.
\end{equation}
It now follows from Corollary \ref{cor:dx0} that the set map 
$\Phi\colon X\to Z^1(A\otimes_R C^*(I;R))$ 
extends to a map of binomial $\cup_1$-dgas, 
$\Phi\colon \T_R(\bX)\to A\otimes_R C^*(I;R)$. By construction, 
this map is a homotopy between $\varphi_0$ and $\varphi_1$.
\end{proof}

\section{Hirsch extensions}
\label{sect:HirschExt}

In this section, we consider Hirsch extensions of $(\T_R(\bX),d)$, the 
free binomial graded algebra with cup-one products on a set $\bX$ 
equipped with a differential $d\colon \T_R(\bX) \to \T_R(\bX)$ making it into a 
$\cup_1$-dga.  Furthermore, for $R=\Z$ or $\F_p$, we show that the map
$\psi_{\bX,R} \colon (\T_R(\bX), d_\bz) \to C^\ast(B(R^n);R)$
induces an isomorphism of cohomology rings.

\subsection{Hirsch extensions of $\T_R(\bX)$}
\label{subsec:hirsch-tx}

The following definition is motivated by the notion 
of Hirsch extension in the context of commutative dgas 
over fields of characteristic $0$. Let $(\T_R(\bX),d)$ be as above.

\begin{definition}
\label{def:ext}
Let $\bY$ be a set.
An inclusion $i\colon (\T_R(\bX), d) \to (\T_R(\bX \cup \bY), \bar{d})$ 
of binomial $\cup_1$-dgas 
is called a \emph{Hirsch extension}\/ (over $R$) 
if $\bar{d}(y)$ is a cocycle in $(\T^2_R(\bX),d)$ for all $y \in \bY$. 
If $Y=\{y\}$ is a singleton, we call such an extension an 
{\em elementary Hirsch extension}.
\end{definition}

\begin{theorem}
\label{thm:ext-bijection}
Let $(\T_R(\bX), d)$ be a free binomial $\cup_1$-dga on a set $\bX$. 
Then,
\begin{enumerate}[itemsep=2pt]
\item \label{prt1}
For every set $\bY$, there is a bijection between Hirsch extensions 
$(\T_R(\bX), d) \to (\T_R(\bX \cup \bY), \bar{d})$ 
and maps of sets $\tau_{\bY}\colon \bY \to Z(\T^2_R(\bX))$.

\item \label{prt2}
If $\tau = d \vert_{\bX}$ is admissible, then
$\bar{\tau} = \bar{d}|_{\bX \cup \bY}$ is admissible. 
\end{enumerate}
\end{theorem}

\begin{proof}
Given a Hirsch extension 
$(\T_R(\bX), d)\inj  (\T_R(\bX \cup \bY), \bar{d})$, the restriction of 
$\bar{d}$ to $\bY$ gives a map $\tau_{\bY}=\bar{d}|_{\bY}\colon 
\bY\to Z(\T^2_R(\bX))$.
In the opposite direction, assume that the map $\tau_{\bY} \colon \bY \to Z(\T^2_R(\bX))$
is given. Set $\bar{\tau} \colon \bX \cup \bY \to \T^2_R(\bX \cup \bY)$
equal to the map given by $\bar{\tau}\vert_{\bX} = d\vert_{\bX}$ and
$\bar{\tau}\vert_{\bY} = \tau\vert_{\bY}$.
By Theorem \ref{thm:extend-diff}, the map $\bar{\tau}$ determines
an extension of $d$ to a map
$\bar{d} = d_{\bar{\tau}} \colon \T_R(\bX\cup \bY) \to \T_R(\bX\cup \bY)$
satisfying the $\cupd$ formula and the Leibniz rule with
$d_{\bar{\tau}}\vert_{\bX \cup \bY} = \bar{\tau}$.
Since $\tau_{\bY}(y)$ is a cocycle for all $y \in \bY$, it follows from
Theorem \ref{thm:dsquared} that $\bar{d}^2(u) = 0$ for all
$u \in \T_R(\bX \cup \bY)$, and the proof of claim \eqref{prt1} is complete.

Recall from Definition \ref{def:admissible} that $\tau$ is admissible 
precisely when the corresponding magma, $M_{\tau}$, is a monoid.  
It follows from the above proof that $M_{\bar\tau}$ is the extension 
of $M_{\tau}$ by the $R$-module $M(\bY,R)$ of functions
from $\bY$ to $R$ given by the normalized
cocycle $\nu\in Z^2(\Delta(M_{\tau});M(\bY,R))$, where
for $y\in \bY$, we have $\tau(y) = \sum_{i=1}^{s_y}
p_{y,i} \otimes q_{y,i}$ and $\nu(\ba,\ba^\prime)(y) = 
\sum_{i=1}^{s_y}p_{\by,i}(\ba) \cdot q_{y,i}(\ba^\prime)$.

Claim \eqref{prt2} now follows from Lemma \ref{lem:assoc}, 
part \eqref{ext:p3}, and the proof is complete.
\end{proof}

Given a map $\tau \colon \bY \to Z(\T^2_R(\bX),d)$, denote by
$[\tau]$ the map from $\bY$ to $H^2(\T_R(\bX),d)$ that sends
each element $y\in \bY$ to the cohomology class of $\tau(y)$.

\begin{definition}
\label{def:eq-hirsch}
Given maps $\tau$ and $\tau^\prime$ from a set $\bY$
to $Z(\T^2_R(\bX),d)$, the corresponding Hirsch extensions,
$(\T_R(\bX \cup \bY), \bar{d})$ and $(\T_R(\bX \cup \bY), \bar{d}^\prime)$, 
are called \textit{equivalent Hirsch extensions}\/ if $[\tau] = [\tau^\prime]$.
\end{definition}

\begin{lemma}
\label{lem:equiv-hirsch}
If $(\T_R(\bX \cup \bY), \bar{d})$ and $(\T_R(\bX \cup \bY), \bar{d}^\prime)$ 
are equivalent Hirsch extensions, then $(\T_R(\bX \cup \bY), \bar{d})$ and 
$(\T_R(\bX \cup \bY), \bar{d}^\prime)$ are isomorphic.
\end{lemma}

\begin{proof}
First recall from Lemmas \ref{lem:map-bcd} and \ref{lem:TtoA-diff} 
that a map $f\colon (\T_R(\bX),d) \to (A,d_A)$ of binomial graded 
$R$-algebras with cup-one products commutes with the differentials
if and only if $d_A f(x)=f(dx)$ for all $x \in \bX$; moreover, if $H^0(A)=R$ 
and $f$ induces an isomorphism on $H^0$, 
then $f$ is determined by its restriction to the set $\bX$.

From the definition of equivalent Hirsch extensions it follows 
that for each $y\in \bY$ there is an element 
$c_1(y) \in \T^1_R(\bX)$ with $\tau^\prime(y) = \tau(y) + dc_1(y)$. 
Define a linear map $f\colon (\T_R(\bX \cup \bY), \bar{d}) \to
(\T_R(\bX \cup \bY), \bar{d}^\prime)$ by setting
$f(u)=u$ for $u \in \T_R(\bX)$ and $f(y) = y - c_1(y)$
for $y \in \bY$. Then
\begin{align*}
\bar{d}^\prime( f(y))
		& = \bar{d}^\prime(y - c_1(y))
		 = \bar{d}^\prime(y) - dc_1(y)
		 = \tau^\prime(y) - dc_1(y)\\
		& = (\tau(y) + dc_1(y)) - dc_1(y)
		 =  \tau(y)
		 = f(\bar{d}(y)), 
\end{align*}
and so $f$ commutes with the differentials. Similarly, a linear map 
$g \colon (\T_R(\bX \cup \bY), \bar{d}^\prime) \to (\T_R(\bX \cup \bY), \bar{d})$ 
is defined by setting $g(u)= u$ for $u \in \T_R(\bX)$ and 
$g(y) = y + c_1(y)$ for $y \in \bY$.
Then $g$ commutes with the differentials, and the result
follows since $f$ and $g$ are inverses of each other.
\end{proof}

\subsection{A spectral sequence}
\label{subsec:specsec}

We now set up a cohomological spectral sequence that will prove 
useful for our purposes.

\begin{lemma}
\label{lem:ss}
Let $(\T_R(\bX), d)$ be a free binomial $\cup_1$-dga on 
a set $X$. Given an elementary Hirsch extension 
$(\T_R(\bX), d) \to (\T_R(\bX \cup \{y\}), \bar{d})$, 
there is a spectral sequence, $(E_r^{p,q},d_r)$, 
converging to $H^\ast(\T_R(\bX \cup \{y\}), \bar{d})$ with 
$d_r\colon E_r^{p,q} \to E_r^{p+r,q-r+1}$ and 
$E_2^{p,q} \cong H^p(\T_R(\bX),d) \ot H^q(\T_R(\{y\}),d_{\bz})$, 
where $d_{\bz}(y) = 0$.
\end{lemma}

\begin{proof}
Denote $\T_R(\bX \cup \{y\})$ by $\T$. A basis for the free 
$R$-module $\T^1$ is given by elements of the form
$\z_{I}(x_1, \ldots,x_\ell)\z_k (y)$, 
where $I = (i_1,  \ldots ,i_\ell)$ and 
$\z_{I}(x_1, \ldots,x_\ell)= \z_{i_1}(x_1)\cdots \z_{i_\ell}(x_\ell) 
\in \T_R(\bX)$, with the $x_j$ distinct elements in $\bX$.
If $k= 0$, then $\z_{I}(x_1, \ldots,x_\ell)\z_k (y)$ denotes 
$\z_{I}(x_1, \ldots,x_\ell)$;
and if $I = (0, \ldots , 0)$, then $\z_{I}(x_1, \ldots,x_\ell)\z_k (y)$
denotes $\z_k (y)$, where in the case $R=\F_p$, we 
have that $1 \le i_j \le p-1$.

Define a bigrading on $\T$ by setting $D^{p,q}$ equal to the
summand of $\T^{p+q}$ generated by the tensor products
$u_1 \ot \cdots \ot u_{p+q}$ of basis elements in $\T^1$ for which
exactly $p$ of the factors have $I \ne (0,\ldots , 0)$. 
We claim that the differential $d$ restricts to maps 
\begin{equation}
\label{eq:ss-1}
d\colon  D^{0,1} \to D^{2,0} \oplus D^{0,2}
\quad\text{and}\quad
d\colon D^{1,0} \to D^{2,0} \oplus D^{1,1} .
\end{equation}
The claim follows by induction 
using the $\cupd$ formula, the Hirsch identity, and the formula
$\z_{n+1}(y) = [\z_n(y)y - n\z_n(y)]/(n+1)$. The group 
$D^{0,1}$ is free abelian, with basis $\{\z_i(y)\}$, and the 
induction is on $i$ with base case $i=1$. The group $D^{1,0}$ has basis 
$\{\z_I(x_1, \ldots x_\ell) \z_i(y)\}$, where $I\ne \{0\}$, and the induction 
is on $i$ with base case $i=0$.

From equation \eqref{eq:ss-1}, it follows that
$F^{\ell}(\T) \coloneqq \bigoplus_{p \ge \ell, q\ge 0} D^{p,q}$
defines a decreasing filtration, 
$\T=F^{0} \supseteq F^{1}\supseteq F^{2} \supseteq \cdots$, 
of subcomplexes. 
A direct computation shows that in the resulting
spectral sequence the $E_2$ term is given by 
$E_{2}^{p,q} = H^p(\T_R(\bX),d)  \otimes H^q(\T_R(y), d_{\bz})$, 
where $d_{\bz}(y) = 0$, and the proof is complete.
\end{proof}

\subsection{The cohomology of $(\T_R(\bX),d_{\bz})$}
\label{subsec:htxd0}

We are now in a position to compute the cohomology algebra of the 
free binomial graded cup-one algebra $\T_R(\bX)$, endowed with the 
differential $d_{\tau}=d_{\bz}$ corresponding to the admissible function 
$\tau\colon \bX\to \T^2_R(\bX)$ given by $\tau(x)=0$ for all $x\in \bX$.
We first assume $R=\Z$, in which case we write $\T(\bX)\coloneqq \T_R(\bX)$.

\begin{proposition}
\label{prop:coho-ext}
Given a finite set $\bX$, there is a natural isomorphism
\begin{equation}
\begin{tikzcd}[column sep=20pt]
\kappa_X\colon H^*(\T(\bX))\ar[r, "\simeq"] & \bwedge^*(\bX)
\end{tikzcd}
\end{equation}
between the cohomology algebra of the dga $(\T(\bX),d_{\bz})$ 
and the exterior algebra on the free abelian group $\Z^{\bX}$.
\end{proposition}

\begin{proof}
We establish the existence of the isomorphism $\kappa_{\bX}$ is by 
induction on $k$, the size of $\bX$. For the base case $k=1$, 
write $\T=\T(\bX)$, and define two subcomplexes, $\T_0$ and $\T_1$, 
as follows.  Set $\T_0^0 = \Z$, $\T_0^1 = \Z$ with generator $x$, and
$\T_0^i = 0$ for $i \ge 2$.
Furthermore, set $\T_1^1$ equal to the submodule of $\T^1(\{x\})$
generated by the elements of the form $\z_k(x)$ with $k \ge 2$, and 
set $\T_1^j = \T^j(\{x\})$ for $j \ge 2$.
It is now readily verified that $\T=\T_0 \oplus \T_1$.

Clearly, $\T_0=\bigwedge (x)$, with zero differential; thus, 
$H^*(\T_0)=\bigwedge (x)$. 
Denote $\z_k(x)$ by $\z_k$, and define homomorphisms 
$h_\ell\colon \T_1^\ell \to \T_1^{\ell-1}$ by
\begin{equation}
\label{eq:h-ell}
h_\ell(\z_{i_1} \otimes \z_{i_2} \otimes \cdots \otimes \z_{i_\ell})=
\begin{cases}
- \z_{i_2+1}\otimes \cdots \otimes \z_{i_\ell} &\text{if $i_1=1$},\\[2pt]
0&\text{if $i_1>1$}.
\end{cases}
\end{equation}
By direct computation using equation \eqref{eq:d-zeta-2}, 
it follows that 
$d_{\bz} \circ h_\ell + h_{\ell+1}\circ d_{\bz} = \id_{\T_1}$. 
Hence, the cohomology of $\T_1$ is zero, and we 
conclude that the cohomology of $(\T(\{x\}),d_{\bz})$ is the exterior
algebra with generator $x$.

For the inductive step, we assume the result holds for 
$\bX_k = \{x_1, \ldots , x_k\}$ and show the result then
holds for $\bX_{k+1} = \{x_1, \ldots , x_k, x_{k+1}\}$.
We use the spectral sequence in Lemma \ref{lem:ss} with
$\bX = \bX_k$ and $y = x_{k+1}$; applying the base 
case with $\T=\T(x_{k+1})$, 
we have by induction that the $E_2$ term in the spectral
sequence of the form 
$E_2=\bigwedge (x_1, \ldots, x_k) \otimes \bigwedge (x_{k+1})$. 
Since $d_{\bz}x_{k+1} = 0$ in $\T$, it follows that the
spectral sequence collapses, from which we obtain 
an isomorphism of graded $\Z$-modules,
$H^*(\T(\bX),d_{\bz})\cong \bwedge (x_1, \ldots, x_{k+1})$.

For $x \in \bX$, set $[x]$ equal to the cohomology class of $x$
in $H^1(\T(\bX),d_{\bz})$; note that $d_{\bz}\z_2(x) = - x \cup x$, 
and so $[x] \cup [x] = 0$
in $H^\ast(\T(\bX))$. Moreover, for $x_1, x_2$ distinct elements
in $X$, we have $d_{\bz}(x_1\cup_1 x_2) = -x_1 \cup x_2 - x_2 \cup x_1$,
so $[x_1] \cup [x_2] = - [x_2] \cup [x_1]$. It follows that
the isomorphism $\kappa_X$ of graded $\Z$-modules between
$H^\ast(\T(\bX),d_{\bz})$ and the exterior algebra is a map
of graded algebras.

To prove the naturality of the isomorphism $\kappa_{\bX}$, let $h\colon \bX\to \bY$ 
be a map of sets, let $\bwedge(h)\colon \bwedge(\bX)\to \bwedge(\bY)$ 
be its extension to exterior algebras, and let $\T(h)\colon \T(\bX)\to \T(\bY)$ 
be its extension to free binomial graded algebras with cup-one products,  
as in Section \ref{subsec:free-bin-alg}. It is readily verified  
that $\kappa_Y\circ \T(h)= \bwedge(h)\circ \kappa_{\bX}$, 
and this completes the proof.
\end{proof}

\begin{corollary}
\label{cor:iso}
If $\bX$ is a set with $n$ elements, then the map
$\psi_{\bX} \colon (\T(\bX),d_{\bz}) \to C^\ast(B(\Z^n))$
induces an isomorphism of cohomology rings.
\end{corollary}
\begin{proof}
The proof is by induction. To prove the result in the case $n=1$,
let $\bX = \{x\}$. The morphism $\psi$ in Theorem \ref{thm:assoc-2} maps
$\T(\{ x \})$ to $C^\ast(\Delta (M(\{x \});\Z) = C^\ast (B(\Z);\Z)$.
Note that $C^1(B(\Z),\Z)$ is the free abelian group
of maps of sets from $\Z$ to $\Z$ and 
$\psi(x)$ is the identity map from $\Z$ to $\Z$.
The identity map of $\Z$ is a generator
for $H^1(\Z)$ and it now follows from Proposition  
\ref{prop:coho-ext} that the map $H^\ast(\psi)$ is an isomorphism.

For the inductive step, write $\bX^n = \{x_1,\ldots , x_n\}$ 
and $\bX = \bX^n \cup \{x_{n+1}\}$, and consider the morphism 
$\psi =\psi_{\bX}\colon (\T(\bX), d_{\bz}) \to C^\ast(B(\Z^n \oplus \Z);\Z)$
from Theorem \ref{thm:assoc-2}.
Assume by induction that the restriction of $\psi$
to a map $\T(\bX^n) \to C^\ast(B(\Z^n);\Z)$
induces an isomorphism on cohomology.
By the case $n=1$ above, the restriction of $\psi$
to the map from $(\T(\{x_{n+1}\}), d_\bz)$ to
$C^\ast(B(\Z);\Z)$ induces an isomorphism on cohomology.
Thus the map of $E_2$ terms from the spectral sequence
of the Hirsch extension $\T(\bX^n) \to \T(\bX)$ to
the spectral sequence of the central extension
$\Z \to \Z^{n+1} \to \Z^n$ is an isomorphism.
Since both spectral sequences collapse, it follows
that $H^\ast(\psi)$ is an isomorphism and the proof
is complete.
\end{proof}

The next step is to consider the case $R=\F_p$.
Recall that for $R=\F_2$, the cohomology ring
$H^\ast(B(R),R)$ is the polynomial algebra $R[x]$ 
on a generator $x$ in degree $1$, and for 
$R=\F_p$ with $p$ odd, $H^\ast(B(R);R)$ 
is the free commutative algebra $\bigwedge(x)\otimes_R R[y]$ 
with one generator $x$ in degree $1$ and one generator
$y$ in degree $2$, with the relation $x^2=0$. 

\begin{proposition}
\label{prop:coho-iso-p}
For $R=\F_p$, the map $\psi \colon (\T_R(\{x\}),d_{\bz}) \to C^\ast(B(R);R)$ 
induces an isomorphism $\psi^* \colon H^*(\T_R(\{x\}), d_\bz) \isom H^* (B(R);R)$.
\end{proposition}

\begin{proof}
Consider first the case $R=\F_2$. Note that in this case
$\T_R(\{x\})$ is equal to $\F_2[x]$, the polynomial 
algebra on a single generator $x$.
Hence, the differential $d_\bz$ is 
identically zero and we have that $H^\ast(\T_R(\{x\}),d_\bz)
= \F_2[x]$. To see that $\psi$ induces an isomorphism on
cohomology, note that the $1$-chain, $[1]$, in the chain
complex of the bar construction on $\F_2$ is a generator
of $H_1(B(\F_2);\F_2)$ and $\psi(x)([1]) = 1\in \F_2$ 
so $\psi(x)$ is a generator of $H^1(B(\F_2);\F_2)$. 
The result follows, since $\psi^\ast$ is an isomorphism in 
degree $1$ and both its source and target 
are polynomial algebras on a single generator in 
degree $1$.

Now consider the case $R=\F_p$ with $p\ge 3$.
Denoting $\z_i(x)$ by $\z_i$, the $R$-vector space 
$\T = \T_R(\{x\})$ has basis consisting of all elements 
of the form $\z_{i_1} \ot \z_{i_2} \ot \cdots \ot \z_{i_\ell}$, 
where each $1 \le i_j \le p-1$.
Set $D$ equal to the graded $R$-submodule of $\T$ 
generated by the basis elements
$\z_i $ for $2 \le i \le p-1$ and 
$\z_{i_1} \ot \z_{i_2} \ot \cdots \ot \z{i_\ell}$ 
for $(i_1, i_2) \ne (1,p-1)$. 
Clearly, $D$ is closed under the differential $d_\bz$; moreover,
the cochain homotopy used over $\Z$ restricts to $D$.
Hence, $H^\ast(D, d_\bz) = 0$, and it follows that 
 $H^\ast(\T, d_\bz) \cong H^\ast(\T/D, d_\bz)$.

Write $x$ and $\zeta_i$ for the images of those elements from $\T$ 
in $\T/D$, and note that both $x$ and $\z_1 \ot \z_{p-1}$ are 
cocycles in $\T/D$. The $\F_p$-algebra $\T/D$ is generated in 
degree $1$ by $[x]$ and in degree $2$ by $[\z_1 \ot \z_{p-1}]$. 
It follows that $H^1(\T/D;\F_p) \cong \F_p$, with generator $[x]$, and 
$H^2(\T/D;\F_p) \cong \F_p$, with generator $[\z_1 \ot \z_{p-1}]$.
Now note that $(\T/D)^{i+2}$ is isomorphic to 
$\z_1 \ot \z_{p-1} \otimes \T^{i}$ for $i \ge 1$.
The degree-$2$ map $\T^{*}\to (\T/D)^{*+2}$ given by
$\alpha \mapsto \z_1 \ot \z_{p-1}\ot \alpha$ for $\alpha \in \T^i$ 
is an isomorphism of graded $R$-algebras. 
Since $\z_1 \ot \z_{p-1}$ is a cocycle
in $\T/D$, it follows from the graded Leibniz rule that this map
commutes with the differentials, and hence induces an
isomorphism on cohomology. This gives
$H^{i}(\T) \cong H^{i}(\T/D) \cong H^{i+2}(\T)$ for $i \ge 1$, 
and it follows that $H^i(\T/D) \cong \F_p$ for $i \ge 1$. 
The generators of these groups are
$\z_1 \ot \z_{p-1} \ot\z_1 \ot \z_{p-1}\ot \cdots \ot\z_1 \ot \z_{p-1}$ 
if $i$ is even and 
$\z_1 \ot \z_{p-1} \ot\z_1 \ot \z_{p-1}\ot \cdots \ot\z_1 \ot \z_{p-1}
\ot \z_1$ if $i$ is odd. 

The next step is to show that $\psi$ induces isomorphisms 
$\psi^i \colon H^{i}(\T(\{x\}),d_\bz)\to H^{i}(B(\F_p);\F_p)$ 
in degrees $i=1$ and $2$. The $1$-chain $[1]$ is a generator of 
$H_1(B(\F_p);\F_p) = \F_p$ and $\psi(x)([1]) = 1$, so
$\psi^1$ is an isomorphism. Note that the cocycle 
$c = \sum_{i=1}^{p} \z_i(x) \cup \z_{k-i}(x)
\in \T$ projects to the cocycle $ \z_1 \ot \z_{p-1} \in 
\T/D$. Moreover, the homology class of the $2$-cycle 
$g= \sum_{i=1}^{p-1} [i|1]$ is a generator 
of $H_2(B(\F_p);\F_p)$. Since $\psi(c)(g)=1$, we 
conclude that $\psi^2$ is also an isomorphism. 

Finally, since $H^\ast(B(\F_p);\F_p)$
is generated in degrees $1$ and $2$ and 
$\psi^\ast$ is an isomorphism in those degrees,
it follows that $\psi^\ast$ is an epimorphism. Since
$H^i(\T(\{x\}),d_\bz)$ and $H^i(B(\F_p);\F_p)$ are both
isomorphic to $\F_p$ for each $i \ge 0$, it follows that
$\psi^\ast$ is an isomorphism, and the proof is complete. 
\end{proof}

The next theorem is a synthesis of the preceding results. 

\begin{theorem}
\label{thm:coho-iso-R}
If $\bX$ is a finite set with $n$ elements and if 
$R=\Z$ or $\F_p$, then the dga morphism 
$\psi_{\bX,R}\colon (\T_R(\bX), d_\bz) \to C^\ast(B(R^n);R)$
induces an isomorphism of cohomology rings.
\end{theorem}

\begin{proof}
For $R=\Z$ the result is Corollary \ref{cor:iso}.
For $R=\F_p$ and $n=1$, the result is 
Proposition \ref{prop:coho-iso-p}. For $n > 1$, the result
follows by induction on $n$, using the property
that the spectral sequence of the Hirsch extension 
$(\T_R(\bX), d_\bz) \to (\T_R(\bX \cup \{y\}), d_\bz)$
collapses with $E_2 = E_\infty$.
\end{proof}

\subsection{Colimits of Hirsch extensions}
\label{subsec:colimits-hirsch}
We now consider a type of free binomial $\cup_1$-dgas that arise 
as unions (or, colimits) of certain sequences of Hirsch extensions.  
These objects will play an important role for the rest of this paper. 

\begin{definition}
\label{def:lim-hirsch}
A free binomial $\cup_1$-dga $(\T_R(\bX),d)$ is called a 
\textit{colimit of Hirsch extensions}\/ if the following 
conditions hold.
\begin{enumerate}[itemsep=2pt]
\item $\bX = \bigcup_{i \ge 1} \bX_i$ with each set $\bX_i$ finite.
\item For $\bX^n = \bigcup_{i=1}^{n} \bX_i$ and
$n \ge 1$, the differential $d$ on $\T_R(\bX)$ restricts to a 
differential $d_n$ on $\T_R(\bX^n)$.
\item Each map 
$i_n \colon (\T_R(\bX^n),d_n) \to (\T_R(\bX^{n+1}),  d_{n+1})$
is a Hirsch extension.
\item $\bX_1 \ne \emptyset$, and $d_1(x) = 0$ for all $x \in \bX_1$.
\end{enumerate}
\end{definition}

A {\em morphism of colimits of 
Hirsch extensions}\/ is a map of binomial $\cup_1$-dgas as above, 
$f\colon (\T_R(\bX),d)\to (\T_R(\bX'),d')$, with the
property that for each $n\ge 1$, the map restricts 
to a morphism $f_n\colon \T_R(\bX^n) \to \T_R({\bX'^{n}})$. 
Note that these morphisms are compatible with the 
respective colimits; that is, the diagram below commutes 
for each $n\ge 1$. 
\begin{equation}
\label{eq:compatible}
\begin{tikzcd}[row sep=2.2pc]
\T_R(\bX^{n+1}) \ar[r, "f_{n+1}" ] 
			& \T_R(\bX'^{n+1}) \phantom{\, .}\\
\T_R(\bX^n) \ar[u, "i_{n}" ] \ar[r, "f_n" ']
			& \T_R(\bX'^n) \, . \ar[u, "i'_{n}" ']
\end{tikzcd}
\end{equation}

\subsection{The group associated to a colimit of Hirsch extensions}
\label{subsec:hirsch-group}
We now associate in a functorial way to each colimit of Hirsch extensions 
a pronilpotent group.

\begin{lemma}
\label{lem:h-group}
Let $\T=(\T_R(\bX),d)$ be a colimit of Hirsch extensions 
$\T_n=(\T_R(\bX^n),d_n)$.
\begin{enumerate}[itemsep=1pt]
\item \label{hg1}
There is a pronilpotent group $G_{\T}$ and a $\cup_1$-dga map 
$\psi_{\T} \colon \T \to C^*(B(G_{\T});R)$

\item \label{hg2}
If $\bX$ is finite, then $G_{\T}$ is a nilpotent group and 
$\psi_{\T}$ is a quasi-isomor\-phism. Moreover, if $R=\Z$, 
then $G_{\T}$ is torsion-free.

\item \label{hg2bis}
In general, the map $\psi_{\T}$  induces a quasi-iso\-morphism 
from $\T$ to $C^*_{\cont}(B(G_{\T});R)=  
\colim_{n\to \infty} C^*(B(G_{\T_n});R)$, 
the continuous cochains on $B(G_{\T})$. 

\item \label{hg3}
Every morphism of colimits of Hirsch extensions, $f\colon \T\to \T'$, 
induces (in a functorial way) a group homomorphism, 
$\tilde{f}\colon G_{\T'}\to G_{\T}$.
\end{enumerate}
\end{lemma}

\begin{proof}
\eqref{hg1}
We start by defining for each $n\ge 1$ a finitely generated 
nilpotent group, $G_n=G_{\T_n}$, corresponding to the free $\cup_1$-dgas 
$\T_n=(\T_R(\bX^n),d_n)$, as well as a $\cup_1$-dga map 
$\psi_n \colon \T_R(\bX^n) \to C^*(B(G_n);R)$ inducing an 
isomorphism on cohomology. 
This is done inductively, as follows.

First let $G_1=M(\bX_1,R)$. As noted in Section \ref{subsec:phi-map}, 
this is a free $R$-module with basis $\bX_1$; we view it now as a finitely 
generated abelian group. By Theorem \ref{thm:coho-iso-R}, there is a 
quasi-isomorphism $\psi_1\colon (\T_R(\bX_1), d_\bz) \to C^\ast(B(G_1);R)$. 
Assume now that a finitely generated nilpotent group $G_n$ has been 
constructed, together with a
$\cup_1$-dga quasi-isomorphism 
$\psi_n \colon (\T_R(\bX^n),d_n) \to C^*(B(G_n);R)$ inducing 
an isomorphism on $H^1$. By Theorem \ref{thm:ext-bijection}, the differential 
$d_{n+1}$ on $\T(\bX_{n+1})$ restricts to an admissible map 
$\tau_{n}\colon \bX_{n+1} \to Z( \T^2_R(\bX^n))$. The composition 
$\psi_n\circ \tau_n$, then, defines a cocycle in $Z^2(B(G_n);M(\bX_{n+1},R))$; let 
\begin{equation}
\label{eq:ext-gn}
\begin{tikzcd}[column sep=22pt]
0 \ar[r] & M(\bX_{n+1},R)\ar[r] & G_{n+1} \ar[r, "q_n"] &G_n \ar[r] & 1
\end{tikzcd}
\end{equation}
be the corresponding central extension. Since $G_n$ is a group, 
Lemma \ref{lem:assoc} insures that $G_{n+1}$ is also a group; 
by construction, this is again a finitely generated nilpotent group 
(torsion-free if $R=\Z$).
Since the map $\tau_n$ is admissible, Theorem \ref{thm:ext-bijection} 
insures that the map $\tau_{n+1}=d_{n+1}|_{\bX_{n+1}}$ is also 
admissible. Since $\bX_{n+1}$ is finite, the Hirsch extension 
$\T_R(\bX_n) \inj \T_R(\bX_{n+1})$ 
can be realized as a sequence of elementary Hirsch extensions.
The inductive assumption together with Lemma \ref{lem:ss}
and Theorem \ref{thm:coho-iso-R} then
show that $\psi_{n+1}$ is a quasi-isomorphism.
This completes the construction of the $G_n$ and 
the argument that $\psi_n$ is a quasi-isomorphism.
 
We now let $G_{\T}=\varprojlim G_n$ be the limit of the inverse system of groups 
$\{G_n,q_n\}_{n\ge 1}$ and  $\psi_{\T} \colon \T \to C^*(B(G_{\T});R)$ be the colimit of the 
directed system of maps $\psi_n \colon (\T_R(\bX^n),d_n) \to C^*(B(G_n);R)$. 
By construction, both $G_{\T}$ and $\psi_{\T}$ satisfy the claimed properties. 
Note that the underlying magma of $G_{\T}$ is $(M(\bX,R),\mu_{\tau})$, 
where $\tau\colon \bX \to Z(\T^2_R(\bX))$ is the colimit of the maps $\tau_n$, 
while $\psi_{\T}$ coincides with the map 
$\psi_{\bX}\colon \T(\bX)\to C^*(\Delta(M_{\tau});R)$.

\eqref{hg2}
If $\bX$ is finite, then $\bX=\bX^n$ for some $n\ge 1$, and the claimed  
properties for $G_{\T}$ and $\psi_{\T}$ follow from the above proof.

\eqref{hg2bis}
The fact that the map $\T \to C^*_{\cont}(B(G_{\T});R)$ induced by $\psi_{\T}$
is a quasi-isomorphism (when $\bX$ is not necessarily finite) 
follows from part \eqref{hg2}.

\eqref{hg3}
Let $f\colon \T\to \T'$ be a morphism of colimits Hirsch extensions, 
so that, for each $n\ge 1$, the diagram \eqref{eq:compatible} 
commutes and $d'_n\circ f_n=f_{n+1}\circ d_n$. Setting 
$\tilde{i}_n\coloneqq q_n$, we define inductively homomorphisms 
$\tilde{f}_n\colon G'_n \to G_n$ which satisfy 
$\tilde{f}_n\circ \tilde{i'}_n = \tilde{i}_n\circ \tilde{f}_{n+1}$, as follows. 
We first let $\tilde{f}_1\colon G'_1\to G_1$ be equal to the $R$-linear 
map $f_1^{\vee}\colon M(\bX'_1,R) \to M(\bX_1,R)$ from Section \ref{subsec:phi-map}.
Assuming that $\tilde{f}_n$ has been defined, we let 
$(f_{n+1}|_{\bX_{n+1}})^{\vee}\colon M(\bX'_{n+1};R)\to M(\bX_{n+1};R)$ 
be equal to the $\Hom(-,R)$-dual of the set map 
$f_{n+1}|_{\bX_{n+1}} \colon \bX_{n+1} \to \bX'_{n+1}$. 
The fact that $f_n$ and $f_{n+1}$ are compatible dga-maps implies that the 
homomorphisms $\tilde{f}_n$ and $(f_{n+1}|_{\bX_{n+1}})^{\vee}$ are compatible 
with the $k$-invariants of the extensions \eqref{eq:ext-gn}, and thus define a 
homomorphism $\tilde{f}_{n+1}\colon G'_{n+1} \to G_{n+1}$ with the claimed  
property. Passing to the limit yields a homomorphism $\tilde{f}\colon G_{\T'}\to G_{\T}$. 
Finally, if $g\colon \T'\to \T''$ is another morphism of colimits of Hirsch extensions, it is 
readily verified that $\tilde{f}\circ \tilde{g}=\widetilde{g\circ f}$.
\end{proof}

\begin{remark}
\label{rem:BK}
In \cite{BK71, BK}, Bousfield and Kan associated to every space $X$ and 
commutative ring $R$ an $R$-completion, $R_{\infty}(X)$. 
This is an $R$-space (i.e., its homology groups in positive degrees 
are $R$-modules) which comes equipped with a structure 
map, $\kappa_X\colon X\to R_{\infty}X$, with the following property: 
Given a map $f\colon X\to Y$, there is an induced map, 
$R_{\infty}f\colon R_{\infty}X\to R_{\infty}Y$, such that 
$R_{\infty}f\circ \kappa_X=\kappa_Y\circ f$; moreover, $f$ 
is an $R$-homology equivalence if and only if 
$R_{\infty}f$ is a weak homotopy equivalence. 
The Bousfield--Kan completion also works at the level of groups, sending in a 
functorial way a group $G$ to its $R$-completion, $\widehat{G}_R$. 
In particular, if $R=\Z$, then $\widehat{G}_\Z=\varprojlim_{n} G/\Gamma_n(G)$ 
is the pronilpotent completion of $G$, built from the lower central series $\{\Gamma_n(G)\}_{n\ge 1}$. 
It is shown in \cite{Porter-Suciu-group-1min} that if $\T$ is the $1$-minimal model of a 
path-connected space $X$, then the classifying space $B(G_{\T})$ is related to a 
truncation of $\Z_{\infty}(X)$, using a suitable replacement for the lower central series. 
\end{remark}

The next lemma describes another type of functoriality property of the above 
construction, related to maps between classifying spaces of $R$-pronilpotent groups.

\begin{theorem}
\label{lem:group-gives-hirsch}
Let $\T=(\T_R(\bX), d)$ be a colimit of Hirsch extensions, let $G=G_{\T}$ 
be the corresponding pronilpotent group, and assume the map 
$\psi_{\T} \colon \T_R(\bX)\to C^*(B{G};R)$ is a quasi-isomorphism. 
Let also $0 \to F \to \overline{G} \to G \to 1$ be a central extension of groups, 
with $F$ a finitely generated, free $R$-module, and  let
$\pi \colon B{\overline{G}} \to B{G}$ be the corresponding fibration. Then 
\begin{enumerate}
\item \label{gh1}
There is a Hirsch extension 
$i\colon \T \inj \overline{\T}= (\T_R(\bX \cup \bY), \bar{d})$
such that $\overline{G}=G_{\overline{T}}$. 
\item \label{gh2}
The diagram below commutes
\begin{equation}
\label{eq:group-gives-a-hirsch}
\begin{tikzcd}[column sep=30pt]
\overline{\T}  \ar[r,  "\psi_{\overline{\T}}"] & C^\ast(B{\overline{G}};R)\phantom{\, .} \\
 \T \ar[r, "\psi_{\T}"] \ar[u, "i" ']& C^\ast(B{G};R) \ar[u, "\pi^*" '] \, .
\end{tikzcd}
\end{equation}
\item \label{gh3}
The map $\psi_{\overline{\T} }$ is a quasi-isomorphism. 
\end{enumerate}
\end{theorem}

\begin{proof}
It suffices to prove the case where $F = R$.
Let $c \in \T_R^2(\bX)$ be a cocycle such that
the cohomology class of $\psi_{\bX}(c)$ in $H^2(B{G};R)$ corresponds 
to the central extension. Set  $\overline{\T}=(\T_{R}(\bX \cup \{y\}), \bar{d})$ 
equal to the Hirsch extension with $\bar{d} (y)= c$ and set $\overline{G}=G_{\overline\T}$. 
This gives the commutative diagram \eqref{eq:group-gives-a-hirsch}.

The map $\psi_{\overline{\T}}$ induces a map of spectral
sequences from the spectral sequence of the
Hirsch extension to the spectral sequence of the
fibration. The map of the terms $E_2^{p,q}$ is
\begin{equation}
\label{eq:hptr}
\begin{tikzcd}[column sep=36pt]
H^p(\T_R(\bX) \otimes H^q(\T_R(\{y\}))
\ar[r, "\, \psi_{\T} \otimes f"]
& H^p(B{G};R) \otimes H^q(B{R};R) ,
\end{tikzcd}
\end{equation}
where $H^q(\T(\{y\})$ denotes the cohomology
computed with $dy=0$ and $f$ denotes the map
on cohomology induced by the map 
$\psi_{\T(\{y\})} \colon \T_R(\{y\}) \to C^\ast(B{R};R)$. The map 
$\psi_{\T}$ is an isomorphism by assumption, while $f$
has been shown to be an isomorphism in Theorem \ref{thm:coho-iso-R}. 
Hence, the map of $E_2$ terms is an isomorphism and it follows that
$\psi_{\overline{\T}}$ is a quasi-isomorphism.
\end{proof}

\begin{corollary}
\label{cor:equiv-cat}
The assignment $\T=(\T_R(\bX),d)  \leadsto G_{\T}$ establishes an equivalence of 
categories between colimits of Hirsch extensions over $R$ and $R$-pronilpotent groups, 
with the property that $H^*(\T_R(\bX)) \cong H^*_{\cont}(G_{\T};R)$.
\end{corollary}

\begin{proof}
By part \eqref{hg1} of Lemma \ref{lem:h-group}
it follows that if $(\T_R(\bX), d)$ is a colimit
of Hirsch extensions over $R$, then the corresponding group
$G_{\T}$ is an $R$-pronilpotent group. 

Conversely, if $G=\varprojlim G_n$ is a pronilpotent group, 
then by Theorem \ref{lem:group-gives-hirsch} and induction
on $n$, it follows that $\T=(\T(\bX), d)$ is a colimit
of Hirsch extensions. Moreover, by part \eqref{hg2bis} 
of Lemma \ref{lem:h-group}, $H^*_{\cont}(G_{\T};R)\cong H^*(\T)$.
\end{proof}

\section{The existence of $1$-minimal models}
\label{sect:1-min}

In this section, we define $1$-minimal models and 
show that every binomial $\cup_1$-dga $A$ over the ring 
$R=\Z$ or $R=\F_p$ admits a $1$-minimal model, 
provided that $H^0(A)=R$ and $H^1(A)$ is a 
finitely generated, free $R$-module.

\subsection{Solving an extension problem}
\label{subsec:ext-maps}
We start by setting up an extension problem in the category of 
binomial $\cup_1$-dgas and give an if and only if criterion to solve it.

\begin{definition}
\label{def:extend-map}
Let $f\colon (\T_R(\bX),d)\to (A,d_A)$  be a morphism of binomial 
$\cup_1$-dgas, and let $i\colon \T_R(\bX)\to \T_R(\bX \cup \bY)$ be a Hirsch 
extension of $\T_R(\bX)$.  A morphism $\bar{f}\colon  \T_R(\bX \cup \bY) \to A$ 
is an \emph{extension of $f$}\/ if the following diagram commutes.
\begin{equation}
\label{eq:txya}
\begin{tikzcd}
 & \T_R(\bX \cup \bY)\phantom{ .}   \ar[dl, "\bar{f}" ']    \\
 A  & \T_R(\bX)  . \ar[u, "i" '] \ar[l, "f" ]		
\end{tikzcd}
\end{equation}
\end{definition}

\begin{theorem}
\label{def:ext-criterion}
Given a morphism of binomial $\cup_1$-dgas, 
$f\colon (\T_R(\bX),d)\to (A,d_A)$, and a Hirsch extension, 
$i\colon (\T_R(\bX),d)\to (\T_R(\bX \cup \bY),\bar{d})$, there is 
a bijection between the set of extensions $\bar f\colon  \T_R(\bX \cup \bY) \to A$ 
of $f$ and the set of functions $\sigma \colon \bY \to A^1$ such that 
$d_A(\sigma(y))= f(d y)$ for all $y\in \bY$.
\end{theorem}

\begin{proof}
First note that there is an extension 
$\bar{f}$ of $f$ if and only if $[f(dy)] = 0$ for all $y \in \bY$. 
Thus, the set of such extensions is empty if and only if the 
set  $\{\sigma \colon \bY \to A^1 :  d_A(\sigma(y))= f(d y),\, \forall y\in \bY\}$
is empty, and we are done in this case.

So assume there is an extension $\bar{f}\colon  \T_R(\bX \cup \bY) \to A$. 
The corresponding map $\sigma \colon \bY \to A^1$ is then given by 
$\sigma( y )= \bar{f}(y)$ for $y\in \bY$. The condition that 
$d_A(\sigma( y))=f(dy)$ follows from the assumption 
that $\bar{f}$ is a map of dgas.

In the opposite direction, assume $\sigma \colon \bY \to A^1$ is a map of sets
with $d_A(\sigma( y))= f(dy)$ for all $y\in \bY$. Then by Lemma \ref{lem:TtoA}, 
the function $y \mapsto \sigma(y)$ extends $f$ uniquely to a map
$\bar{f}\colon \T_R(\bX \cup \bY)\to A$ of binomial graded $R$-algebras 
with cup-one products. Since $d_A(\sigma( y)) = f(dy)$, it then follows from 
Theorem \ref{thm:tx-a} that this extension commutes with the differentials  
$d$ and $d_A$, and so $\bar{f}$ is a morphism of binomial $\cup_1$-dgas.
This completes the proof.
\end{proof}

\begin{lemma}
\label{lem:fgtf}
With notation as above, 
let $\bar{f}\colon \T_R(\bX \cup \bY ) \to A$ be an extension of $f\colon \T_R(\bX)\to A$  
with $Y = \{y\}$, where $\bar{d}$ denotes the differential
on $\T_R(\bX \cup \bY)$ and $d$ denotes its restriction to
$\T_R(\bX)$. Assume both $H^1(A)$ and $\ker(H^2(f))$ are 
finitely generated, free $R$-modules,
and that the cohomology classes of the cocycles
$dy, c_1, c_2, \ldots, c_\ell$ in $\T_R(\bX)$ 
form a basis for $\ker (H^2(f))$. Then,
\begin{enumerate}[itemsep=1.5pt]
\item 
\label{prop:1}
The inclusion $i\colon\T_R(\bX)\to \T_R(\bX \cup \{y\})$ induces
an isomorphism on $H^1$.
\item 
\label{prop:2}
The set $\{[i(c_1)], \ldots , [i(c_\ell)]\}$ is a basis for 
$\im(H^2(i))\cap \ker \big(H^2(\bar{f})\big)$.
\item 
\label{prop:3}
The kernel of $H^2(\bar{f})$ is a finitely generated, free $R$-module.
\end{enumerate}
\end{lemma}

\begin{proof}
Consider the spectral sequence from Lemma \ref{lem:ss} 
associated with the elementary Hirsch extension 
$(\T_R(\bX), d)\inj (\T_R(\bX \cup \{y\}),\bar{d})$.
We then have $E_2^{0,1} = R$ with generator $y$, and 
$\bar{d} y \in \ker(H^2(f))$. Hence, by assumption, 
$n dy\ne 0$ for all $n \in R$, $n\ne 0$. 
It follows that $E_3^{0,1} = 0$. Therefore, the induced homomorphism 
$H^1(f)\colon H^1(\T_R(\bX))\to H^1(\T_R(\bX \cup \{y\}))$ is an isomorphism, 
and the proof of claim \eqref{prop:1} is complete.

To prove claim \eqref{prop:2} in the case $R=\F_p$,
note that since $\F_p$ is a field, we can write
$H^2(\T_R(\bX))$ as a direct sum 
$\spn([dy], [c_1], \ldots, [c_\ell]) \oplus \bar{B}$, with
$H^2(f)$ restricted to $\bar{B}$ a monomorphism.
Then $\im(H^2(i)) = E_\infty^{2,0} =
\spn([c_1], \ldots, [c_\ell]) \oplus \bar{B}$, and the result
follows since $H^2(\bar{f})\circ H^2(i) = H^2(f)$.

To prove claim \eqref{prop:2} in the case
$R =\Z$, note that since $E_2^{0,2}=0$, the terms 
$E_\infty^{2,0}$ and $E_\infty^{1,1}$ give an exact sequence, 
\begin{equation}
\label{eq:split}
\begin{tikzcd}[column sep=20pt]
0 \ar[r]
			& E_\infty^{2,0} = H^2(\T_R(\bX))/[\bar{d}y]\ar[r]
 			& H^2(\T_R(\bX \cup \{y\})) \ar[r]
 			& E_\infty^{1.1} \ar[r]
 			& 0 ,
\end{tikzcd}
\end{equation}
and claim \eqref{prop:2} follows at once.

For $R=\F_p$, claim \eqref{prop:3} follows since in this
case every submodule of a finitely generated, free 
$R$-module is a finitely generated, free $R$-module.

To prove claim \eqref{prop:3} in the case $R=\Z$,
let $\{[dy], [c_1], \ldots, [c_\ell]\}$ be a basis for $\ker(H^2(f))$, 
and let $[c_2, \ldots , c_\ell]$ denote the submodule of
$H^2(\T_R(\bX \cup \{y\}))$ generated by the elements 
$[i(c_2)], \dots, [i(c_{\ell})]$.
Then since $E_\infty^{1,1}$ is finitely generated and 
torsion-free, the sequence \eqref{eq:split} is split exact and the
kernel of the map from $H^2(\T_R(\bX \cup \{y\}))/[c_2, \ldots, c_\ell]$
to $H^2(A)$ is the submodule $K$ of $E_\infty^{1,1}$ consisting
of all elements $k$ for which there is an element
$\alpha \in \im(H^2(i))/[c_2, \ldots , c_\ell]$
with $H^2(\bar{f})(k + \alpha)=0$.  It follows that 
$\ker \big(H^2(\bar{f}) \big)\cong [c_2, \ldots , c_\ell] \oplus K$. 
This establishes claim \eqref{prop:3} in the case $R=\Z$, 
and the proof is complete.
\end{proof}

\subsection{A lifting criterion}
\label{subsec:lift-crit}
The next theorem corresponds to an analogous rational homotopy 
result from \cite{FHT} (Lemma 12.4). 

\begin{theorem} 
\label{thm:1-qi-lift}
Let $(A,d_A)$ and $(A',d_{A'})$ be binomial cup-one $R$-dgas over 
$R=\Z$ or $\F_p$, let $f\colon A\to A'$ be a surjective $1$-quasi-iso\-morphism, 
and let $\varphi$ be a morphism from a colimit of Hirsch extensions, $(\T_R(\bX),d)$, 
to $(A',d_{A'})$. There is then a lift of $\varphi$ through $f$; that is, a morphism 
$\widehat{\varphi} \colon \T_R(\bX) \to A$ such that
the following diagram commutes
\[
\begin{tikzcd}
& A \phantom{.} \ar[d, "f" ] \\
\T_R(\bX) \ar[r, "\varphi" ] \ar[ur, dashed, "\widehat{\varphi}" ] & A^\prime .
\end{tikzcd}
\]
\end{theorem}

\begin{proof}
As in Definition \ref{def:lim-hirsch}, 
let $\{(\T_R(\bX^n),d_n)\}_{n\ge 1}$ be the sequence of 
binomial $\cup_1$-dgas whose colimit is $(\T_R(\bX),d)$.
Set $\varphi_n\colon \T_R(\bX^n) \to A'$ equal to the restriction of 
$\varphi$ to $\T_R(\bX^n)$. It suffices to show that for each $n\ge 1$, 
there is a lift $\widehat{\varphi}_n$ of $\varphi_n$ through $f$.

The argument proceeds by induction.
For $n=1$, we have that $d_1(x) =0$ for all $x \in \bX_1$.
Thus, by Lemmas \ref{lem:TtoA} and 
\ref{lem:map-bcd} it suffices to show that for each
$x \in X_1$ there is a cocycle $a_x$ in $A$ with 
$f(a_x) = \varphi(x)$. To do this, given $x \in \bX_1$, 
let $b_x$ be an element in $A^1$ with 
$f(b_x) = \varphi(x)$. Then $d_A( b_x)$ is a 
cocycle in $\ker(f)$. By Remark \ref{rem:coho-vanish}, 
we have that $H^2(\ker(f)) = 0$. Hence, there
is an element $c_x \in \ker(f)$ with $d_A(c_x) = d_A(b_x)$.
Then $a_x = b_x - c_x$ is a cocycle in $A$ with
$f(a_x) = \varphi(x)$. This completes the argument
for the case $n=1$.

Now assume there is a lifting $\widehat{\varphi}_n$
of $\varphi_n$ through $f$. 
In order to show that $\widehat{\varphi}_n$ extends to a
lifting $\widehat{\varphi}_{n+1}$ through $f$, 
it suffices by Lemmas \ref{lem:TtoA} and 
\ref{lem:map-bcd} to show that for each
$x \in \bX_{n+1}$ there is an element $a_x$ in $A$ with 
$f(a_x) = \varphi(x)$ and 
$\widehat{\varphi}_n(dx)= d_A(a_x)$.
Given $x \in \bX_{n+1}$, let $b_x$ be an element in $A$
with $f(b_x) = \varphi(x)$. Then 
\begin{align*}
f\left( \widehat{\varphi}_n (dx) - d_A(b_x)\right) 
			& = \varphi(dx) - f(d_A(b_x))
			 = \varphi(dx) - d_{A'}(f(b_x))\\
			& = \varphi(dx) - d_{A'}(\varphi(x))
			= 0,
\end{align*}
and so $\widehat{\varphi}_n (dx) - d_A(b_x) \in \ker(f)$.
We have that $\widehat{\varphi}_n (dx) - d_A(b_x)$ is a 
cocycle in $\ker(f)$ and $H^2(\ker(f))=0$; therefore, 
there is an element $c_x \in \ker(f)$ with
$d_A(c_x) = \widehat{\varphi}_n (dx) - d_A(b_x)$. 
Setting $a_x = b_x - c_x$, we have that  
$f(a_x) = \varphi(x)$ and $\widehat{\varphi}_n(dx)= d_A(a_x)$,
and the argument is complete.
\end{proof}

\subsection{$1$-minimal models}
\label{subsec:1-min-models}

Colimits of Hirsch extensions lead to the notion 
of $1$-minimal model, which is central to the study done in this paper.
Let $(A,d_A)$ be a binomial $\cup_1$-dga over $R=\Z$ or $\F_p$ such that 
$H^0(A)=R$ and $H^1(A)$ is a finitely generated, free $R$-module.  

\begin{definition}
\label{def:1-min-model}
A {\em $1$-minimal model}\/ for $A$ is a free binomial $\cup_1$-dga 
$\mcm=(\T_R(\bX),d)$ which arises as the colimit of a sequence of Hirsch extensions, 
$\mcm_n=(\T_R(\bX^n),d_n)$, together with morphisms $\rho_n\colon  \mcm_n\to A$ 
such that, for each $n\ge 1$, the diagram below,
\begin{equation}
\label{eq:atx-diag}
\begin{tikzcd}[row sep=2.5pc]
    & \mcm_{n+1} \phantom{\, ,}\ar [dl, "\rho_{n+1}" ']\\
A  & \mcm_{n}\, , \ar[ l, "\rho_n" ] \ar[u , "i_{n}" ']
\end{tikzcd}
\end{equation}
is a commutative diagram of binomial $\cup_1$-dgas and 
the following conditions are satisfied:
\begin{enumerate}[itemsep=2pt]
\item \label{mm4}  The maps $H^i(\rho_1)\colon H^i(\mcm_1)\to H^i(A)$ are 
isomorphisms for $i=0$ and $i=1$.
\item \label{mm5}  
The submodule $\ker( H^2(\rho_n)) \subset H^2(\mcm_n)$ is a free $R$-module
with basis given by the cohomology classes of the cocycles
$\{ d_{n+1}(x) \mid x\in \bX_{n+1}\} \subset Z^2(\T_R(\bX^n))$.
\end{enumerate}
\end{definition}

Set $\mcm\coloneqq\bigcup_n \mcm_n$.  
Since all diagrams of type \eqref{eq:atx-diag} commute, 
there is a morphism of $\cup_1$-dgas, $\rho\colon (\mcm,d) \to (A,d_A)$, 
whose restriction to $ \mcm_{n}$ coincides with $\rho_n$ 
for all $n\ge 1$.  We will oftentimes refer to 
$\rho\colon \mcm \to A$, or simply to 
$\mcm$ as being a $1$-minimal model for $A$; when needed, 
we will refer to the map $\rho\colon \mcm \to A$ as the structural  
morphism for $\mcm$.
If the inclusion maps $i_n \colon \mcm_n \to \mcm$
satisfy properties \eqref{mm4} and \eqref{mm5},
then we simply say that $\mcm$ is a $1$-minimal model.
If $\mcm = \T(\bX)$ is a $1$-minimal model
with $\bX$ a finite set, then we say 
that $\mcm$ is a finitely-generated $1$-minimal model.

From the definition of colimit of Hirsch extensions
we have that $\mcm_1=(\T_R(\bX_1), d_{\bz})$. 
Also note that from part \eqref{prt2} of Theorem \ref{thm:ext-bijection}
it follows that the map $d\vert_{\bX^n}$ is admissible for 
each $n \ge 1$.

\begin{lemma}
\label{lem:deg1-min}
Let $\mcm=\bigcup_{n\ge 1} \mcm_n$ be a $1$-minimal model
for $A$.  Then, for all $n\ge 1$,
\begin{enumerate}[itemsep=1pt]
\item \label{zz1}
$Z^1(\mcm_{n}) = H^1(\mcm_{n})$.
\item  \label{zz2}
The inclusion $\mcm_1\inj \mcm_n$ induces an 
isomorphism, $H^1(\mcm_1) \isom H^1(\mcm_n)$.
\end{enumerate}
\end{lemma}

\begin{proof}
Part \eqref{zz1} follows since 
$d\colon (\mcm_n)^0 \to (\mcm_n)^1$ is the zero map for all $n$.
Part \eqref{zz2} follows from Lemma \ref{lem:fgtf}, part \eqref{prop:1}.
\end{proof}

The next corollary follows at once from the lemma. 

\begin{corollary}
\label{cor:h1-mcm}
If $\mcm$ is a $1$-minimal model for $A$, then, for all $n\ge 1$,
\begin{enumerate}[itemsep=1pt]
\item \label{zx1}
$H^1(\mcm_{n})\cong M(\bX_1, R)$, the free 
$R$-module with basis given by the elements in $\bX_1$. 
\item \label{zx2}
$Z^1(\mcm_{n+1}) = Z^1(\mcm_{n})$.
\end{enumerate}
\end{corollary}

\subsection{Existence of $1$-minimal models}
\label{subsec:exist-1min}

We are now in a position to state and prove the main result of this section.

\begin{theorem}
\label{thm:cupd-minmodel}
Let $A$ be a binomial $\cup_1$-dga over $R=\Z$ or $\F_p$,  
with  $H^0(A) = R$ and $H^1(A)$ a 
finitely-generated, free $R$-module. There is then a $1$-minimal model, $\mcm$, and a 
structural morphism, $\rho\colon \mcm\to A$, which is a $1$-quasi-isomorphism.
\end{theorem}

\begin{proof}
Since $\mcm$ is connected, we can 
define $\rho^0 \colon \mcm^0 \to A^0$ to be the
composition of the inverse of the structure map
from $R$ to $\mcm^0$ followed by the structure map for $A$.

Now let $u_1, \ldots , u_k$ be cocycles in $A^1$ whose cohomology classes
give a basis for $H^1(A)$. Let $\bX_1=\{x_1, \ldots , x_k\}$ and 
set $\mcm_1 = (\T_R(\bX_1), d_{\tau_1})$, 
where $\tau_1=0$, that is, $d_{\tau_1}(x_i) = 0$ for all $i$. In view of 
Corollary \ref{cor:dx0}, 
we may define a morphism $\rho_1 \colon \mcm_1  \to A$ by setting 
$\rho_1(x_i) = u_i$ for $1\le i\le k$ such that the induced map 
on $H^0$ is an isomorphism. By construction, the map $H^1(\rho_1)$ 
is also an isomorphism.

Assume by induction that an extension $(\mcm_{n}, d_{\tau_n}) 
= (\T_R(\bX_1 \cup \cdots \cup \bX_n),d_{\tau_n})$ 
has been constructed, along with
a map $\rho_n \colon \mcm_n \to A$ inducing  
isomorphisms on
$H^0$ and $H^1$ and such that the kernel of $\rho_n$ is a
finitely generated, free $R$-module. Then by repeated applications 
of  Corollary \ref{cor:dx0}, Theorem \ref{thm:ext-bijection}, and 
Lemma \ref{lem:fgtf}, it follows that there is a finite set $\bX_{n+1}$,
an extension $\mcm_{n+1} = \T_R(\bX_1 \cup \cdots \cup \bX_{n+1})$ 
with differential $d_{\tau_{n+1}}$, 
and an extension $\rho_{n+1}$ of $\rho_n$ such that
$\rho_{n+1}$ induces isomorphisms on $H^0$ and $H^1$, the kernel
of $H^2(\rho_{n+1})$ is a finitely generated, free $R$-module, and 
the restriction of $H^2(\rho_{n+1})$ to the image of $\mcm_n$ in
$\mcm_{n+1}$ is a monomorphism.

If for some $n$ the map $H^2(\rho_n)$ is a monomorphism, then
set $\mcm = \mcm_n$; if not, then set $\mcm = \bigcup_{n\ge 1}\mcm_n$.
It then follows that $\mcm$ is a $1$-minimal model for $A$. Its structural 
morphism, $\rho\colon \mcm\to A$, is defined to be the direct limit of the morphisms 
$\rho_n\colon \mcm_n\to A$; that is, $\rho|_{\mcm_n}=\rho_n$, for all $n\ge 1$. 
By construction, the map $H^i(\rho)\colon H^i(\mcm)\to H^i(A)$ is an isomorphism 
for $i=0$ and $1$ and a monomorphism for $i=2$. Therefore, $\rho$ is 
a $1$-quasi-isomorphism, and the proof is complete.
\end{proof}

The theorem has an immediate corollary in the case when $A$ is a cochain 
algebra of a space.

\begin{corollary}
\label{cor:minmodel-space}
Let $X$ be a connected $\Delta$-complex, and assume $H^1(X;R)$ is a 
finitely generated module over $R=\Z$ or $\F_p$. There is then a 
$1$-minimal model, $\rho\colon \mcm\to A$, for the cochain algebra $A=C^*(X;R)$.
\end{corollary}

\subsection{Augmented $1$-minimal models}
\label{subsec:aug-1min}
When the binomial $\cup_1$-dga $A$ admits an augmentation, the above theorem 
can be enhanced, accordingly. 

\begin{theorem}
\label{thm:aug-minmodel}
Let $A$ be binomial $\cup_1$-dga such that there is an augmentation 
$\varepsilon_A\colon A\to R$ which induces an isomorphism from 
$H^0(A)$ to $R$ and such that $H^1(A)$ is a finitely generated, free $R$-module. 
There is then an augmented $1$-minimal model, $\mcm$, such 
that the structural morphism, $\rho\colon \mcm\to A$, is an augmentation-preserving 
$1$-quasi-isomorphism.
\end{theorem}

\begin{proof}
By Theorem \ref{thm:cupd-minmodel}, the binomial 
$\cup_1$-dga $A$ 
has a $1$-minimal model, $\rho\colon \mcm\to A$, which is a $1$-quasi-isomorphism.
Since the tensor algebra $\mcm=\T_R(\bX)$ is connected, it admits a canonical augmentation, 
$\varepsilon_{\mcm}\colon \mcm\to R$,  which sends $\mcm^{>0}$ to $0$ 
and identifies $\mcm^0$ with $R$. 

Since both $\varepsilon_A$ and $\rho$ are dga maps, their composite, 
$\varepsilon_A\circ \rho\colon \mcm\to R$ 
is again a dga map. Owing to our hypothesis on $\varepsilon_A$, 
the map from $R=H^0(\mcm)$ to $R$ induced by the composition
is an isomorphism of rings from $R$ to $R$
and so equals the identity of $R$. It follows that $\varepsilon_A\circ \rho$
is an augmentation for $\mcm$. By the uniqueness of 
augmentation maps for connected dgas, we have that  
$\varepsilon_A\circ \rho=\varepsilon_{\mcm}$,
and the proof is complete.
\end{proof}

Recall from Section \ref{subsec:augmented} that a choice of basepoint $x_0$ 
for a space $X$ yields an augmentation map, $\varepsilon_0\colon C^*(X;R)\to R$.

\begin{corollary}
\label{cor:minmodel-basepoint}
Let $(X,x_0)$ be a connected, pointed $\Delta$-complex, and assume 
$H^1(X;R)$ is a finitely generated module over $R=\Z$ or $\F_p$.
There is then an augmented $1$-minimal model, $\mcm$, for the 
cochain algebra $C^*(X;R)$ and a structural morphism, $\rho\colon \mcm\to A$, 
which is a $1$-quasi-isomorphism preserving augmentations, that is, 
$\varepsilon_0\circ \rho = \varepsilon_{\mcm}$.
\end{corollary}

\section{Uniqueness and functoriality of $1$-minimal models}
\label{subsec:1-min-unique}

\subsection{Maps between $1$-minimal models}
\label{subsec:min1-maps}

Let $(A,d)$ and $(A^\prime,d^\prime)$
be binomial $\cup_1$-dgas over the ring $R=\Z$ or $\F_p$, 
with  $H^0 = R$ and $H^1$ a finitely generated, free $R$-module. 
By Theorem \ref{thm:cupd-minmodel}, there are $1$-minimal 
models $\rho\colon \mcm\to A$ and 
$\rho^\prime \colon \mcm^\prime \to A^\prime$, 
in the sense of Definition \ref{def:1-min-model}.
We then have $\mcm=\bigcup_{n\ge 1} \mcm_n$, with inclusion maps 
$i_n\colon \mcm_n \inj \mcm_{n+1}$ which are Hirsch extensions,  
and morphisms $\rho_n\colon \mcm_n\to A$ such that 
$\rho_n=\rho\vert_{\mcm_n}$ and $\rho_{n+1}\circ i_n=\rho_n$ 
for all $n\ge 1$, and similarly for $A^\prime$ and $\mcm^\prime$.

Let $\varphi \colon A \to A^\prime$ be a map of binomial $\cup_1$-dgas. 
A map $f\colon \mcm\to \mcm'$ between the corresponding $1$-minimal 
models is said to be a {\em morphism compatible with $\varphi$}\/ if $f$ is  
a map of binomial $\cup_1$-dgas such that $\rho' \circ f=\varphi \circ \rho$. 
That is, we have a sequence of morphisms, 
$f_n\colon  \mcm_n\to \mcm'_n$, such the following diagrams 
commute, for all $n\ge 1$,
\begin{minipage}{.5\columnwidth}
\begin{equation}
\label{eq:square-f}
\begin{tikzcd}[row sep=2.5pc]
\mcm_{n+1} \ar[r, "f_{n+1}" ] 
			& \mcm_{n+1}^\prime \phantom{\, .}\\
\mcm_{n} \ar[u, "i_{n}" ] \ar[r, "f_n" ']
			& \mcm_{n}^\prime \, , \ar[u, "i'_{n}" ']
\end{tikzcd}
\end{equation}
\end{minipage}
\begin{minipage}{.5\columnwidth}
\begin{equation}
\label{eq:triangle-f}
\begin{tikzcd}[row sep=2.5pc, column sep=1 pc]
\mcm_n  \ar[rr,"f_n" ] \ar[d,  "\rho_{n}" ']
&& \mcm_n^\prime \phantom{\, .} \ar[d, "\rho_n^\prime" ]
\\
A \ar[rr, "\varphi" '] && A^\prime \, .
\end{tikzcd}
\end{equation}
\end{minipage}
The commutativity of the diagrams \eqref{eq:square-f} means that the 
morphisms $f_n\colon \mcm_n\to \mcm'_n$ between the $n$-th stages 
of the respective colimits of Hirsch extensions are compatible, in the sense  
delineated in Section \ref{subsec:colimits-hirsch}.

A weaker notion is that of a morphism \textit{compatible up to homotopy}; 
that is, a morphism $f\colon \mcm\to \mcm'$ preserving dga and binomial 
structures and such that $\rho' \circ f\simeq \varphi \circ \rho$. 
In this case, the diagram \eqref{eq:square-f} commutes
for all $n \ge 1$ and the diagram \eqref{eq:triangle-f}
commutes only up to homotopy, in the sense that there are homotopies
$\Phi_n\colon \mcm_n\to A'\otimes_R C^*(I;R)$  between $\varphi\circ \rho_n$ 
and $\rho_n^\prime\circ f_n$ for all 
$n \ge 1$, with the restriction of $\Phi_{n+1}$ 
to $\mcm_n$ equal to $\Phi_n$ for all $n \ge 1$.
Since by Theorem \ref{thm:cup-coho}
homotopic maps induce the same map on cohomology,
we still get commuting diagrams in cohomology,
\begin{equation}
\label{eq:triangle-f-coho}
\begin{tikzcd}[row sep=2.5pc, column sep=1.3 pc]
H^{i}(\mcm_n)  \ar[rr,"H^{i}(f_n)" ] \ar[d,  "H^{i}(\rho_{n})" ']
&&	H^{i}(\mcm_n^\prime)\phantom{\, ,} \ar[d, "H^{i}(\rho_n^\prime)" ]
\\
H^{i}(A) \ar[rr, "H^{i}(\varphi)" '] && H^{i}(A^\prime)\, ,
\end{tikzcd}
\end{equation}
for all $i \ge 0$ and all $n \ge 1$. 
 
\subsection{Extending dga maps to $1$-minimal models}
\label{subsec:extend}
Our next goal is to show that, given a morphism $\varphi \colon A \to A^\prime$, 
there is a morphism $f\colon \mcm \to \mcm^\prime$ compatible with $\varphi$ 
up to homotopy. We start with a lemma. 

\begin{lemma}
\label{lem:ext-bij}
Let $\varphi\colon A\to A'$ be a morphism of binomial 
$\cup_1$-dgas as above.  Let 
$\rho\colon \mcm \to A$ and $\rho'\colon \mcm' \to A$ be $1$-minimal 
models for $(A,d)$ and $(A',d')$, respectively, and 
let $n$ be a positive integer.  Assume  
there is a morphism $f_n\colon \mcm_n \to \mcm^\prime_n$ 
such that the diagram below commutes. 
\begin{equation}
\label{eq:triangle-f-coho-2}
\begin{tikzcd}
H^2(\mcm_n) \ar[r, "H^2(f_n)"] \ar[d, "H^2(\rho_n)"  ']
& H^2(\mcm_n^\prime) \phantom{.} \ar[d, "H^2(\rho_n^\prime)"] \\
H^2(A) \ar[r, "H^2(\varphi)"] & H^2(A^\prime) .
\end{tikzcd}
\end{equation}
Then 
\begin{enumerate}[itemsep=2pt]
\item \label{mcm1}
There is a morphism 
$f_{n+1}\colon \mcm_{n+1} \to \mcm_{n+1}^\prime$ 
such that diagram \eqref{eq:square-f} commutes.
\item \label{mcm2}
There is a bijection between morphisms $f_{n+1}$ as above 
and maps of sets from $\bX_{n+1}$ to 
$H^1(\mcm_{n}^\prime)$.
\item \label{mcm3}
If $f_n$ is an isomorphism and $H^2(\varphi)$ is a monomorphism, 
then $f_{n+1}$ is also an isomorphism.
\end{enumerate}
\end{lemma}

\begin{proof}
Set $K_n=\ker H^2(\rho_n)$ and $K'_n=\ker H^2(\rho'_n)$.
To prove parts \eqref{mcm1} and  \eqref{mcm2}, there are  three 
cases to consider: 
(a) $K_n = 0$, 
(b) $K_n \ne 0$, $K_n^\prime =0$, and
(c) $K_n \ne 0$, $K_n^\prime \ne 0$.

(a) If $K_n=0$, then $\mcm_n=\mcm_{n+1}$ and the
results in both parts follow, since
$f_{n+1} = i_n^\prime \circ f_n$ is the unique map
for which the diagram \eqref{eq:square-f} commutes.

(b) Next consider the case where $K_n\ne 0$ and $K_n^\prime =0$. 
In this case, $\mcm_{n+1} = \T(\bX_n \cup \bX_{n+1})$ with 
$X_{n+1} \ne \emptyset$, and $\mcm_{n+1}^\prime = \mcm_n^\prime$.
For $x \in X_{n+1}$, it follows from the commutativity of the diagram 
\eqref{eq:triangle-f-coho-2} that the cocycle $f_n(d_{\mcm_n}(x))$ is 
cohomologous to $0$ in $H^2(\mcm_n^\prime)$. Thus, for each 
$x \in \bX_{n+1}$, we can choose an element, 
$f_{n+1}(x) \in \mcm_n^\prime$ with 
$d_{\mcm_n^\prime} (f_{n+1}(x))= f_n(d_{\mcm_n}(x))$. Then by
Theorem \ref{def:ext-criterion} the map of sets
$\bX_{n+1} \to \mcm_n^\prime$ given by
$x \mapsto f_{n+1}(x)$ extends uniquely to a morphism
$f_{n+1} \colon \mcm_{n+1} \to \mcm_{n+1}^\prime$ 
of binomial $\cup_1$-dgas.
This completes the proof of part \eqref{mcm1} in this case.

Note that if $\hat{f}_{n+1}\colon \mcm_{n+1}\to 
\mcm_{n+1}^\prime$ is a morphism with $\hat{f}_{n+1}\circ i_n = f_n$, then 
$d_{\mcm_{n+1}^\prime} (f_{n+1}(x)) - d_{\mcm_{n+1}^\prime} (\hat{f}_{n+1}(x)) = 
f_n(d_{\mcm_n}(x))-f_n(d_{\mcm_n}(x))=0$.
Hence, the map of sets $X_{n+1} \to Z^1(\mcm_n^\prime)$
given by $x \mapsto c_x \coloneqq f_{n+1}(x) -\hat{f}_{n+1}(x)$
is a bijection, in this case, between morphisms
$\hat{f}_{n+1}$ with  $\hat{f}_{n+1}\circ i_n = f_n$ 
and maps of sets from $X_{n+1}$ to 
$Z^1\bigl(\mcm_n^\prime\bigr)$. 
From part \eqref{zz1} of Lemma \ref{lem:deg1-min}, 
we have that $Z^1(\mcm_n) = H^1(\mcm_n)$. This
completes the proof of part \eqref{mcm2} in case (b).

(c) Finally, consider the case where $K_n \ne 0$ and 
$K_n^\prime \ne 0$. Since the diagram \eqref{eq:triangle-f-coho-2}
commutes, the map $H^2(f_n)$ restricts to a homomorphism 
$k_n\colon K_n\to K'_n$ which fits into the commuting diagram below,
\begin{equation}
\label{eq:cd-snake}
\begin{tikzcd}[column sep=32]
0 \ar[r]              & K_{n} \ar[r ]  \ar[d, "k_{n}" ] 
			& H^2(\mcm_{n}) \ar[r, "H^2(\rho_n)" ]  \ar[d, "H^2(f_{n})" ]
			& H^2(A)  \ar[d, "{H^2(\varphi)}" ] \phantom{\, .}
\\
0 \ar[r]               & K'_{n}\ar[r]  
			& H^2(\mcm'_{n})  \ar[r, "H^2(\rho'_n)" ']  
			& H^2(A') \, .
\end{tikzcd}
\end{equation}
Note that by assumption both $K_n$ and $K'_n$ are non-zero, 
finitely generated, free $R$-modules.
Now let $\mcm_{n+1} = \T(\bX^n \cup \bX_{n+1})$
and $\mcm_{n+1}^\prime 
= \T\bigl(\bX^{\prime n} \cup \bX_{n+1}^\prime)$.
From hypothesis \eqref{mm5} in the definition
of a $1$-minimal for $A$, it follows that
the composition of $d\colon \T^1(\bX_{n+1})
\to Z^2(\mcm_n)$ followed by the projection of
a cocycle to its cohomology class gives an isomorphism
$\T^1(\bX_{n+1})\isom K_n$.  
Similarly, the differential from $\T^1(\bX_{n+1}^\prime)$
to $\mcm_{n+1}^\prime$ gives an isomorphism 
$\T^1(\bX_{n+1}^\prime)\isom K'_n$.
Therefore, $k_n$ gives a homomorphism
$\T^1(\bX_{n+1}) \to \T^1(\bX_{n+1}^\prime)$. 
By Theorem \ref{thm:tx-a}, this homomorphism 
extends uniquely to a map  
$f_{n+1}\colon \mcm_{n+1} \to \mcm^\prime_{n+1}$
of binomial $\cup_1$-dgas such that the diagram
\eqref{eq:square-f} commutes. This completes the 
proof of part \eqref{mcm1} in this case. 

Now fix a choice for $f_{n+1}$ and 
let $\hat{f}_{n+1}\colon \mcm_{n+1} 
\to \mcm^\prime_{n+1}$ be any morphism such that
the diagram \eqref{eq:square-f} commutes.
Let $x$ be any element in $X_{n+1}$.
Since $\hat{f}_{n+1}$ commutes with the differentials
and $dx \in \mcm_n$, we have 
$f_n \circ d(x) =d^\prime \circ f_{n+1}(x)$ and 
similarly $f_n \circ d(x) = d^\prime \circ\hat{f}_{n+1}(x)$. 
Hence,
\begin{equation}
\label{eq:dprime-f}
d^\prime f_{n+1}(x )- d^\prime \hat{f}_{n+1}(x)
			= f_n (d x) - f_n (d x )= 0,
\end{equation}
and it follows that $f_{n+1}(x )- \hat{f}_{n+1}(x)$ is a cocycle 
in $\bigl(\mcm^\prime_{n}\bigr)^1$.  
Therefore, for each $x \in \bX_{n+1}$ there is a cocycle
$c(x) \in \bigl(\mcm^\prime_{n}\bigr)^1$ such that
\begin{equation}
\label{eq:hat-f}
\hat{f}_{n+1}(x) = f_{n+1}(x) + c(x).
\end{equation}
By Lemma \ref{lem:deg1-min}, part \eqref{zz1} 
we have that 
$Z^1(\mcm_{n}^\prime) = H^1(\mcm_{n}^\prime)$.
This shows that a choice for $f_{n+1}$ gives a map
that sends isomorphisms $\hat{f}_{n+1}$ such that
the diagram \eqref{eq:square-f} commutes to
maps from $\bX_{n+1}$ to $H^1(\mcm_{n}^\prime)$.
This completes the proof of part \eqref{mcm2} in the last case. 

We now turn to part \eqref{mcm3}. 
Suppose that $f_n$ is an isomorphism and $H^2(\varphi)$ is  
a monomorphism.  Then $H^2(f_n)$ is also an isomorphism, 
and so chasing diagram \eqref{eq:cd-snake} 
shows that $H^2(f_n)$ restricts to an isomorphism 
$k_n\colon K_n\isom K'_n$. Since the differentials 
$d\colon \T^1(\bX_{n+1}) \to Z^2(\mcm_n)$ and
$d^\prime\colon \T^1(\bX_{n+1}^\prime) 
\to Z^2(\mcm_n^\prime)$ are monomorphisms, it 
follows that $k_n$ gives an isomorphism
$\T^1(\bX_{n+1}) \isom \T^1(\bX_{n+1}^\prime)$. 
By Theorem \ref{thm:tx-a}, this isomorphism 
extends uniquely to a morphism 
$f_{n+1}\colon \mcm_{n+1} \to \mcm^\prime_{n+1}$
of binomial $\cup_1$-dgas such that the diagram
\eqref{eq:square-f} commutes. 

We claim that if $\hat{f}_{n+1}$ is any morphism
from $\mcm_{n+1}$ to $\mcm_{n+1}^\prime$ such
that the diagram \eqref{eq:square-f} commutes, then
$\hat{f}_{n+1}$ is in fact, an isomorphism.  
To prove the claim, first note that since
$f_{n+1}$ restricts to an isomorphism from
$\T^1(\bX_{n+1})$ to $\T^1(\bX_{n+1}^\prime)$,
it follows that
$f_{n+1}\colon \mcm_{n+1} \to \mcm_{n+1}^\prime$
is an isomorphism. By equation \eqref{eq:hat-f} it follows
that $f_{n+1}$ and $\hat{f}_{n+1}$ induce the same
map of $R$-modules from $\mcm_{n+1}/\mcm_n$ to
$\mcm_{n+1}^\prime/\mcm_{n}^\prime$. 
Consider the following commutative diagram of
exact sequences of $R$-modules.
\begin{equation}
\label{eq:cd-min}
\begin{tikzcd}
0 \ar[r]   
			& \mcm_{n} \ar[r, "i_n" ]  \ar[d, "f_n" ] 
			& \mcm_{n+1} \ar[r, "q_n" ]  \ar[d, "\hat{f}_{n+1}" ]
			& \mcm_{n+1}/\mcm_n \ar[r]
								\ar[d, "{[f_{n+1}] =[\hat{f}_{n+1}]}" ]
			& 0 \phantom{\, .}\\
0 \ar[r]   
			& \mcm_{n}^\prime \ar[r, "i_n^\prime" ']  
			& \mcm_{n+1}^\prime \ar[r, "q_n^\prime" ']  
			& \mcm_{n+1}^\prime/\mcm_n^\prime \ar[r]
			& 0	\, .
\end{tikzcd}
\end{equation}
By assumption, $f_n$ is an isomorphism; moreover, $[\hat{f}_{n+1}]$
is an isomorphism, since $[f_{n+1}]$ is an isomorphism.
The claim now follows from the Five Lemma, and 
this completes the proof.
\end{proof}

\subsection{Lifting homotopies to $1$-minimal models}
\label{subsec:homotopy-lift}

The next step is to show that homotopies between maps 
of binomial $\cup_1$-dgas lift to the respective $1$-minimal models.

\begin{lemma}[Homotopy Lifting Lemma]
\label{lem:homotopy-lift}
Let $(A,d_A)$ and $(A',d_{A'})$ be binomial cup-one dgas over $R=\Z$ or $\F_p$ 
with $H^0=R$ and $H^1$ finitely generated, free $R$-modules.
Let $\rho\colon \mcm(A) \to A$ and $\rho' \colon  \mcm(A') \to A'$ 
be $1$-minimal models, and let $\varphi \colon A \to A'$ be a morphism. 
Suppose for a given $n \ge 1$ there is a morphism
$f_n\colon \mcm_n \to \mcm_n^\prime$ and a
homotopy $\Phi_n\colon \mcm_n\to A'\otimes_R C^*(I;R)$ 
between $\varphi\circ \rho_n$ and $\rho_n^\prime\circ f_n$. Then,
\begin{enumerate}[itemsep=2pt]
\item \label{hl1}
There is a unique morphism 
$f_{n+1}\colon \mcm_{n+1} \to \mcm_{n+1}^\prime$
such that $f_{n+1}\circ i_n = i_n^\prime \circ f_n$ and such that
there is a homotopy $\Phi_{n+1}\colon \mcm_{n+1}\to A'\otimes_R C^*(I;R)$ 
between $\varphi\circ \rho_{n+1}$ 
and $\rho_{n+1}^\prime\circ f_{n+1}$ with
$\Phi_{n+1}\vert_{\mcm_n} = \Phi_n$.
\item \label{hl2}
If in addition $f_n$ is an isomorphism and 
$H^2(\varphi)$ is a monomorphism, 
then $f_{n+1}$ is also an isomorphism.
\end{enumerate}
\end{lemma}

\begin{proof}
As before, let $\mcm_{n+1} = \T(\bX^n \cup \bX_{n+1})$ and 
let $\mcm_{n+1}^\prime = \T\bigl(\bX^{\prime n} \cup \bX_{n+1}^\prime)$, 
with corresponding maps $\rho$ and $\rho^\prime$ 
as pictured in the diagram below.
\begin{equation}
\label{eq:1-mm-maps}
\begin{tikzcd}[column sep= 18pt, row sep=30pt]
\mcm_{n+1} \ar[rrrr,  "f_{n+1}" ]
             \ar[ddr, bend right=40, "\rho_{n+1}" ']
		&&&& \mcm^\prime_{n+1}
					\ar[ddl, bend left=40, "\rho_{n+1}^\prime" ]
					\\
    &\mcm_{n} \ar[rr, "f_{n}" ] \ar[d, "\rho_n" ]
       \ar[ul, "i_n" ']
		&& \mcm^\prime_{n} \ar[d, "\rho_n^\prime" ']	
		   \ar[ur, "i_n^\prime" ]
		\\
		&A  \ar[rr, "\varphi" ']
		&& A^\prime
\end{tikzcd}
\end{equation}
Let $d_{n+1}$ denote the differential on $\mcm_{n+1}$,
and note that for $x \in \bX_{n+1}$, we have $d_{n+1} (x) \in \mcm_n$.
By Theorem \ref{thm:tx-a}, it suffices to define for each $x \in \bX_{n+1}$ 
an element $\Phi_{n+1}(x) \in A' \otimes_R C^\ast(I;R)$ such that
$ \Phi_{n+1} \circ d_{n+1}(x) = d_{A'\otimes_R C^*(I;R)} \circ \Phi_{n+1} (x)$.
Let $x \in \bX_{n+1}$. Then $d_{n+1}(x)$ is a cocycle in $\mcm_n$,
and hence, $\Phi_n(d_{n+1}(x))$ is a cocycle in $A' \otimes_R C^\ast(I;R)$.
We can assume this cocycle has the form
\begin{equation}
\label{eq:varphi-n-dx}
\Phi_n(d_{n+1} (x)) = 
\varphi(\rho_n(d_{n+1} x)) t_0 
		+ \rho_n^\prime ( f_n(d_{n+1} x)) t_1 + c_1(x) u ,
\end{equation}
with $c_1(x) \in (A^\prime)^1$.

The condition that $\Phi_n(d_{n+1} (x))$ is a cocycle 
then leads to an equation for 
$d_{A^\prime} (c_1(x))$, as follows.
Recall that if $u_0, u_1, c$ are homogeneous elements
in $A^\prime$ 
with $\abs{u_0}=\abs{u_1}= \abs{c}+1$, then
$u = u_0t_0 + u_1t_1 + cu$ is a homogeneous
element in $A' \otimes_R C^\ast(I;R)$ with
\begin{equation}
\label{eq:dAtimesI}
\begin{split}
d_{A' \otimes C^\ast(I;R)}(u_0t_0 + u_1t_1 + cu)
			& = d_{A^\prime}(u_0)t_0 + d_{A^\prime}(u_1)t_1\\
	 & \quad + \left( (-1)^{|u_0|+1} u_0 + (-1)^{|u_1|}u_1
	 						+ d_{A^\prime}(c) \right)u.
\end{split}	 						
\end{equation}
In particular, if $u_0t_0 +u_1t_1 + cu$ is a cocycle
in $A' \otimes_R C^\ast(I;R)$, then $u_0$ and $u_1$ are 
cocycles in $A^\prime$ and 
$d_{A^\prime} (c) 
			=  (-1)^{|u_0|} u_0 + (-1)^{|u_1|+1}u_1$.
Since $\Phi_n(d_{n+1} x)$ is a cocycle 
in $A' \otimes_R C^\ast(I;R)$,
we have that
\begin{equation}
\label{eq:dc1x}
d_{A^\prime} (c_1(x) )= \varphi(\rho_n(d_{n+1}x)) 
				- \rho_n^\prime ( f_n(d_{n+1}x)).
\end{equation}
Now by Lemma \ref{lem:ext-bij} there are morphisms
$f_{n+1}\colon \mcm_{n+1} \to \mcm^\prime_{n+1}$
such that the diagram \eqref{eq:square-f} commutes.
Choose such a map $f_{n+1}$. 
The map $f_{n+1}$ then determines a map from
$X_{n+1}$ to $H^1(\mcm_{n}^\prime)$, as follows.
For each $x \in X_{n+1}$, we have
\begin{equation}
\label{eq:drho}
d_{A^\prime}(\varphi(\rho_{n+1}(x)))
		=\varphi(\rho_n(d_{n+1} x))
\quad\text{and}\quad
d_{A^\prime}(\rho_{n+1}^\prime (f_{n+1}(x)))
		= \rho_n^\prime ( f_n (d_{n+1}x)),
\end{equation}
and it follows from equations \eqref{eq:dc1x} and \eqref{eq:drho} that the element 
$z(x)\coloneqq c_1(x)  -\varphi(\rho_{n+1}(x) ) + \rho^\prime_{n+1}( f_{n+1}(x))$ 
is a cocycle in ${A^\prime}$. Thus, we have a map of sets
$\bX_{n+1} \to H^1(A^\prime)$ given by $x \mapsto [z(x)]$, 
where $[w]$ is the cohomology class of a cocycle $w$.
By Definition \ref{def:1-min-model} and 
Lemma \ref{lem:deg1-min}, this map corresponds uniquely
to a map of sets from $\bX_{n+1}$ to $H^1(\mcm_{n}^\prime)$.

By Lemma \ref{lem:ext-bij} we can assume that $f_{n+1}$
has been chosen so that for each $x \in \bX_{n+1}$, 
we have that $[z(x)]=0$.
It then follows that for each $x$ there is an element
$c_0(x) \in (A^\prime)^0 $ with 
\begin{equation}
\label{eq:dz}
d_{A^\prime}( c_0(x) )
		= z(x)= c_1(x) -\varphi(\rho_{n+1}(x)) 
								+ \rho^\prime_{n+1}( f_{n+1}(x)).
\end{equation}
 Now set
\begin{equation}
\label{eq:varphi-n1-x}
\Phi_{n+1}(x) \coloneqq 
\varphi(\rho_{n+1}(x))t_0 + \rho_{n+1}^\prime (f_{n+1}(x)) t_1 + c_0(x) u.
\end{equation}
The final step is to show that with this choice
of $\Phi_{n+1}$, it follows that
$\Phi_n (d_{n+1}(x)) 
		= d_{A^\prime \otimes C^\ast(I)} 
					(\Phi_{n+1} (x))$ for all $x\in \bX_{n+1}$. 
Using equations \eqref{eq:varphi-n1-x}, \eqref{eq:dz}, 
and  \eqref{eq:varphi-n-dx}, we have that
\begin{gather}
\begin{aligned}
\label{eq:dphi-n}
d_{A' \otimes C^\ast(I)} (\Phi_{n+1}(x) )
		& =d_{A^\prime \otimes C^\ast(I)}
					\big(\varphi(\rho_{n+1}(x))t_0 
					+ \rho_{n+1}^\prime (f_{n+1}(x))t_1 
					+ c_0(x) u \big)\\
		& = \varphi(\rho_n(d_{n+1}x) )t_0 
					+ \rho_n^{\prime}( f_n(d_{n+1}x)) t_1\\		
		& \quad			+[\varphi(\rho_{n+1}(x) )
								- \rho_{n+1}^\prime( f_{n+1}(x))
								+d_{A^\prime}(c_0(x))]u\\
		& =\varphi( \rho_n(d_{n+1}x) )t_0 
					+ \rho_n^{\prime} (f_n(d_{n+1}x)) t_1		
					+c_1(x)u \\
		& = \Phi_n(d_{n+1}x), 
\end{aligned}
\end{gather}
and the proof is complete.
\end{proof}

\subsection{Homotopy functoriality of $1$-minimal models}
\label{subsec:unique-functoriality}
We are now in a position to show that $1$-minimal models are 
unique up to homotopy.

\begin{theorem} 
\label{thm:1-min-lift}
Let $(A,d)$ and $(A',d')$ be binomial cup-one dgas over  $R=\Z$ or $\F_p$ 
such that $H^0(A)$ and $H^0(A')$ are isomorphic to $R$ and 
$H^1(A)$ and $H^1(A')$ are finitely generated, free $R$-modules.
Let $\rho \colon \mcm(A) \to A$ and $\rho' \colon  \mcm(A') \to A'$ 
be $1$-minimal models, and let $\varphi \colon A \to A'$ be a morphism. 
Then
\begin{enumerate}
\item \label{1-min-1}
There is a morphism $\widehat{\varphi} \colon \mcm(A) \to \mcm(A^\prime)$ 
compatible with the respective colimit structures (that is, 
$\widehat{\varphi}_{n+1}\circ i_n = i_n^\prime \circ \widehat{\varphi}_n$ 
for all $n$), and there is a homotopy 
$\Phi\colon \mcm(A)\to A'\otimes_R C^*(I;R)$  
between $\varphi \circ \rho$ and $\rho' \circ \widehat{\varphi}$ (also 
preserving colimit structures), so that 
the diagram below commutes up to homotopy.
\begin{equation}
\label{diag:1-mm-lift}
\begin{tikzcd}
\mcm(A) \ar[d, "\rho" ] \ar[r, dashed, "\widehat{\varphi}"]
		& \mcm(A')\phantom{.} \ar[d, "\rho'"]\\
A \ar[r, "\varphi"]  & A' .
\end{tikzcd}
\end{equation}
\item \label{1-min-2}
If $\varphi$ is a $1$-quasi-isomorphism, 
 then $\widehat{\varphi}$ is an isomorphism.
\end{enumerate}
\end{theorem}

\begin{proof}
To prove part \eqref{1-min-1}, first set $\mcm_n=\mcm_n(A)$ and 
$\mcm'_n=\mcm_n(A')$.  We need to construct isomorphisms 
$\widehat{\varphi}_n \colon \mcm_n \to \mcm'_n$ 
and homotopies $\Phi_n\colon \mcm_n\to A\otimes_R C^*(I;R)$ 
between $\rho_n$ and $\rho_n^\prime\circ \widehat{\varphi}_n$ 
such that $\widehat{\varphi}_{n+1}\circ i_n = i_n^\prime \circ \widehat{\varphi}_n$ 
and $\Phi_{n+1}\vert_{\mcm_n} = \Phi_n$.
The proof is by induction on $n$.

The base case is to show that there is an isomorphism
$\widehat{\varphi}_1\colon \mcm_1 \to \mcm_1^\prime$ and a homotopy
$\Phi_1\colon \mcm_1\to A \otimes_R C^*(I;R)$ between 
$\rho_1$ and $\rho_1^\prime \circ \widehat{\varphi}_1$. Since 
$H^1(\rho_1)=H^1(\rho_1^\prime \circ \widehat{\varphi}_1)$, 
the claim follows from Lemma \ref{lem:hh}.
The induction step now follows from the homotopy
lifting lemma (Lemma \ref{lem:homotopy-lift}).

To prove part \eqref{1-min-2}, assume that $\varphi$ is a $1$-quasi-isomorphism; 
that is, $H^1(\varphi)\colon H^1(A)\to H^1(A')$ is an isomorphism and 
$H^2(\varphi)\colon H^2(A)\to H^2(A')$ is a monomorphism. 
Now, the map $H^1(\varphi)$ determines a map 
$\widehat{\varphi}_1\colon \mcm_1\to \mcm'_1$, 
which must also be an isomorphism. 
Since $H^2(\varphi)$ is a monomorphism,  
Lemma~\ref{lem:ext-bij}, part \eqref{mcm3} 
insures that $\widehat{\varphi}_1$ lifts to compatible isomorphisms, 
$\widehat{\varphi}_n\colon \mcm_n \to \mcm'_n$, 
for all $n\ge 1$. It follows that the family of maps $\big\{\widehat{\varphi}_n\big\}_{n\ge 1}$ 
defines the desired isomorphism $\widehat{\varphi}\colon \mcm(A)\to \mcm(A')$, 
and this completes the proof.
\end{proof}

Taking $A=A'$ in the above theorem, we obtain the following corollary.

\begin{corollary}
\label{cor:1-min-unique}
Let $A$ be a binomial cup-one $R$-dga as above, 
and let $(\mcm, \rho)$ and $(\mcm^\prime, \rho^\prime)$ be any two 
$1$-minimal models for $A$. Then there is an isomorphism 
$f\colon \mcm \to \mcm'$ such that $ \rho' \circ f$ is 
homotopic to $\rho$.
\end{corollary}

\subsection{$1$-minimal models and homotopies}
\label{subsec:1-min-homotopy}

The next lemma corresponds to an analogous result in \cite{Griffiths-Morgan} 
(Corollary 11.4); see also \cite[Proposition 12.7]{FHT}.

\begin{lemma}
\label{lem:1-min-equiv}
Homotopy is an equivalence relation on the set of morphisms 
from a colimit of Hirsch extensions, $\mcm$, to a binomial cup-one dga, $A$.
\end{lemma}

\begin{proof}
Clearly, $\simeq$ is reflexive. To show symmetry, let 
$\Phi\colon \mcm\to A\otimes_R C^*(I;R)$ be a homotopy 
from $\varphi_0$ to $\varphi_1$, given on elements $a\in \mcm^i$ by  
$\Phi(a)=\varphi_0(a)t_0 + \varphi_1(a)t_1 + c(a) u$, for some 
$c(a)\in A^{i-1}$; then the map $\overline{\Phi}$ given by 
$\overline{\Phi}(a)=\varphi_1(a)t_0 + \varphi_0(a)t_1 + c(a) u$ 
is a homotopy from $\varphi_1$ to $\varphi_0$.

It remains to show $\simeq$ is transitive. With $\Phi$ as above,
let $I'$ be another copy of the interval, let $t'_0, t'_1, u'$ be the 
corresponding generators of $C^*(I';R)$, and let 
$\Phi'\colon \mcm\to A\otimes_R \C^*(I';R)$  be a homotopy 
from $\varphi_1$ to $\varphi_2$. 
Finally, let $C^*(I;R)\vee C^*(I';R)$ be the fiber product 
corresponding to the augmentations 
$\varepsilon\colon C^*(I;R)\to R$ and 
$\varepsilon'\colon C^*(I';R)\to R$ given by 
$\varepsilon(t_1)=\varepsilon'(t'_0)=1$ and 
$\varepsilon(t_0)=\varepsilon'(t'_1)=0$. 
With this setup, we define a map 
\begin{equation}
\label{diag:vee-map}
\begin{tikzcd}[column sep=20pt]
\Psi\colon \mcm\ar[r]& A\otimes_R (C^*(I;R)\vee C^*(I';R))
\end{tikzcd}
\end{equation}
by setting $\Psi(a)=(\Phi(a),\Phi'(a))$. Now let $\Delta$ be a triangle
with oriented edges $e_1=I$, $e_2=I'$, and $e_3=I''$. The inclusions of the 
edges in the triangle induce epimorphisms $q_j\colon C^*(\Delta;R) \surj C^*(e_j;R)$, 
which give a surjection $f\colon C^*(\Delta;R) \surj C^*(I;R)\vee C^*(I';R)$. 
By Theorem \ref{thm:1-qi-lift}, the morphism $\Psi$ lifts through $f$ to a 
morphism $\widehat{\Psi}\colon \mcm\to A\otimes_R C^*(\Delta;R)$. 
The map $q_3\circ \widehat{\Psi}\colon \mcm\to A\otimes_R \C^*(I'';R)$, 
then, is the desired homotopy from $\varphi_0$ to $\varphi_2$.
\end{proof}

We will write $[\mcm, A]$ for the set of homotopy classes of morphisms 
$\varphi \colon \mcm \to A$. Given a morphism $\xi\colon A\to A'$, composition 
with $\xi$ defines a function, $\xi_\ast\colon [\mcm, A] \to [\mcm, A^\prime]$.
Similar notions hold for augmented dgas and  augmentation-preserving 
morphisms between them. 
The next lemma corresponds to an analogous result in \cite{Griffiths-Morgan} 
(Theorem 11.5); see also \cite[Proposition 12.9]{FHT}.

\begin{lemma}
\label{lem:1qiso-homotopy}
Let $(\T_R(\bX), d)$ be a colimit of Hirsch extensions and
assume $A$ and $A^\prime$ are binomial $\cup_1$-dgas 
with augmentations, $\varepsilon_A, \varepsilon_{A^\prime}$
that induce isomorphisms from $H^0$ to $R$.
Assume further that there is 
an augmentation-preserving $1$-quasi-isomorphism, 
$\xi \colon A \to A^\prime$.
Then the induced map of equivalence classes of 
augmentation-preserving homotopies of 
augmentation preserving maps,
$\xi_{*}\colon [\T_R(\bX), A] \to [\T_R(\bX), A^\prime]$, 
is injective.
\end{lemma}

\begin{proof}
Let $f,g\colon \T_R(\bX)\to A$ be augmentation-preserving 
morphisms and let $H \colon \T_R(\bX) \to A^\prime \otimes_R C^\ast(I;R)$
be an augmentation-preserving homotopy
between $\xi\circ f$ and $\xi \circ g$.
We will show that $H$ lifts to 
an augmentation-preserving homotopy $\widehat{H}$
between $f$ and $g$.

Write $\bX=\{x_1, x_2, \ldots\}$ with $\bX^n =\{x_1, \ldots , x_n\}$.
Set $\mcm_n = (\T_R(\bX^n),d_n)$ and let $f_n, g_n$ and 
$H_n$ denote the restrictions of $f, g$, and $H$ to
$\mcm_n$. Then $x_1$ is a cocycle in $\mcm_1$, and we have 
\begin{equation}
\label{eq:Hx1}
H(x_1) = \xi \circ f(x_1)t_0 + \xi \circ g(x_1)t_1 + c(x_1)u
\end{equation}
with 
\begin{equation}
\label{eq:HL-1}
dc(x_1) = \xi \circ f(x_1) - \xi \circ g(x_1).
\end{equation}
Since $H^1(\xi) \colon H^1(A) \to H^1(A^\prime)$ is a 
monomorphism, it follows that there is an element
$\hat{c}(x_1) \in \ker (\varepsilon_A)$ with 
\begin{equation}
\label{eq:HL-2}
d\hat{c}(x_1) = f(x_1) - g(x_1) 
= f_1(x_1) - g_1(x_1).
\end{equation}
Since
$f_1(x_1)t_0 + g_1(x_1)t_1 + \hat{c}(x_1)u$ is a 
cocycle in $A \otimes_R C^1(I;R)$ and 
$\hat{c}(x_1) \in \ker(\varepsilon_A)$ with
$\varepsilon_A$ a binomial subalgebra of $A^0$,
it follows that the map
$x_1 \mapsto f_1(x_1)t_0 + g_1(x_1)t_1 + \hat{c}(x_1)u$ 
extends uniquely to an augmentation-preserving
homotopy $\widehat{H}_1$ between $f_1$ and $g_1$.

The next step is to show that $\widehat{H}_1$ is a lifting
of $H_1$. Since $\hat{c}(x_1) \in \ker(\varepsilon_A)$ and $\xi$
is augmen\-tation-preserving, it follows that
$\xi\circ \hat{c}(x_1) \in \ker (\varepsilon_{A^\prime})$.
From equations \eqref{eq:HL-1} and \eqref{eq:HL-2},
we have that the elements $\xi \circ \hat{c}(x_1)$ and
$c(x_1)$ have the same coboundary. 
Since $\varepsilon_{A^\prime}$ induces an isomorphism
from $H^0(A^\prime)$ to $R$, it follows that two elements in
$\ker (\varepsilon_{A^\prime})$ with the same coboundary
are equal to each other; hence 
$\xi \circ \hat{c}(x_1) = c(x_1)$ and $\hat{H}_1$ is
a lifting of $H_1$.

Now assume by induction that the homotopy
$H_n$ lifts to a homotopy $\widehat{H}_n$ between
$f_n$ and $g_n$ and show that $\widehat{H_n}$ then
extends to a lifting $\widehat{H}_{n+1}$ of $H_{n+1}$.
Note that $dx_{n+1}$ is a cocycle in $\T(X^n)$.
Hence
$\widehat{H}_n (dx_{n+1})
= f_n(dx_{n+1}) t_0 + g_n(dx_{n+1})t_1+ \hat{c}(dx_{n+1}) u$
is a cocycle in $A \otimes_R C^\ast(I;R)$, and it follows that
$d\hat{c}(dx_{n+1}) = f_n(dx_{n+1}) - g_n(dx_{n+1})$. 

The obstruction to extending 
$\widehat{H}_n$ to a homotopy $\widehat{H}_{n+1}$ is
finding an element $\hat{c}(x_{n+1}) \in A^0$ such that the map
$x_{n+1} \mapsto \widehat{H}_{n+1}(x_{n+1}) = 
f(x_{n+1})t_0 + g(x_{n+1})t_1 + \hat{c}(x_{n+1})u$ 
commutes with the coboundary map; that is,
$d\hat{c}(x_{n+1}) = g(x_{n+1}) - f(x_{n+1}) + \hat{c}(dx_{n+1})$. 
Since $\widehat{H}_{n}$ is a lifting of $H_n$, we have
that $H_{n+1}(x_{n+1})$ has the form
$ \xi \circ f(x_{n+1})t_0 + \xi \circ g(x_{n+1})t_1+ c(x_{n+1})u$,  
where 
\[
d{c}(x_{n+1}) = \xi \circ g(x_{n+1}) 
			-\xi \circ f(x_{n+1}) 
		        + \xi \circ \hat{c}(dx_{n+1}).
\]
In particular, the cocycle 
$\xi \circ g(x_{n+1}) -\xi \circ f(x_{n+1}) 
					+ \xi \circ \hat{c}(dx_{n+1})$ in $A^\prime$
is cohomologous to zero, and then since 
$\xi \colon H^1(A) \to H^1(A^\prime)$ is a
monomorphism, it follows that the cocycle
$g(x_{n+1}) - f(x_{n+1}) + \hat{c}(dx_{n+1})$ 
is cohomologous to zero in $A$.
Thus, there is an element $\hat{c}(x_{n+1})$ in
$\ker (\varepsilon_A)$ with 
$ d\hat{c}(x_{n+1}) 
			= g(x_{n+1}) - f(x_{n+1}) + \hat{c}(dx_{n+1})$,
and hence, $\widehat{H}_n$ extends to a homotopy
$\widehat{H}_{n+1}$.

The final step is to see that $\widehat{H}_{n+1}$ is a 
lifting of $H_{n+1}$. Since 
$\hat{c}(x_{n+1}) \in \ker (\varepsilon_A)$ and 
$\xi$ is augmentation-preserving, it follows that
$\xi \circ \hat{c}(x_{n+1}) \in 
\ker (\varepsilon_{A^\prime})$. The elements
$\xi \circ \hat{c}(x_{n+1})$ and $c(x_{n+1})$
have the same coboundary and are both elements
in $\ker (\varepsilon_{A^\prime})$; hence they are equal
to each other given that the augmentation
$\varepsilon_{A^\prime}$ induces an isomorphism
from $H^0(A^\prime)$ to $R$. This completes the argument that
$\widehat{H}_{n+1}$ is a lifting of $H_{n+1}$, and hence
the proof is complete.
\end{proof}

\begin{remark}
\label{rem:pc-aug}
Note that if $A = C^\ast(X;R)$ with $X$ a path-connected $\Delta$-complex, 
then there is an augmentation $\varepsilon_A \colon A \to R$
inducing an isomorphism from $H^0(A)$ to $R$.
\end{remark}

\subsection{$1$-minimal models of augmented binomial dgas}
\label{subsec:augmented-functoriality}
Let $R= \Z$ or $\F_p$. 
For binomial cup-one $R$-dgas $(A,d_A)$ that come equipped with an augmentation, 
$\varepsilon_{A} \colon A\to R$, that induces an isomorphism from $H^0(A)$ to $R$, 
and for which $H^1(A)$ is a finitely generated, free $R$-module, 
Theorem \ref{thm:1-min-lift} may be refined. 
Recall from Theorem \ref{thm:aug-minmodel} that any such dga admits an 
augmented $1$-minimal model, $\rho \colon \mcm \to A$. 

\begin{theorem}
\label{thm:1-min-cup-aug}
Let $(A,d)$ and $(A',d')$ be augmented binomial cup-one dgas as above. 
Let $\rho \colon \mcm \to A$ and $\rho' \colon  \mcm' \to A'$ 
be augmented $1$-minimal models,  
and let $\varphi \colon A \to A'$ be an augmentation-preserving morphism. 
There is then a unique augmentation-preserving morphism
$\widehat{\varphi} \colon \mcm \to \mcm^\prime$ 
such that $\varphi \circ \rho$ is  augmentation-preserving
homotopic to $\rho^\prime \circ \widehat{\varphi}$.
\end{theorem}

\begin{proof}
By Theorem \ref{thm:1-min-lift}, there is an isomorphism 
$\widehat{\varphi} \colon \mcm \to \mcm^\prime$ 
and a homotopy $\Phi\colon \mcm\to A \otimes_R C^*(I;R)$ between 
$\varphi\circ \rho$ and $\rho^\prime \circ \widehat{\varphi}$. 
Let $\widetilde{\varphi} \colon \mcm \to \mcm^\prime$ be 
another such isomorphism.
Since $\rho^\prime \circ \widehat{\varphi}$ and 
$\rho^\prime \circ \widetilde{\varphi}$ are both
homotopic to $\varphi \circ \rho$, it follows 
from Lemma \ref{lem:1-min-equiv} 
that $\rho^\prime \circ \widehat{\varphi}$ and 
$\rho^\prime \circ \widetilde{\varphi}$ are homotopic 
to each other. Then from Lemma \ref{lem:1qiso-homotopy}, 
it follows that $\widehat{\varphi}$ and $\widetilde{\varphi}$ are
homotopic morphisms from $\mcm$ to $\mcm^\prime$.

Since the proofs of Theorem \ref{thm:1-min-lift}
and Lemma \ref{lem:1-min-equiv} apply as well
to augmentation preserving homotopies,
it follows that $\widehat{\varphi}$ and 
$\widetilde{\varphi}$ are homotopic by an 
augmentation-preserving homotopy.
It then follows from Lemma \ref{lem:miraculous} that 
$\widehat{\varphi} = \widetilde{\varphi}$,
and the proof is complete.
\end{proof}

The following uniqueness result follows from Theorem \ref{thm:1-min-cup-aug} 
by talking $A=A^\prime$.

\begin{corollary}
\label{cor:1-mm-unique}
Let $A$ be an augmented binomial cup-one $R$-dga as above, 
and let $(\mcm, \rho)$ and $(\mcm^\prime, \rho^\prime)$ be any two 
augmented $1$-minimal models for $A$.
Then there is a unique augmentation-preserving
isomorphism $f\colon \mcm \to \mcm^\prime$
such that $ \rho' \circ f$ is 
augmentation-preserving homotopic to $\rho$.
\end{corollary}

\section{Compatibility of integral and rational $1$-minimal models}
\label{sect:Z-vs-Q}

In this section we define the Postnikov $1$-tower of 
a connected space $Y$ with $H^1(Y;R)$ finitely
generated and show that the steps $\mcm_n$ in the 
$1$-minimal model of $Y$ are quasi-iso\-morphic to
the cochains of the corresponding spaces $Y_n$ in the 
Postnikov system. We then use this result
to show that the integral $1$-minimal model
of $Y$ tensored with $\Q$ is weakly equivalent
as a dga to the $1$-minimal model of $Y$
defined in rational homotopy theory.

\subsection{Postnikov $1$-towers and $1$-minimal models}
\label{subsec:postnikov}

A connected space $Y$ with $H^1(Y;R)$ finitely
generated determines a tower of spaces 
with compatible maps from $Y$ to the tower,
as pictured in display \eqref{eq:postnikov-space}, 
as follows.
\vspace*{-4pt}
\begin{equation}
\label{eq:postnikov-space}
\begin{tikzcd}[row sep=22pt, column sep=60pt]
&\ar[d, dotted] \\
&  Y_3 \ar[d, "\pi_2"]\\
&  Y_2 \ar[d, "\pi_1"]\\
Y \ar[r, "h_1" '] \ar[ur, "h_2" '] \ar[uur, "h_3"]
&  Y_1
\end{tikzcd}
\end{equation}
The spaces $Y_n$ and maps $h_n$ are defined by
induction as follows.
Set $Y_1$ equal to the Eilenberg--MacLane space
$K(H^1(Y;R), 1)$ and let $h_1$ be a map inducing
an isomorphism from $H^1(Y_1;R)$ to $H^1(Y;R)$.

Assume $Y_n$ and $h_n \colon Y \to Y_n$ have been defined.
Let $P_n$ be the kernel of $H^2(h_n)$. Using the argument in the proof 
of property \eqref{prop:3} of Lemma \ref{lem:fgtf}, we can assume by 
induction that $P_n$ is a finitely-generated, free $R$-module.
Set $Q_n = \Hom_R(P_n,R)$.
Then the set of homotopy classes of maps from
$Y_n$ to $K(Q_n;2)$ can be identified with
the set of $R$-linear maps from $P_n$ to $H^2(Y_n;R)$.
Choose a map $f_n\colon Y_n \to K(Q_n;2)$
corresponding to the map of $P_n$ to the kernel
of $H^2(h_n)$. Set $Y_{n+1}$ equal to the 
pullback over $f_n$ of the path space fibration on
$K(Q_n;2)$, and let $h_{n+1}$ be a lifting of $h_n$.
We call the resulting tower the
{\em Postikov $1$-tower}\/ of the space $Y$.
It will be shown in \cite{Porter-Suciu-group-1min} that
the $R$-modules $Q_n$ are successive quotients in a 
suitable descending series of normal subgroups of $\pi_1(Y)$.

\begin{lemma}
\label{lem:setup}
Let $Y$ be a connected space with $H^1(Y;R)$ finitely generated. 
Let $\{Y_n\}_{n\ge 1}$ be a Postnikov $1$-tower for $Y$, as in 
diagram \eqref{eq:postnikov-space}. Then there is a colimit of 
Hirsch extensions $\mcm$ and quasi-isomorphisms 
$\psi_n \colon \mcm_n \to C^\ast(Y_n;R)$ 
such that $\mcm$ with structure maps
$\rho_n \colon \mcm_n \to C^\ast(Y;R)$
given by $\rho_n = h_n^\ast \circ \psi_n$
is a $1$-minimal model for $C^\ast(Y;R)$.
\end{lemma}
\begin{proof}
Note that $H^2(K(H^1(Y;R),1); R)$ is a finitely generated
free $R$-module. Then by induction, using the argument in the proof 
of property \eqref{prop:3} of Lemma \ref{lem:fgtf}, it follows
that the kernel of the map
$H^2(h_n) \colon H^2(Y_n;R) \to H^2(Y;R)$ is also a 
finitely generated, free $R$-module. The existence
of the colimit of Hirsch extensions $\mcm$ and 
quasi-isomorphisms 
$\psi_n \colon \mcm_n \to C^\ast(Y_n;R)$ such that
$\mcm$ with structure maps
$\rho_n = h_n^\ast \circ \psi_n$ is a 
$1$-minimal model for $C^\ast(Y;R)$ now follows 
from Theorem \ref{lem:group-gives-hirsch}.
\end{proof}

\subsection{Polynomial differential forms}
\label{subsec:apl}
We now briefly review a construction in rational 
homotopy theory due to Sullivan \cite{Sullivan}. 
For each integer $n \ge 0$, set
\begin{equation}
\label{eq:apl}
(A_{\PL})_n = 
			\frac{\Lambda(t_0, \ldots, t_n, y_0,\ldots,y_n)}
			{(\sum t_i-1, \sum y_j)},	
\end{equation}
where $\Lambda(t_0, \ldots, t_n, y_0,\ldots,y_n)$ 
denotes the free commutative algebra over $\Q$
generated by elements $t_i$ of degree zero and 
elements $y_j$ of degree one, and define a 
differential $d$ on $(A_{\PL})_n$ by setting 
$dt_i  = y_i$ and $dy_j=0$. 
Given a space $Y$,
an element $u \in A^p_{\PL}(Y)$ is a rule
that associates to each singular $n$-simplex $\sigma$
of $Y$ an element $u(\sigma)\in (A_{\PL})_n^p$
compatible with the face and degeneracy maps
$\partial_i\colon (A_{\PL})_{n+1} \to (A_{\PL})_n$ 
and $s_j\colon (A_{\PL})_n \to (A_{\PL})_{n+1}$ 
as given in \cite[p.~122]{FHT}.
Then $A_{\PL}(Y)$ is a commutative differential graded
algebra over the rationals and the assignment $Y\leadsto A_{\PL}(Y)$
is functorial.

In addition to the Sullivan algebra $A_{\PL}(Y)$ and the cochain algebra 
$C^\ast(Y;\Q)$, there is a differential graded algebra over the rationals,  
$CA(Y)$, with the following property.

\begin{theorem}[{\cite[Corollary~10.10]{FHT}}]
\label{thm:fht-ca}
For topological spaces $Y$ there are natural
quasi-isomorphisms
\[
\begin{tikzcd}
C^\ast(Y;\Q) \ar[r] & CA(Y) & A_{\PL}(Y). \ar[l]
\end{tikzcd} 
\]
Consequently, $A_{\PL}(Y)$ is weakly equivalent (as a dga) to  
$C^\ast(Y;\Q)$.
\end{theorem}

\subsection{Rational $1$-minimal models}
\label{subsec:1min-q}
Let $(\Lambda(\bX),d)$  be the free commutative algebra
over $\Q$ generated by the elements in a set $\bX$, 
equipped with a differential $d$. We say
that a dga $(\Lambda(\bX) \otimes_{\Q} \Lambda(\bY),\bar{d})$
is a {\em Hirsch extension of $(\Lambda(\bX),d)$}\/ if
$\bar{d}x = dx$ for all $x \in \bX$ and 
$\bar{d}y$ is a cocycle in $\Lambda(\bX)$ for all $y \in \bY$.

Now let $Y$ be a topological space. A sequence 
$\mcm_{n,\Q}(Y) = (\Lambda(\bX_1 \cup \cdots \cup \bX_n), d_n)$
of Hirsch extensions with $n \ge 1$ together with maps
\begin{equation}
\label{eq:apl-mm}
\begin{tikzcd}
& \mcm_{n+1,\Q}(Y)\ar[dl,"\rho_{n+1}" ']
\\
A_{\PL}(Y) & \mcm_{n,\Q}(Y) \ar[u, "i_n" ]  \ar[l,"\rho_n" ]  
\end{tikzcd}
\end{equation}
is a {\em rational $1$-minimal model}\/ for $Y$ 
if the following conditions are satisfied:
\begin{enumerate}[itemsep=2pt]
\item \label{rt1}
$d_1(x) = 0$ for all $x \in \bX_1$.
\item  \label{rt2}
The map 
$\rho_1^1 \colon H^1(\mcm_{1,\Q}) \to H^1(A_{\PL}(Y))$
is an isomorphism.
\item  \label{rt3}
Under the assignment
$x \mapsto d_{n+1}(x)$, the set $\bX_{n+1}$ corresponds to
a basis for $\ker(\rho_n^2) \subset H^2(\mcm_{n,\Q})$
given by the cohomology classes of the cocycles
$\{d_{n+1}(x) \mid x \in \bX_{n+1}\} \subset Z^2(\mcm_{n,\Q})$.
\end{enumerate}

As shown by Sullivan \cite{Sullivan} (see also 
\cite{Griffiths-Morgan, FHT, Suciu-Wang-forum, Suciu-2023}), 
every connected space $Y$ admits a rational $1$-minimal model, 
unique up to an isomorphism of cdgas. 

The following gives the definition of $n$-step equivalence 
in rational homotopy theory.

\begin{definition}
\label{def:n-step-Q}
Given commutative dgas $(A,d_A)$ and $(A',d_{A'})$ over $\Q$
with $1$-minimal models $(\mcm_\Q,\rho)$ and 
$(\mcm^\prime_\Q, \rho^\prime)$; respectively, and an
integer $n \ge 1$, we say that $A$ and $A^\prime$ are
\textit{$n$-step equivalent over $\Q$} 
if there are isomorphisms
$f_n\colon \mcm_{n,\Q} \to \mcm_{n,\Q}^\prime$ 
and $e_n\colon H^2(A) \to H^2(A^\prime)$ such that
the diagram below commutes.
\begin{equation}
\label{eq:ladder-n-step-Q}
\begin{tikzcd}[column sep=32pt]
H^2(\mcm_{n,\Q}) \ar[r, "H^2(f_n)"] 
\arrow["H^2(\rho_n)"  ']{d}
& H^2(\mcm_{n,\Q}^\prime)  \phantom{.} \arrow["H^2(\rho'_n)"]{d}
\\
H^2(A) \ar[r, "e_n"] 
& H^2(A^\prime) .
\end{tikzcd}
\end{equation}
\end{definition}

\subsection{Relating the integer and rational $1$-minimal models}
\label{subsec:comp-mod}
We are now in a position to state and prove the main result of this section.

\begin{theorem}
\label{thm:1-min-mod-zq}
Let $Y$ be a connected topological space with
$H^1(Y;\Z)$ finitely generated.
Then the $1$-minimal model for $C^\ast(Y;\Z)$ tensored 
with $\Q$ is weakly equivalent as a differential
graded algebra to the $1$-minimal model in rational
homotopy theory for $A_{\PL}(Y)$.
\end{theorem}

\begin{proof}
Let $\{Y_n\}_{n\ge 1}$ be a Postnikov $1$-tower for $Y$, and let
$\mcm=\{\mcm_{n,\Z}(Y), \rho_{n,\Z}\}_{n\ge 1}$ be an integral 
$1$-minimal model for $Y$, with quasi-isomorphisms 
$\psi_n\colon \mcm_{n,\Z}(Y) \to C^\ast(Y_n;\Z)$
as in Lemma \ref{lem:setup}.
The proof is to show by induction that for the rational
$1$-minimal model 
$\{\mcm_{n,\Q}(Y), \rho_{n,\Q}\}_{n\ge 1}$
for $A_{\PL}(Y)$, there are $1$-quasi-isomorphisms
$e_n \colon \mcm_{n,\Q}(Y) \to A_{\PL}(Y_n)$
with $\rho_{n,\Q} = f^*_n\circ e_n$. The result then
follows from the natural equivalences between
$A_{\PL}(Y_n)$ and $C^\ast(Y_n;\Q)$ and between
$A_{\PL}(Y)$ and $C^\ast(Y;\Q)$.

Assume that for the sets $\bX_i$ ($i \ge 1)$, we have
$\mcm_{n,\Z}(Y) = \T(\bX_1 \cup \cdots \cup \bX_n)$.
In order to prove the base case,
set $\mcm_{1,\Q} = \Lambda(\bX_1)$ and define
$e_1\colon \Lambda(\bX_1) \to A_{\PL}(Y_1)$ by setting
$e_1(x)$ equal to a cocyle in $A_{\PL}(Y_1)$ whose
cohomology class corresponds to $\psi_1(x)$ under the
weak equivalence between $C^\ast (Y_1;\Q)$ and
$A_{\PL}(Y_1)$. Then $e_1$ is a $1$-quasi-isomorphism
and $\rho_{1,\Q}= f_1^* \circ e_1$ induces an isomorphism
from $H^1(\mcm_{1,\Q}(Y))$ to $H^1(A_{\PL}(Y))$.

To prove the inductive step, assume $\mcm_{n,\Q}(Y)$ and 
a $1$-quasi-isomorphism $e_n$ have been constructed. 
Consider the diagram \eqref{eq:fht}. Note that the diagram
of solid arrows commutes and each horizontal arrow 
represents a $1$-quasi-isomorphism.
\begin{figure}
\small{
\begin{equation}
\label{eq:fht}
\begin{tikzcd}[column sep=22pt]
\mcm_{n+1,\Z}(Y) \otimes \Q \ar[r, "\psi_{n+1}"]
			&	C^\ast(Y_{n+1};\Q) \ar[r]
			& CA(Y_{n+1}) 
			& A_{\PL}(Y_{n+1}) \ar[l]
			& \mcm_{n+1,\Q}(Y)		\ar[l, dashed, "\: e_{n+1}" ']\\
\mcm_{n,\Z}(Y) \otimes \Q 
					\ar[r, "\psi_{n}" ] \ar[u] 
					\ar[dr, "\rho_{n,\Z}" '] 
			&	C^\ast(Y_{n};\Q) \ar[r] \ar[u] \ar[d]
			& CA(Y_{n}) \ar[u] \ar[d]
			& A_{\PL}(Y_{n}) \ar[l] \ar[u, "p_{n}"] \ar[d, "f_{n}"]
			& \mcm_{n,\Q}(Y)		\ar[l, "e_{n}" ']
						\ar[u, "i_{n}"] \ar[dl, "\rho_{n,\Q}"]   \\
			& C^\ast(Y;\Q) \ar[r]
			& CA(Y) 
			& A_{\PL}(Y) \ar[l]
\end{tikzcd}
\end{equation}
}
\end{figure}
Define $\mcm_{n+1,\Q}(Y)$ and $e_{n+1}$ as follows.
Set $\mcm_{n+1,\Q}(Y)$ equal to the Hirsch extension
$(\mcm_{n,\Q}(Y) \otimes \Lambda(\bX_{n+1}),d_{n,\Q})$, 
where for each $x \in \bX_{n+1}$, its differential
$d_{n,\Q}(x)$ is equal to a cocycle
in $\mcm_{n,\Q}(Y)$ whose cohomology class
corresponds to the cohomology class of 
$d_{n,\Z} (x) \in \mcm_{n,\Z}(Y) \otimes \Q$
under the $1$-equivalence between
$\mcm_{n,\Z}(Y) \otimes \Q$ and $\mcm_{n,\Q}(Y)$.

Given $x \in \bX_{n+1}$, set $e_{n+1}(x)$ equal to
a cochain in $A_{\PL}(Y_{n+1})$ whose coboundary
equals $e_n(d_{n,\Q}(x))$. 
Then $p_n \circ e_n =e_{n+1}\circ i_n$. 
As in the proof of Theorem \ref{lem:group-gives-hirsch},
the map $e_{n+1}$ gives a map of the spectral 
sequence of the Hirsch extension $i_n$ 
to the spectral sequence
of the fibration $p_n$ inducing an isomorphism of $E_2$ terms.
Thus, $e_{n+1}$ is a quasi-isomorphism 
and the proof is complete.
\end{proof}

\section{Distinguishing homotopy types via $1$-minimal models}
\label{sect:dcs-1min}
In this section we use the $1$-minimal model of a binomial 
$\cup_1$-dga over $R=\Z$ or $\F_p$ to define $n$-step 
equivalence for $n \ge 1$. We show in Theorem \ref{thm:iso-pi1} 
that if $X$ and $X^\prime$ have isomorphic fundamental
groups, then $C^\ast(X;\Z)$ and $C^\ast(X^\prime;\Z)$
are $n$-step equivalent for all $n\ge 1$, and 
in Proposition \ref{prop:borro} we give an example
of a family of spaces that can be distinguished
using $n$-step equivalence with $R=\Z$, where the
corresponding approach in rational homotopy theory
fails to distinguish the spaces.
\subsection{$n$-step equivalence}
\label{subsec:coker-h2}

Let $(A,d_A)$ be a binomial $\cup_1$-dga over the ring $R=\Z$ or 
$R=\F_p$. We will assume throughout  this section that $H^0(A) = R$ 
and $H^1(A)$ is a finitely generated, free $R$-module. 
By Theorem \ref{thm:cupd-minmodel}, there 
is a $1$-minimal model, $(\mcm,d)$, which comes equipped with a 
structure map, $\rho\colon \mcm\to A$, that induces an isomorphism 
on $H^1$ and a monomorphism on $H^2$. 
Furthermore, by Theorem \ref{thm:1-min-lift}, any morphism 
$\varphi\colon A\to A'$ between two such binomial $\cup_1$-dgas
lifts to a morphism $\widehat{\varphi}\colon \mcm \to \mcm'$ which 
is compatible with the respective colimit structures and with the  
structure maps (up to homotopy). The next result extracts further information 
from these data.

\begin{proposition}
\label{prop:nstep-test}
Let $A$ and $A'$ be two binomial $\cup_1$-dgas as above, with 
$1$-minimal models $\rho\colon \mcm\to A$ and $\rho'\colon \mcm'\to A'$. 
Let $\varphi\colon A\to A'$ be a morphism, and let $\widehat{\varphi}\colon 
\mcm \to \mcm'$ be a lift of $\varphi$. Then,
\begin{enumerate}
\item \label{ck1}
The map $H^2(\varphi)$ induces homomorphisms 
$\bar\varphi_n \colon \coker (H^2(\rho_n)) \to  \coker (H^2(\rho'_n))$, 
for all $n\ge 1$. 
\item  \label{ck2}
If $\varphi$ is a $1$-quasi-isomorphism, then all the homomorphisms 
$\bar\varphi_n$ are injective.
\end{enumerate}
\end{proposition}

\begin{proof}
Since the morphism $\widehat{\varphi}\colon \mcm \to \mcm'$  
is compatible with the colimit structures, it restricts to $R$-linear maps 
$\widehat{\varphi}_n \colon \mcm_n \to \mcm'_n$.  Since $\rho'\circ \widehat{\varphi}$ 
is homotopic to $\varphi\circ \rho$ via a homotopy which is also compatible 
with the colimit structures, we have that $H^2(\rho')\circ H^2(\widehat{\varphi}) = 
H^2(\varphi)\circ H^2(\rho)$ and similarly for the maps at level $n\ge 1$. 
We thus obtain for each $n\ge 1$ a commuting diagram in the category 
of $R$-modules,
\begin{equation}
\label{eq:ladder-def}
\begin{tikzcd}[column sep=32pt]
H^2(\mcm_n) \ar[r, "H^2(\widehat{\varphi}_n)"] \arrow["H^2(\rho_n)"  ']{d}
& H^2(\mcm_n^\prime)  \phantom{,} \arrow["H^2(\rho'_n)"]{d}
\\
H^2(A) \ar[r, "H^2(\varphi)"] \ar[d, two heads] 
& H^2(A^\prime) \phantom{,}\ar[d, two heads]
\\
\coker (H^2(\rho_n))  \ar[r, dashed, "\bar\varphi_n"] & \coker (H^2(\rho'_n)) , 
\end{tikzcd}
\end{equation}

where, by definition, $\bar\varphi_n$ is the $R$-linear 
map induced by $H^2(\varphi)$ on cokernels. 

Now suppose $\varphi\colon A\to A'$ is a $1$-quasi-isomorphism.  
Then the map $H^2(\varphi)\colon H^2(A)\to H^2(A')$ is injective; 
moreover, by Theorem \ref{thm:1-min-lift}, the map 
$H^2(\widehat{\varphi})$ is an isomorphism. A straightforward 
diagram chase with \eqref{eq:ladder-def} then shows that 
the map $\bar\varphi_n$ is injective, and the proof is complete.
\end{proof}

Next we define an equivalence relation such that
$\coker H^2(\rho_n)$ is an invariant for each $n\ge 1$.

\begin{definition}
\label{rem:n-step}
Given binomial cup-one dgas $A, A^\prime$ with
$1$-minimal models $(\mcm,\rho)$ and 
$(\mcm^\prime, \rho^\prime)$; respectively, and an
integer $n \ge 1$ we say that $A$ and $A^\prime$ are
\textit{$n$-step equivalent} if there are isomorphisms
$f_n\colon \mcm_n \to \mcm_n^\prime$ and
$e_n\colon H^2(A) \to H^2(A^\prime)$ such that
the diagram below commutes.
\begin{equation}
\label{eq:ladder-n-step}
\begin{tikzcd}[column sep=32pt]
H^2(\mcm_n) \ar[r, "H^2(f_n)"] \arrow["H^2(\rho_n)"  ']{d}
& H^2(\mcm_n^\prime)  \phantom{,} \arrow["H^2(\rho'_n)"]{d}
\\
H^2(A) \ar[r, "e_n"] 
& H^2(A^\prime) \phantom{,}
\end{tikzcd}
\end{equation}
\end{definition}

Note that if $A$ and $A^\prime$ are $q$-equivalent for some
$q \ge 2$, then $A$ and $A^\prime$ are $n$-step equivalent
for all $n\ge 1$.

For the rest of this paper, we will assume $R=\Z$ and that 
$H^2(A)$ and $H^2(A')$ are finitely generated. 
In this case, the cokernels of the maps 
$H^2(\rho_n)$ and $H^2(\rho'_n)$ 
are also finitely generated.  
Given a $1$-minimal model $\rho\colon \mcm\to A$
over $\Z$, we define for each $n\ge 1$ a finite 
abelian group $\kappa_n(A)$ by 
\begin{equation}
\label{eq:kappa-n}
\kappa_n(A)\coloneqq \Tors \Big(\!\coker \big(H^2(\rho_n) \colon 
H^2(\mcm_n) \to H^2(A)\big)\Big).
\end{equation}
That is, $\kappa_n(A)$ is the torsion subgroup of the finitely generated 
abelian group $\coker (H^2(\rho_n))$. 

From Proposition \ref{prop:nstep-test} it follows that
the groups $\coker H^2(\rho_n)$ and hence 
$\kappa_n(A)$ are independent of the choice of 
$1$-minimal model for $A$ used 
in constructing them. The considerations above
yield the following proposition.

\begin{proposition}
\label{cor:1iso-kappa}
Let $A$ and $A'$ be binomial $\cup_1$-dgas over $\Z$, 
with $H^0=\Z$, $H^1$ finitely generated and torsion-free, 
and $H^2$ finitely generated. 
Let $(\mcm,\rho)$ and $(\mcm^\prime,\rho^\prime)$
be $1$-minimal models over $\Z$ for $A$ and 
$A^\prime$, respectively. Then,
if $A$ and $A^\prime$ are $n$-step equivalent, then
$\coker H^2(\rho_n)$ is isomorphic to 
$\coker H^2(\rho_n^\prime)$ and 
$\kappa_n(A)$ is isomorphic to $\kappa_n(A^\prime)$ 
\end{proposition}
\begin{proof}
The result follows from the definition of $n$-step equivalence
by a diagram chase using the diagram 
\eqref{eq:ladder-n-step} expanded, as in diagram
\eqref{eq:ladder-def}, to include the cokernels.
\end{proof}

\subsection{Distinguishing homotopy $2$-types}
\label{subsec:2-type}

Making use of the above setup, we obtain in this section an invariant of 
$2$-type for spaces based on a cohomological comparison 
between their integral cochain algebra and the corresponding 
$1$-minimal model.

In this section, all our spaces are assumed to be connected 
$\Delta$-complexes $X$ with finitely generated cohomology 
groups $H^1(X;\Z)$ and $H^2(X;\Z)$ (for instance, presentation 
$2$-com\-plexes for finitely presented groups). Let $C^*(X;\Z)$ 
be the simplicial cochain algebra 
of such a space, viewed as a $\cup_1$-algebra over $\Z$. 
We consider the sequence of finite abelian groups 
$\kappa_n(X)\coloneqq \kappa_n(C^*(X;\Z))$, for $n\ge 1$,
The next result shows that these groups depend only on the 
homotopy $2$-type of $X$, or, equivalently, on the isomorphism 
type of its fundamental group.

\begin{theorem}
\label{thm:iso-pi1}
Let $X$ and $X'$ be two $\Delta$-complexes as above. 
\begin{enumerate}[itemsep=2pt]
\item \label{kn1}
If $\pi_1(X)\cong \pi_1(X')$, then $\kappa_n(X)\cong \kappa'_n(X)$ 
for all $n\ge 1$.
\item \label{kn2}
If $\kappa_n(X)\not\cong \kappa'_n(X)$ 
for some $n\ge 1$, then the cochain algebras $C^*(X;\Z)$ and 
$C^*(X';\Z)$  are not $n$-step equivalent.
\end{enumerate}
\end{theorem}

\begin{proof}
Let $\rho\colon \mcm\to  C^*(X;\Z)$ and 
$\rho'\colon \mcm'\to  C^*(X';\Z)$  be $1$-minimal models for 
the respective cochain algebras. 
We begin by proving property \eqref{kn1} in the case where
$X$ and $X^\prime$ are $2$-dimensional. 
By assumption, $\pi_1(X)\cong \pi_1(X')$, that is, 
$X$ and $X'$ have the same $2$-type.
As in the proof of Theorem \ref{thm:quasi-iso}, it follows that  
there is a homotopy equivalence, $f\colon \overline{X}\to \overline{X'}$,  
where $\overline{X} = X \vee \bigvee_{i\in I} S^2_i$ 
and $\overline{X'} =X'\vee \bigvee_{j\in J} S^2_j$, 
for some indexing sets $I$ and $J$. 
We let $q\colon \overline{X} \to X$ and 
$q'\colon \overline{X'} \to X'$ be the maps that collapse the 
wedges of $2$-spheres to the basepoint, and we 
let $\bar\rho\colon \overline{\mcm}\to  C^*(\overline{X};\Z)$ and 
$\bar\rho'\colon \overline{\mcm'}\to  C^*(\overline{X'};\Z)$ be 
$1$-minimal models for the respective cochain algebras. 

By Theorem \ref{thm:1-min-lift}, the induced maps on cochain 
algebras, $f^{\sharp}$, $q^{\sharp}$, and ${q'}^{\sharp}$, lift to 
maps between $1$-minimal models. For each $n\ge 1$, this leads 
to  the following diagram, which commutes up to homotopy.
\begin{equation}
\label{eq:wh-1min}
\begin{tikzcd}[row sep=22pt, column sep=22pt]
\mcm_n \ar[d, "\rho"] \ar[r, "\widehat{q^{\sharp}}", "\simeq" '] 
& \overline{\mcm}_n \ar[d, "\bar\rho_n "] 
& \overline{\mcm}'_n \ar[d, "\bar\rho'_n"] \ar[l, "\widehat{f^{\sharp}}" ', "\simeq"]  
&  \mcm'_n \ar[d, "\rho'_n "] \ar[l, "\widehat{{q'}^{\sharp}}" ', "\simeq"] 
\\
C^\ast(X;\Z) \ar[r, hook, "q^{\sharp}"] 
&C^\ast(\overline{X};\Z)
& C^\ast(\overline{X'};\Z)
\arrow{l}{\simeq}[swap, pos=0.4]{f^{\sharp}} 
&  C^\ast(X';\Z) . \ar[l, hook', pos=0.4, "{q'}^{\sharp}"']
\end{tikzcd}
\end{equation}
Since all the maps on the bottom row are $1$-quasi-isomorphisms, 
the maps on the top row are isomorphisms, again by Theorem \ref{thm:1-min-lift}.
Applying now the $H^2(-)$ functor to diagram \eqref{eq:wh-1min}, we obtain 
the following commuting diagram.
\begin{equation}
\label{eq:wh-h2}
\begin{tikzcd}[row sep=22pt, column sep=22pt]
H^2(\mcm_n) \ar[d, "H^2(\rho_n)"] \ar[r, "H^2(\widehat{q^{\sharp}})", "\simeq" '] 
& H^2(\overline{\mcm}_n) \ar[d, "H^2(\bar\rho_n)"] 
& H^2(\overline{\mcm}'_n) \ar[d, "H^2(\bar\rho'_n)"] \ar[l, "H^2(\widehat{f^{\sharp}})" ', "\simeq"]  
& H^2(\mcm'_n) \ar[d, "H^2(\rho'_n) "] \ar[l, "H^2(\widehat{{q'}^{\sharp}})" ', "\simeq"] 
\\
H^2(X;\Z) \ar[r, hook, "H^2(q)"]  \ar[d, two heads]
& H^2(\overline{X};\Z) \ar[d, two heads]
& H^2(\overline{X'};\Z) \ar[d, two heads]
\arrow{l}{\simeq}[swap, pos=0.4]{H^2(f)} 
&  H^2(X';\Z) \ar[l, hook', pos=0.4, "H^2(q')"'] \ar[d, two heads]
\\
\coker(H^2(\rho_n)) \ar[r, hook]
& \coker(H^2(\bar\rho_n)) 
& \coker(H^2(\bar\rho'_n)) \ar[l, pos=0.4, "\simeq" ']
& \coker(H^2(\rho'_n)) .\ar[l, hook']
\end{tikzcd}
\end{equation}

From the way the collapse map $q$ is defined, the induced homomorphism 
$H^2(q)$ may be identified with the natural inclusion $H^2(X,\Z)\to H^2(X,\Z)\oplus \Z^{I}$; 
thus the induced map on cokernels restricts to an isomorphism on torsion subgroups. 
Similar considerations apply to the map $H^2(q')$, while of course $H^2(f)$ is an 
isomorphism. It follows that $\Tors(\coker(H^2(\bar\rho_n)))\cong \Tors(\coker(H^2(\bar\rho'_n)))$.
This completes the proof of property \eqref{kn1} in the
case where $X$ and $X^\prime$ are $2$-dimensional.

Now let $X$ be an arbitrary $\Delta$-complex satisfying our assumptions. 
We let  $(C^*(X;\Z),\delta^*)$ be its simplicial cochain complex, 
$i\colon X^{(2)} \inj X$ the inclusion of the $2$-skeleton into $X$, 
and $\rho^{(2)}\colon \mcm^{(2)} \to C^*(X^{(2)};\Z)$ a  
$1$-minimal model for $X^{(2)}$. 
The induced homomorphism 
$H^2(i)\colon H^2(X;\Z)\to H^2(X^{(2)};\Z)$ may be 
identified with the natural inclusion
$H^2(X;\Z) \inj  H^2(X;\Z) \oplus\im(\delta^2)$. 
Since $\im(\delta^2)\subset C^3(X;\Z)$ is a free abelian group, 
the map induced by $H^2(i)$ on cokernels, 
$\coker(H^2(\rho_n)) \to  \coker(H^2(\rho^{(2)}_n))$, 
restricts to an isomorphism on torsion subgroups, and the proof 
of part \eqref{kn1} is complete.

Part \eqref{kn2} follows at once from Corollary \ref{cor:1iso-kappa}.
\end{proof}

\subsection{Integral versus rational $1$-minimal models}
\label{subsec:zq}

In this final section we give a family of links in the $3$-sphere
that can be distinguished from each other using the integral  
$1$-minimal model, but are not distinguished from
each other using the corresponding approach
in rational homotopy theory. We begin with a quick 
review of triple Massey products and the cohomology 
of link complements. 

We refer to \cite{May-1969} for the 
general definition and properties of Massey products. Here 
we restrict our attention to triple Massey products of elements 
in the first cohomology. Given a dga $(A,d)$ and elements 
$u_1, u_2, u_3$ in $H^1(A)$ with $u_1 u_2 = u_2 u_3=0$, let
$c_i$ be cocycles with cohomology classes
$[c_i] = u_i$, and let $c_{i,j}$ for $i<j$
be elements in $A^1$ with $dc_{i,j} = c_i c_j$.
Then $c_1 c_{2,3} + c_{1,2}c_3$ is a $2$-cocyle;
let $\la u_1, u_2, u_3 \ra$ denote the subset
of $H^2(A)$ consisting of the set of cohomology
classes $[c_1 c_{2,3} + c_{1,2}c_3]$ of this type.
It follows that the difference between any two elements in
$\la u_1, u_2, u_3 \ra$ is an element in
$u_1 \cup H^1(A) + H^1(A) \cup u_3$.

Now let $L = \bigcup_{i=1}^n L_i$ be an $n$-component link 
in the $3$-sphere, with complement $X_L = S^3\setminus L$. 
This space has the homotopy type of a finite, $2$-dimensional CW-complex;  
moreover, $H^1(X_L;\Z) = \Z^n$, with basis elements corresponding 
by Alexander duality to generators of $H_1(L_i;\Z)$, and with 
$H^2(X_L;\Z) = \Z^{n-1}$ generated by the Lefschetz duals 
$\gamma_{ij}$ to paths from $L_i$ to $L_j$.
A formula for the Massey products of elements in the
first cohomology of a CW-complex in terms
of the Magnus coefficients of the attaching maps
of the $2$-cells is given in \cite{Porter}, along 
with a proof that Massey products in the complement 
of a link and Milnor's $\bar{\mu}$-invariants of the link 
determine each other.
 
\begin{figure}
\[
\definecolor{linkcolor0}{rgb}{0.85, 0.15, 0.15}
\definecolor{linkcolor1}{rgb}{0.15, 0.15, 0.85}
\definecolor{linkcolor2}{rgb}{0.15, 0.85, 0.15}
\begin{tikzpicture}[line width=1.4, line cap=round, line join=round, scale=0.65]
  \begin{scope}[color=linkcolor0]
    \draw (2.58, 0.80) .. controls (2.57, 0.27) and (1.94, 0.07) .. 
          (1.32, 0.07) .. controls (0.07, 0.08) and (0.07, 2.39) .. 
          (0.07, 4.27) .. controls (0.07, 6.16) and (0.07, 8.46) .. 
          (1.34, 8.46) .. controls (1.92, 8.46) and (2.61, 8.45) .. (2.61, 8.04);
    \draw (2.61, 7.83) .. controls (2.61, 7.66) and (2.61, 7.48) .. (2.60, 7.31);
    \draw (2.60, 7.31) .. controls (2.60, 7.10) and (2.60, 6.89) .. (2.60, 6.68);
    \draw (2.60, 6.68) .. controls (2.60, 6.56) and (2.60, 6.45) .. (2.60, 6.33);
    \draw (2.60, 6.12) .. controls (2.60, 6.02) and (2.60, 5.92) .. (2.60, 5.82);
    \draw (2.60, 5.61) .. controls (2.60, 5.53) and (2.60, 5.44) .. (2.60, 5.36);
    \draw (2.60, 5.36) .. controls (2.59, 4.76) and (2.59, 4.15) .. (2.59, 3.55);
    \draw (2.59, 3.55) .. controls (2.59, 3.44) and (2.59, 3.33) .. (2.59, 3.23);
    \draw (2.59, 3.01) .. controls (2.59, 2.96) and (2.59, 2.92) .. (2.59, 2.87);
    \draw (2.58, 2.66) .. controls (2.58, 2.55) and (2.58, 2.45) .. (2.58, 2.35);
    \draw (2.58, 2.35) .. controls (2.58, 2.16) and (2.58, 1.97) .. (2.58, 1.78);
    \draw (2.58, 1.78) .. controls (2.58, 1.53) and (2.58, 1.27) .. (2.58, 1.02);
  \end{scope}
  \begin{scope}[color=linkcolor1]
    \draw (7.70, 0.68) .. controls (7.71, 0.25) and (8.28, 0.22) .. 
          (8.78, 0.22) .. controls (9.85, 0.23) and (9.86, 2.48) .. 
          (9.86, 4.29) .. controls (9.87, 6.10) and (9.88, 8.34) .. 
          (8.73, 8.36) .. controls (8.20, 8.36) and (7.59, 8.36) .. (7.60, 7.90);
    \draw (7.60, 7.90) .. controls (7.60, 7.71) and (7.60, 7.51) .. (7.61, 7.31);
    \draw (7.61, 7.31) .. controls (7.61, 7.13) and (7.61, 6.94) .. (7.61, 6.75);
    \draw (7.62, 6.54) .. controls (7.62, 6.44) and (7.62, 6.35) .. (7.62, 6.25);
    \draw (7.62, 6.04) .. controls (7.63, 5.91) and (7.63, 5.78) .. (7.63, 5.65);
    \draw (7.63, 5.65) .. controls (7.63, 5.52) and (7.63, 5.39) .. (7.64, 5.26);
    \draw (7.64, 5.26) .. controls (7.64, 5.16) and (7.64, 5.06) .. (7.64, 4.96);
    \draw (7.64, 4.74) .. controls (7.65, 4.25) and (7.66, 3.77) .. (7.66, 3.28);
    \draw (7.67, 3.06) .. controls (7.67, 2.96) and (7.67, 2.86) .. (7.67, 2.75);
    \draw (7.67, 2.75) .. controls (7.67, 2.59) and (7.68, 2.42) .. (7.68, 2.25);
    \draw (7.68, 2.25) .. controls (7.68, 2.13) and (7.68, 2.00) .. (7.69, 1.87);
    \draw (7.69, 1.65) .. controls (7.69, 1.40) and (7.70, 1.15) .. (7.70, 0.90);
  \end{scope}
  \begin{scope}[color=linkcolor2]
    \draw (7.70, 0.79) .. controls (8.31, 0.78) and (8.29, 2.81) .. 
          (8.27, 4.34) .. controls (8.26, 5.85) and (8.24, 7.90) .. (7.70, 7.90);
    \draw(7.49, 7.91) .. controls (5.86, 7.91) and (4.23, 7.92) .. (2.61, 7.93);
    \draw (2.61, 7.93) .. controls (2.28, 7.94) and (1.90, 7.91) .. 
          (1.90, 7.63) .. controls (1.89, 7.38) and (2.21, 7.31) .. (2.50, 7.31);
    \draw (2.71, 7.31) .. controls (4.31, 7.31) and (5.90, 7.31) .. (7.50, 7.31);
    \draw (7.71, 7.31) .. controls (7.84, 7.31) and (7.84, 7.13) .. 
          (7.84, 6.98) .. controls (7.83, 6.81) and (7.76, 6.64) .. (7.62, 6.65);
    \draw (7.62, 6.65) .. controls (5.98, 6.66) and (4.35, 6.67) .. (2.71, 6.68);
    \draw (2.49, 6.68) .. controls (2.23, 6.68) and (1.93, 6.68) .. 
          (1.93, 6.46) .. controls (1.93, 6.24) and (2.29, 6.23) .. (2.60, 6.23);
    \draw (2.60, 6.23) .. controls (4.27, 6.20) and (5.95, 6.17) .. (7.62, 6.15);
    \draw (7.62, 6.15) .. controls (7.75, 6.14) and (7.85, 6.03) .. 
          (7.84, 5.90) .. controls (7.83, 5.79) and (7.82, 5.65) .. (7.74, 5.65);
    \draw (7.52, 5.65) .. controls (5.88, 5.67) and (4.24, 5.69) .. (2.60, 5.71);
    \draw (2.60, 5.71) .. controls (2.32, 5.72) and (1.98, 5.72) .. 
          (1.98, 5.55) .. controls (1.98, 5.38) and (2.25, 5.37) .. (2.49, 5.36);
    \draw (2.70, 5.36) .. controls (4.31, 5.33) and (5.92, 5.30) .. (7.53, 5.26);
    \draw (7.74, 5.26) .. controls (7.82, 5.26) and (7.87, 5.17) .. 
          (7.87, 5.08) .. controls (7.87, 4.96) and (7.76, 4.88) .. (7.64, 4.85);
    \draw[dashed] (7.64, 4.85) .. controls (5.99, 4.43) and (4.34, 4.00) .. (2.69, 3.58);
    \draw(2.48, 3.52) .. controls (2.29, 3.47) and (2.05, 3.41) .. 
          (2.05, 3.26) .. controls (2.06, 3.11) and (2.35, 3.12) .. (2.59, 3.12);
    \draw (2.59, 3.12) .. controls (4.28, 3.14) and (5.97, 3.15) .. (7.67, 3.17);
    \draw(7.67, 3.17) .. controls (7.79, 3.17) and (7.89, 3.08) .. 
          (7.88, 2.96) .. controls (7.87, 2.87) and (7.87, 2.75) .. (7.78, 2.75);
    \draw (7.56, 2.75) .. controls (5.90, 2.76) and (4.25, 2.76) .. (2.59, 2.76);
    \draw (2.59, 2.76) .. controls (2.33, 2.76) and (2.04, 2.77) .. 
          (2.02, 2.56) .. controls (2.01, 2.38) and (2.26, 2.35) .. (2.48, 2.35);
    \draw (2.69, 2.35) .. controls (4.32, 2.32) and (5.95, 2.29) .. (7.57, 2.26);
    \draw (7.79, 2.25) .. controls (7.89, 2.25) and (7.90, 2.12) .. 
          (7.90, 2.01) .. controls (7.89, 1.88) and (7.81, 1.76) .. (7.69, 1.76);
    \draw (7.69, 1.76) .. controls (6.02, 1.77) and (4.35, 1.77) .. (2.69, 1.78);
    \draw (2.47, 1.78) .. controls (2.26, 1.78) and (2.11, 1.58) .. 
          (2.10, 1.35) .. controls (2.08, 1.10) and (2.31, 0.92) .. (2.58, 0.91);
    \draw (2.58, 0.91) .. controls (4.29, 0.87) and (5.99, 0.83) .. (7.70, 0.79);
  \end{scope}
  \node at (1.7,7.6) {\footnotesize$1$};
  \node at (1.7,5.6) {\footnotesize$2$};
  \node at (1.9,4.8) {\footnotesize$\vdots$};
  \node at (1.9,4.2) {\footnotesize$\vdots$};
  \node at (1.7,2.6) {\footnotesize$n$};
  \draw[color=linkcolor0, ->] (0.08,4.6)--(0.08,4.7);
  \node at (0.65, 4.6) {\footnotesize$L_1$};
  \draw[color=linkcolor1, ->] (9.87,4.6)--(9.87,4.7);
  \node at (10.44, 4.6) {\footnotesize$L_2$};
  \draw[color=linkcolor2, ->] (8.27,4.6)--(8.27,4.7);
  \node at (8.84, 4.6) {\footnotesize$L_3$};
\end{tikzpicture}
\]
\caption{Generalized Borromean Rings}
\label{borromean}
\end{figure}
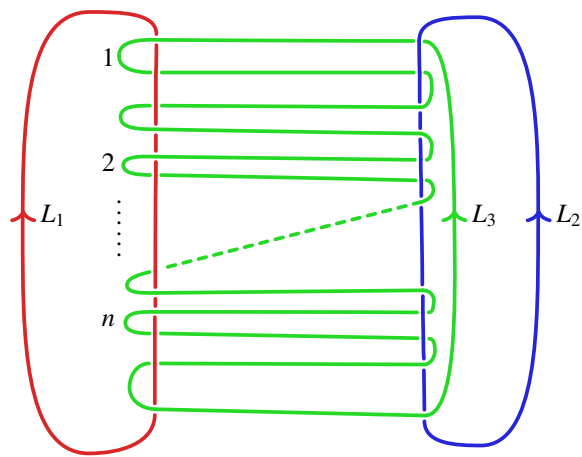

Consider now the following family of links in the $3$-sphere. 
For each $n\ge 1$, let $L(n)$ be the link pictured in Figure \ref{borromean}, 
where the pattern in the middle repeats $n$ times. 
Note that $L(1)$ is the well-known Borromean rings, and that the three 
components of $L(n)$ have pairwise linking numbers equal to $0$, 
for all $n\ge 1$. 
We denote by $X(n)=X_{L(n)}$ the complement of $L(n)$ in $S^3$, 
and we let $\rho(n) \colon \mcm(n)\to A(n)$ be a $1$-minimal model 
for the cochain algebra $A(n)\coloneqq C^*(X(n); \Z)$. 

\begin{proposition}
\label{prop:borro}
For the links $L(n)$ with complement $X(n)$ described above, we have:
\begin{enumerate}[itemsep=2pt]
\item \label{borro1}
The cokernel of $H^2(\rho_2(n))$ is isomorphic to $\Z_n \oplus \Z_n$, 
and so $\kappa_2(X(n))=\Z_n \oplus \Z_n$.
\item \label{borro2}
The cochain algebras $C^*(X(m); \Z)$ and $C^*(X(n); \Z)$
are not $2$-step equivalent for $m\ne n$.
\item \label{borro3}
The Sullivan algebras $A_{\PL}(X(m))$ and 
$A_{\PL}(X(n))$ are $2$-step equivalent for all 
positive integers $m$ and $n$.
\end{enumerate}
\end{proposition}

\begin{proof}
\eqref{borro1}
To compute the cokernel of $H^2(\rho_2(n))$, we proceed as follows.
First note that $\mcm_1(n) = (\T(\{x_1, x_2, x_3\}, d_{\bz})$, with
$\rho_1(n)(x_i)$ equal to a cocycle whose cohomology class
$u_i$ is Alexander dual to the generator of $H_1(L_i;\Z)$ 
given by the orientation of $L_i$. 

Since all linking numbers for $L(n)$ are zero, we have that all
cup products of elements in $H^1(X(n);\Z)$ 
are zero, and it follows that $\mcm_2(n) 
= (\T(\{x_1, x_2, x_3, x_{1,2}, x_{1,3}, x_{2,3}\}, d_2(n))$
where $d_2(n)(x_{i,j}) = x_i \otimes x_j$ 
and $\rho_2(n)$ is any extension of $\rho_1(n)$. 
In the spectral sequence of the Hirsch extension
$\mcm_1(n)\inj \mcm_2(n)$,
it follows from Theorem \ref{thm:coho-iso-R} that
the $E_2$ term is the tensor product of two exterior algebras,
\begin{equation}
\label{eq:e2-borro}
E_2=\bwedge ([x_1], [x_2], [x_3])\otimes_\Z 
  \bwedge ([x_{1,2}], [x_{2,3}], [ x_{1,3}]),
\end{equation}
where $[x]$ denotes the image of $x$ in $E_2$.
The differential $d_2\colon E_2\to E_2$ is given by 
$d_2[x_i]=0$ and $d_2[x_{i,j}] = [x_i x_j]$.
Then by direct computation of the $E_\infty^{p,q}$
terms with $p+q=2$, it follows that
$H^2(\mcm_2(n)) = \Z^8$, with basis given by 
the triple Massey products 
\begin{equation}
\label{eq:massey-borro}
\begin{gathered}
\la [x_{1}], [x_{1}], [x_{2}]\ra,
\la [x_{1}], [x_{2}], [x_{2}]\ra,
\la [x_{1}], [x_{1}], [x_{3}]\ra, 
\la [x_{1}], [x_{3}], [x_{3}]\ra,\\
\hspace*{3pc} \la [x_{2}], [x_{2}], [x_{3}]\ra,
\la [x_{2}], [x_{3}], [x_{3}]\ra,
\la [x_{1}], [x_{2}], [x_{3}]\ra, 
\la [x_{1}], [x_{3}], [x_{2}]\ra.
\end{gathered}
\end{equation}
The correspondence between Massey products and 
elements in $E_\infty$ is indicated as follows.
An element in $E_\infty^{1,1}$ such as 
$[x_1] \otimes [x_{1,2}]$ is represented by the 
element in the set $\la [x_{1}], [x_{1}], [x_{2}]\ra$ given
by the cohomology class of the cocycle
$x_1 \otimes x_{1,2} - \z_2(x_1)\otimes x_2$ in
$\mcm_2(n)$ and  $\la [x_{1}], [x_{2}], [x_{3}]\ra$ is
taken to be the cohomology class of the cocycle
$x_1 \otimes x_{2,3} + x_{1,2} \otimes x_3$, which
corresponds to the element 
$[x_1] \otimes [x_{2,3}] - [x_3] \otimes [x_{1,2}]$ in
$E_\infty^{1,1}$.

We now determine the cokernel of $H^2(\rho_2(n))$.
By the naturality of Massey products, 
$H^2(\rho_2(n))$ sends a Massey product
$\la [x_i], [x_j], [x_k] \ra$ to $\la u_i, u_j, u_k \ra$.
All cup products of elements in 
$H^1(X(n);\Z)$ are zero, so each triple
Massey product contains only one element.  
Since each $2$-component sublink of $L$ is isotopic 
to the unlink, it follows that each of the first six Massey
products from \eqref{eq:massey-borro}  maps to zero. 
From the computations in \cite[Example 3, p.~46]{Porter},
we have that
$\la u_1, u_2, u_3 \ra = - n\gamma_{1,3}$ and
$\la u_1, u_3, u_2 \ra = n \gamma_{1,2}$.
Since $\{ \gamma_{1,3}, \gamma_{1,2}\}$ is a basis
for $H^2(X(n);\Z)$, the argument that 
$\coker H^2(\rho_2(n)) = \Z_n \oplus \Z_n$ is complete.

Part \eqref{borro2} follows at once from Part \eqref{borro1} 
and Theorem \ref{thm:iso-pi1}.

\eqref{borro3}
We now show that if the binomial cup-one dga 
$C^\ast(X(n);\Z)$ is replaced by the cdga $B(n) = A_{\PL}(X(n))$ 
and the integral $1$-minimal model $\mcm$ is replaced with the 
Sullivan rational $1$-minimal model $\mcm_{\Q}\simeq \mcm\otimes \Q$ 
from \cite{Sullivan}, then $B(m)$ and $B(n)$ are $2$-step
equivalent for all positive integers $m$ and $n$.

Let $\rho_{\Q}(n)\colon (\mcm_{\Q},d) \to  B(n)$ be the rational 
$1$-minimal model for $B(n)$.  Then $\mcm_{2,\Q}$ is the 
exterior algebra over $\Q$ given by 
$\bigwedge (y_1, y_2, y_3, y_{1,2}, y_{1,3},y_{2,3})$, 
with differential given by $dy_i=0$ and $dy_{i,j} = y_i \wedge y_j$.
Let $v_i \in H^1(B(n))$ be the elements that correspond to
the elements $u_i \in H^1(X(n);\Q)$,
and let $\omega_{i,j} \in H^2(B(n))$ be the elements
that correspond to the elements 
$\gamma_{i,j} \in H^2(X(n);\Q)$
via the zig-zag of quasi-isomorphisms
\begin{equation}
\label{Massey-products-equal}
\begin{tikzcd}[column sep=20pt]
C^\ast(X(n);\Q) \ar[r] & CA(X(n))
& A_{\PL}(X(n)) = B(n) \, . \ar[l]
\end{tikzcd}
\end{equation}
We can assume that $\rho_{1,\Q}(y_i) = v_i$ for $i \in\{ 1,2, 3\}$.
Since the maps in equation \eqref{Massey-products-equal}
are dga maps inducing isomorphisms on cohomology,
it follows from the computation of Massey products
in $C^\ast(X(n);\Q)$ that in $H^2(B(n))$
we have 
$\la v_1, v_2, v_3 \ra = -n\omega_{1,3}$ and 
$\la v_1, v_3, v_2 \ra = n \omega_{1,2}$, while 
all triple products of the form $\la v_i, v_j , v_k\ra$
with $\{ i,j,k\}$ a proper subset of $\{1,2,3\}$
are zero. It follows that 
$H^2(\rho_2) \la [y_1], [y_2], [y_3]\ra = -n\omega_{1,3}$ and 
$H^2(\rho_2) \la [y_1], [y_3], [y_2]\ra = n\omega_{1,2}$,
while 
$H^2(\rho_2) \la [y_i], [y_j], [y_k]\ra = 0$ if
 $\{ i,j,k\}\subsetneqq \{1,2,3\}$.
 
Now given positive integers $n$ and $m$, we define a homomorphism 
$e_{n,m} \colon H^{2}(B(n)) \to H^{2}(B(m))$ by
$\omega_{i,j} \mapsto  \frac{m}{n} \cdot \omega_{i,j}$
for $i = 1$ and $j\in \{2,3\}$.
Then $e_{n,m}$ is an isomorphism and the 
following diagram commutes, 
\begin{equation}
\label{eq:triangle-diagram}
\begin{tikzcd}[row sep=2.8 pc, column sep=0.7 pc]
       & H^{2}(\mcm_{2,\Q})
       \ar[dl, "H^{2}(\rho_{2,\Q}(n))"  '] 
       \ar[dr, "H^{2}(\rho_{2,\Q}(m))"  ] \\
H^{2}(B(n)) 
       \ar[rr, "e_{n,m}"  '] 
       &&
       H^{2}(B(m)) \, .
\end{tikzcd}
\end{equation}
The argument that over the rationals the link complements 
$X(m)$ and $X(n)$ are $2$-step equivalent for all $m, n\ge 1$ is complete. 
(We do not know whether the two link complements have the same 
rational homotopy type.)
\end{proof}

\section{Torsion-free nilpotent groups and integral $1$-minimal models}
\label{sect:nilpotent}

\subsection{Nilmanifolds, torsion-free nilpotent groups, and Hall polynomials}
\label{subsec:Hall poly}

We start by reviewing some well-known facts about nilmanifolds and 
finitely generated, torsion-free nilpotent groups, for short, $\TT$-groups. 
We refer to \cite{Cenkl-Porter-2000, Hall-1976, Lambe-Priddy-1982, 
Malcev-1951, Suciu-Wang-forum} for more details.

Given a $\TT$-group $G$, there is a refinement of the upper central series,  
$G=G_1 >  \cdots > G_n > G_{n+1}=1$, 
with each subgroup $G_i$ a normal subgroup of $G_{i+1}$, and 
each quotient $G_i/G_{i+1}$ an infinite cyclic group.
The integer $n$ is an invariant of the group, called the {\em length of $G$}.
From the well-known correspondence
\cite{Brown} between central extensions of a group $G$ by an
abelian group $B$ and elements in $H^2(G;B)$,
it follows that every $\TT$-group $G$ 
can be viewed as a multiplication defined on
the set $\Z^{n}$, where $n$ is the length of $G$, 
with the multiplication given by the cocycles representing 
the iterated central extensions.

More precisely, we can choose a \textit{Malcev basis}\/ $\{u_1,\dots, u_n\}$ 
for $G$, which satisfies $G_i=\langle G_{i+1}, u_i\rangle$. Consequently, 
each element $u\in G$ can be  written  uniquely as 
$u_1^{a_1}u_2^{a_2}\cdots u_n^{a_n}$ for some $a_i\in \Z$.
Using this basis, the group $G$, as a set, can be identified with $\Z^n$ 
via the map sending $u_1^{a_1}\cdots u_n^{a_n}$ to $a=(a_1,\dots, a_n)$. 
The multiplication in $G$ then takes the form
\begin{equation}
\label{eq:malcev-basis}
u_1^{a_1}\cdots u_n^{a_n} \cdot u_1^{b_1}\cdots u_n^{b_n}=
u_1^{\rho_1(a,b)}\cdots u_n^{\rho_n(a,b)},
\end{equation}
where each map $\rho_i\colon \Z^n\times \Z^n \to \Z$ 
is given by an integrally-valued rational polynomial. 
This procedure identifies 
the group $G$ with  the group $(\Z^n,\rho)$, with multiplication the map 
$\rho=(\rho_1,\dots,\rho_n)\colon \Z^n\times \Z^n \to \Z^n$ 
given by 
\begin{equation}
\label{eq:hall-poly}
\rho_i(x,y)= x_i + y_i 
+ \tau_i(x_1, \ldots x_{i-1},y_1, \ldots , y_{i-1}),
\end{equation}
where each $\tau_i$ is a linear combination of integrally valued rational 
polynomials of the form 
$\z_I(x_1, \ldots , x_{i-1})\z_J(y_1, \ldots , y_{i-1})$
with both $I$ and $J$ nonzero. 
The polynomials $\tau_i$ are called the 
\emph{Hall polynomials}\/ of the group $G$, see \cite{Hall-1976}.

The group $G=(\Z^n,\rho)$ is a discrete subgroup of the real Lie group $G\otimes \R$. 
The quotient space, $M=(G\otimes \R)/G$, is a compact, connected, orientable 
manifold, called a {\em nilmanifold}.  By construction, the fundamental group
$\pi_1(M)$ is isomorphic to $G$.  Moreover, the nilmanifold $M$ is an 
Eilenberg--MacLane space of type $K(G,1)$.  

Let $R$ be the subring of $\Q$ generated by the coefficients of $\rho$. 
As shown by Lambe and Priddy in \cite{Lambe-Priddy-1982}, one 
can associate with the group $G$ a Lie algebra over $R$, 
by setting $L(G)=(R^n,[\:,\:])$, with Lie bracket given by 
$[x,y]= \sigma(x,y)-\sigma(y,x)$, where $\sigma=(\sigma_1,\dots,\sigma_n)$ 
is the quadratic part of $\rho$. The real Lie algebra 
$\mathcal{L}(G)=L(G)\otimes_{R} \R$, then, coincides 
with the classical Lie algebra of $G\otimes \R$.
Moreover, as shown by Nomizu in \cite{Nomizu}, 
the Chevalley--Eilenberg complex $\bwedge^{*} \mathcal{L}(G)^{\vee}$ 
is a (commutative) minimal model over $\R$ for the nilmanifold $M$, and thus 
$H^{*}(M;\R)\cong H^*(\mathcal{L}(G);\R)$. 

In \cite{Lambe-Priddy-1982}, Lambe and Priddy strengthened Nomizu's result, 
by showing that $H^*(G;S)\cong H^*(L(G);S)$, where $S$ is a subring 
of $\Q$ obtained from $\Z$ by adjoining fractions of the 
form $1/n$, for $n$ large enough (depending on $G$). 
One can ask whether this isomorphism holds over the 
(possibly smaller) ring $R$. In \cite{Kuzmin-Semenov}, Kuzmin 
and Semenov showed that $H^*(G;\Z)\cong H^*(L(G);\Z)$ 
when $G$ is a free nilpotent group of class $2$. 
It still appears to be an open question when this 
is the case in general. 

\subsection{An equivalence between nilmanifolds and 
finitely-generated $1$-minimal models}
\label{subsec:nil-min}

In this section we show that there is a one-to-one correspondence 
between nilmanifolds (or, equivalently, $\TT$-groups) and finite 
colimits of Hirsch extensions that leads to an
equivalence between the category of isomorphism
classes of $\TT$-groups and isomorphism classes of 
finitely-generated $1$-minimal models.
We start with a correspondence between Hall polynomials
and differentials of generators in a colimit of Hirsch
extensions over the ring $R=\Z$.

Given a finite, ordered set $\bX = \{x_1, \ldots , x_n\}$, write 
$\bX^i = \{ x_1, \ldots , x_i\}$ and 
$\bY^i = \{y_1, \ldots , y_i\}$.
Let $\T(\bX^{i-1},\bY^{i-1})$ denote the submodule of 
$\T^1(\bX \cup \bY)$ consisting of linear combinations
of polynomials of the form
$\z_I(x_1, \ldots , x_{i-1})\z_J(y_1, \ldots , y_{i-1})$
with both $I$ and $J$ nonzero.
There is then a bijection between maps
$\rho\colon \Z^n \times \Z^n \to \Z^n$
satisfying \eqref{eq:hall-poly} where each $\tau_i$ is 
a linear combination of polynomials 
$\z_I(x_1, \ldots , x_{i-1})\z_J(y_1, \ldots , y_{i-1})$
with both $I$ and $J$ nonzero and 
elements $\tau_i$ with $\tau_1 = \tau_2 = 0$
and $\tau_i \in \T(\bX^{i-1},\bY^{i-1})$ 
for $3 \le i \le n$.

From the basis theorem for $\T^1(X)$ and the fact 
that $\T^2(\bX) = \T^1(\bX) \otimes \T^1(\bX)$, 
it follows that the map 
$e\colon \T(\bX^{i-1},\bY^{i-1}) \to  
\T^1(\bX^{i-1}) \otimes \T^1(\bX^{i-1})$ 
given by 
\begin{equation}
\label{eq:e-map}
e\left( \z_I(x_1, \ldots , x_n) \z_J(y_1, \ldots ,y_n)
\right) 
= - \z_I(x_1, \ldots , x_n)\otimes \z_J(x_1, \ldots ,x_n) 
\end{equation}
is a bijection. Given $\tau=(\tau_1,\dots, \tau_n)$ with 
$\tau_i \in \T(\bX^{i-1}, \bY^{i-1})$,
define a map $\rho_\tau \colon \Z^n \times \Z^n \to \Z^n$ by
\begin{equation}
\label{eq:rho-tau}
(\rho_\tau)_i(x,y) 
= x_i + y_i 
		+\tau_i(x_1, \ldots , x_{i-1}, y_1, \ldots y_{i-1}) 
\end{equation}
and set $d_\tau \colon \T(\bX) \to \T(\bX)$ equal to the 
unique degree-one map that satisfies the $\cupd$ formula 
and the graded Leibniz rule with $d_\tau(x_i) = - e\circ \tau_i$.

\begin{theorem}
\label{thm:equivalence}
Given $\tau=(\tau_1,\dots, \tau_n)$ as above, 
$G=(\Z^n,\rho_\tau)$ is a $\TT$-group if 
and only if $\T=(\T(\bX), d_{\tau})$ is a colimit of Hirsch 
extensions, in which case, $H^*(G;\Z)\cong H^*(\T)$. 

Moreover, this bijection between $\TT$-groups 
and finite colimits of Hirsch extensions 
induces an equivalence of 
categories between isomorphism 
classes of $\TT$-groups and aug\-mentation-preserving 
isomorphism classes of finitely-generated 
$1$-minimal models.
\end{theorem}

\begin{proof}
The first statement follows from the discussion above and
Lemma \ref{lem:h-group}.

The second statement follows from Corollaries \ref{cor:equiv-cat} 
and \ref{cor:1-mm-unique}, using the fact that the classifying space 
$BG$ is connected, and thus $C^*(BG;\Z)$ has a canonical augmentation.
\end{proof}

\subsection{On the integral cohomology of nilmanifolds}
\label{subsec:nil-coho}

We conclude this section with two examples (extracted from  \cite{Lambe-Priddy-1982}) 
which illustrate the computation of the integral $1$-minimal model of a $\TT$-group, 
and its use in the determination of the cohomology ring of the corresponding nilmanifold. 

\begin{example}
\label{ex:heisenberg} 
For a fixed integer $k>0$, let $G(k)$ be the group of 
upper triangular matrices
\begin{equation}
\label{eq:heis-group}
\left(\begin{smallmatrix}
1 & a_1 & a_{1,2}/k\\
0 & 1    & a_2\\
0 & 0    & 1 
\end{smallmatrix}\right), 
\end{equation}
with $a_1,a_2, a_{1,2}\in \Z$. The multiplication in 
$G(k)$ has the form
$(a_1, a_2, a_{1,2})\cdot 
(a_1^\prime, a_2^\prime, a_{1,2}^\prime)
=
(a_1 + a_1^\prime, a_2 + a_2^\prime, 
a_{1,2} + a_{1,2}^\prime + k a_1 a_2)$.
Then $G(k)$ is a $\TT$-group and hence
from Theorem \ref{thm:equivalence} it follows that
$(\T(\{x_1, x_2, x_{12}\}), d)$ with 
$d(x_1) = d(x_2) = 0$, and 
$d(x_{12}) = - k x_1 \ot x_2$
is a colimit of Hirsch extensions. Moreover, 
\[
H^*(G(k);\Z)\cong H^\ast(\T\{x_1, x_2, x_{12}\})=
\bwedge(x_1, x_2, x_{12}, x_1 x_{12}, x_2 x_{12})/(k x_1 x_2).
\]

Now take $\bX_1 = \{x_1,x_2\}$ and $\bX_2=\{x_{12}\}$
and let $i_1$ be the inclusion of $\T(\bX_1)$ in $\T(\bX)$.
Since $\ker (H^2(i_1))$ is the free $\Z$-module generated
by the cohomology class of the cocycle
$ -k x_1 \ot x_2$, it follows that 
$(\T(\bX), d)$ is a $1$-minimal model for $G(k)$.  
Furthermore, the Lie algebra $L(G)$ is defined over $\Z$, 
and, as noted in \cite{Lambe-Priddy-1982}, $H^*(G(k);\Z)\cong H^*(L(G(k));\Z)$. 
\end{example}

\begin{example}
\label{ex:gkell} 
Let $k$ and $\ell$ be positive integers, and let
$G(k,\ell)$ be the group of upper triangular matrices
\vspace*{-2pt}
\begin{equation}
\label{eq:LP-group}
\left(\begin{smallmatrix}
1 & a_1 & a_{1,2}/k & a_{13}/k\ell\\
0 & 1    & a_2          & a_{2,3}/\ell\\
0 & 0    & 1             & a_3\\
0 & 0   & 0             & 1
\end{smallmatrix}\right), 
\end{equation}
with $a_1,a_2, a_3, a_{1,2}, a_{2,3}, a_{13}\in \Z$. 
The group $G(k,\ell)$ is a $\TT$-group, with multiplication 
map $\rho$ given by $\rho_i(a,a^\prime) = a_i + a_i^\prime$ 
for $i=1,2,3$; $\rho_{12}(a,a^\prime) 
= a_{12} + a^\prime_{12} + k a_1 a^\prime_2$; 
$\rho_{23}(a, a^\prime) = a_{23} + a^\prime_{23} + \ell a_2 a_3^\prime$; 
and $\rho_{13}(a,a^\prime) = a_{13} + a^\prime_{13} 
+ k a_1a^\prime_{23} + \ell a_{12} a^\prime_3$. 
Applying Theorem \ref{thm:equivalence}, we find that 
$\T=(\T(\{x_1, x_2, x_3, x_{12}, x_{23}, x_{13}\}), d)$ with 
$d(x_1) = d(x_2) = d(x_3) = 0$, 
$d(x_{12}) = - k x_1 \ot x_2$,
$d(x_{23}) = -\ell x_2 \ot x_3$, and
$d(x_{13}) = -k x_1 \ot x_{23} - \ell x_{12} \ot x_3$
is a colimit of Hirsch extensions, whose cohomology 
is isomorphic to that of $G(k,\ell)$.

Take $\bX_1 = \{x_1,x_2, x_3\}$,
$\bX_2=\{x_{12}, x_{23}\}$, and
$\bX_3 =\{x_{13}\}$.
Let $i_1$ be the inclusion of $\T(\bX_1)$ in $\T(\bX)$,
and let $i_2$ be the inclusion of 
$\T(\bX_1 \cup \bX_2)$ in $\T(\bX)$.
Since $\ker (H^2(i_1))$ is the free $\Z$-module generated
by the cohomology classes of the cocycles
$-k x_1 \ot x_2$ and $-\ell x_2 \ot x_3$ and 
$\ker (H^2(i_2))$ is the free $\Z$-module generated
by the cohomology class of the cocycle
$-k x_1 \ot x_{23} - \ell x_{12} \ot x_3$, it follows 
that $(\T(\bX), d)$ is a $1$-minimal model for $G(k,\ell)$.
It is shown in \cite{Lambe-Priddy-1982} that $H^*(G(k,\ell);S)\cong 
H^*(L(G(k,\ell));S)$, provided the ring $S$ contains the set $\{1/n : 
2 \le n \le 10\}$. Using our integral $1$-minimal model, we can sharpen
their result, and show that the two cohomology groups
are in fact isomorphic over $\Z$. 
\end{example}

\begin{ack}
We wish to thank the referee for a careful reading of the manuscript and for 
many valuable comments and suggestions that helped us to improve both 
the substance and the exposition of the paper.
\end{ack}


\end{document}